\documentclass{amsart}

\usepackage{url}
\DeclareUnicodeCharacter{221E}{infinity}
\usepackage[T1]{fontenc}

\usepackage{amsmath,amssymb,amstext,amsthm}
\usepackage{mathrsfs}
\usepackage{mathtools}
\usepackage{todonotes}
\usepackage{xcolor}
\usepackage{esint}
\usepackage[shortlabels]{enumitem}
\usepackage{tikz-cd}
\usetikzlibrary{calc}
\usepackage[all,cmtip]{xy}
\usepackage{hyperref}
\usepackage{graphicx}
\usepackage{wrapfig}
\usepackage{stmaryrd} 
\usepackage{microtype}

\usepackage{tensor}
\usepackage{faktor}

\usepackage[margin=1.5in]{geometry}
\usepackage[noabbrev,nameinlink,capitalise]{cleveref}		

\theoremstyle{definition} 
\numberwithin{equation}{section}

\newtheorem{introtheorem}{Theorem}[section]

\newtheorem{introcor}[introtheorem]{Corollary}

\newtheorem{theorem}{Theorem}[section]
\newtheorem*{theorem*}{Theorem}
\newtheorem{lemma}[theorem]{Lemma}
\newtheorem{cor}[theorem]{Corollary}
\newtheorem{prop}[theorem]{Proposition}
\newtheorem*{prop*}{Proposition}

\newtheorem{constr}[theorem]{Construction}
\newtheorem{ex}[theorem]{Example}

\newtheorem{defin}[theorem]{Definition}
\newtheorem*{defin*}{Definition}
\newtheorem{obs}[theorem]{Observation}
\newtheorem{reminder}[theorem]{Reminder}
\newtheorem{notat}[theorem]{Notation}

\newtheorem{rem}[theorem]{Remark}
\newtheorem{warning}[theorem]{Warning}

\Crefname{rmk}{Remark}{Remarks}
\Crefname{cor}{Corollary}{Corollaries}
\Crefname{defn}{Definition}{Definitions}
\Crefname{thm}{Theorem}{Theorems}
\Crefname{corollary}{Corollary}{Corollaries}
\Crefname{axioms}{Axiom}{Axioms}
\Crefname{exercise}{Exercise}{Exercises}
\Crefname{exercisenum}{Exercise}{Exercises}
\Crefname{construction}{Construction}{Constructions}
\Crefname{problem}{Problem}{Problems}
\Crefname{theorem}{Theorem}{Theorems}
\Crefname{definition}{Definition}{Definitions}
\Crefname{proposition}{Proposition}{Propositions}
\Crefname{lemma}{Lemma}{Lemmas}
\Crefname{lem}{Lemma}{Lemmas}
\Crefname{remark}{Remark}{Remarks}
\Crefname{obs}{Observation}{Observations}
\Crefname{reminder}{Reminder}{Reminders}
\Crefname{example}{Example}{Examples}
\Crefname{examplealph}{Example}{Examples}
\Crefname{section}{Section \S\!}{Sections}
\Crefname{subsection}{Subsection \S\!}{Subsections}
\Crefname{summary}{Summary}{Summaries}
\Crefname{warning}{Warning}{Warnings}
\Crefname{part}{Part}{Parts}
\Crefname{conjecture}{Conjecture}{Conjectures}

\crefformat{nul}{(#2#1#3)} 
\crefname{nul}{}{}
\Crefformat{nul}{(#2#1#3)}
\Crefname{nul}{}{}
	
\DeclareFontFamily{U}{rcjhbltx}{}	
\DeclareFontShape{U}{rcjhbltx}{m}{n}{<->rcjhbltx}{}
\DeclareSymbolFont{hebrewletters}{U}{rcjhbltx}{m}{n}

\DeclareMathSymbol{\lamed}{\mathord}{hebrewletters}{108}
\DeclareMathSymbol{\mem}{\mathord}{hebrewletters}{109}
\DeclareMathSymbol{\ayin}{\mathord}{hebrewletters}{96}
\DeclareMathSymbol{\tsadi}{\mathord}{hebrewletters}{118}
\DeclareMathSymbol{\qof}{\mathord}{hebrewletters}{113}
\DeclareMathSymbol{\shin}{\mathord}{hebrewletters}{152}
\DeclareMathSymbol{\tav}{\mathord}{hebrewletters}{116}

\newcommand{\kcpt}{\kappa}

\newcommand{\Z}{\mathbb{Z}}

\newcommand{\C}{\mathbb{C}}
\newcommand{\E}{\mathbb{E}}

\newcommand{\unicodeinfty}{∞}
\newcommand{\Hom}{\operatorname{Hom}}
\newcommand{\Endo}{\operatorname{End}}
\newcommand{\End}{\operatorname{End}}
\newcommand{\Fun}{\operatorname{Fun}}
\newcommand{\LinFun}{\Fun^{\mathrm{L}}_{\mathcal{V}}}
\newcommand{\LinEnd}{\Endo^{\mathrm{L}}_{\mathcal{V}}}
\newcommand{\LinEndW}{\Endo^{\mathrm{L}}_{\mathcal{W}}}

\newcommand{\Arr}{\operatorname{Arr}}
\newcommand{\Ind}{\operatorname{Ind}}
\newcommand{\lInd}{\Ind_{\tav}}
\newcommand{\iHom}{\underline{\Hom}}
\newcommand{\iEnd}{\underline{\Endo}}

\newcommand{\Env}{\operatorname{Env}}

\newcommand{\RMod}{\operatorname{RMod}}
\newcommand{\Mod}{\operatorname{Mod}}

\newcommand{\Spc}{\mathcal{S}}
\newcommand{\Sp}{\mathcal{S} {p}}
\newcommand{\Spcn}{\Sp^{\mathrm{cn}}}
\newcommand{\Spaces}{\Spc}

\newcommand{\Ab}{\operatorname{Ab}}

\newcommand{\Set}{\operatorname{Set}}
\newcommand{\Tw}{\operatorname{Tw}}

\newcommand{\id}{\operatorname{id}}

\newcommand{\op}{\mathrm{op}}

\newcommand{\cat}{\mathop{\mathcal{C}at}\nolimits}

\newcommand{\BMod}[2]{{}_{#1 \!} \operatorname{Bimod}_{#2}}
\newcommand{\Catcolim}{\widehat{\cat} {}^{\mathrm{colim}}}

\newcommand{\Ass}{\operatorname{Ass}}
\newcommand{\Alg}{\operatorname{Alg}}
\newcommand{\CAlg}{\operatorname{CAlg}}
\newcommand{\LM}{\operatorname{LM}}
\newcommand{\RM}{\operatorname{RM}}
\newcommand{\BM}{\operatorname{BM}}
\newcommand{\Op}{\mathop{\mathcal{O} {p}}\nolimits}
\newcommand{\OpAss}{\Op_{\Ass}}
\newcommand{\lOpAss}{\widehat{\Op}_{\Ass}}

\newcommand{\lOpRM}{\widehat{\Op}_{\RM}}

\newcommand{\Monlax}{\operatorname{Mon}^{\text{lax}}}

\newcommand\ev{\operatorname{ev}}

\newcommand{\Fin}{\operatorname{Fin}_*}

\newcommand{\PrL}{\operatorname{Pr}}

\newcommand{\PrV}{\Pr_{\mathcal{V}}}
\newcommand{\PrW}{\Pr_{\mathcal{W}}}
\newcommand{\PrViL}{\operatorname{Pr}_\mathcal{V}^{\mathrm{iL}}}
\renewcommand{\Pr}{\operatorname{Pr}\nolimits}

\newcommand{\Map}{\operatorname{Map}}

\newcommand{\PSh}{{\mathcal{P}}} 
\newcommand{\lPSh}{\hat{\PSh}}
\newcommand{\freecoc}{\PSh^{\tav\text{-rex}}}

\newcommand{\PShV}{{\PSh_\mathcal{V}}}
\newcommand{\PShW}{{\PSh_\mathcal{W}}}

\newcommand{\LMod}{\operatorname{LMod}}

\newcommand{\kerodon}[1]{\cite[\href{https://kerodon.net/tag/#1}{Tag #1}]{kerodon}}
\newcommand{\kerodons}[2]{\cite[\href{https://kerodon.net/tag/#1}{Tags #1} and \href{https://kerodon.net/tag/#2}{#2}]{kerodon}}

\newcommand{\catV}{\cat(\mathcal{V})}
\newcommand{\catW}{\cat(\mathcal{W})}

\newcommand{\vcat}{{\mathop{v\mathcal{C}at}\nolimits}}
\newcommand{\lvcat}{\widehat{\vcat}}
	\newcommand{\vcatsmall}{\vcat^{\tav}}
	\newcommand{\vcatlarge}{\vcat^{\hat{\tav}}}
	\newcommand{\catsmall}{\cat^{\tav}}
	\newcommand{\catlarge}{\cat^{\hat{\tav}}}
\newcommand{\lcat}{\widehat{\cat}}
\newcommand{\llcat}{\doublewidehat{\cat}}

\newcommand{\lSpaces}{\hat{\Spaces}}
\newcommand{\vcatV}{\vcat(\mathcal{V})}
\newcommand{\vcatW}{\vcat(\mathcal{W})}
\newcommand{\vcatS}{\vcat(\Spaces)}
\newcommand{\vcatXV}{\vcat_X(\mathcal{V})}
\newcommand{\vcatXW}{\vcat_X(\mathcal{W})}

\newcommand{\vcatYV}{\vcat_Y(\mathcal{V})}

\newcommand{\Enr}{\mathop{\mathcal{E}nr}\nolimits}
\newcommand{\vEnr}{\mathop{v\mathcal{E}nr}\nolimits}

\newcommand{\lEnr}{\widehat{\Enr}}
\newcommand{\lvEnr}{\widehat{\vEnr}}

\newcommand{\FCat}{{\mathop{\!\mathcal{FC}at}\nolimits}}

\newcommand{\PrVag}{\Pr_{\cV \mathrm{, ag}}}

\newcommand{\lPrL}{\widehat{\Pr}{}}

\newcommand{\MMod}{{\operatorname{Mod}^\star}}
\newcommand{\MModV}{{\operatorname{Mod}^\star_\cV}}
\newcommand{\MModW}{{\operatorname{Mod}^\star_\cW}}
\newcommand{\MModS}{{\operatorname{Mod}^\star_{\Spaces}}}

\newcommand{\MModXV}{{\operatorname{Mod}^\star_{X, \cV}}}
\newcommand{\MModXW}{{\operatorname{Mod}^\star_{X, \cW}}}
\newcommand{\MModYV}{{\operatorname{Mod}^\star_{Y, \cV}}}
\newcommand{\MModbV}{{\operatorname{Mod}^\star_{(-), \cV}}}
\newcommand{\MModXb}{{\operatorname{Mod}^\star_{X,-}}}
\newcommand{\MModb}{{\operatorname{Mod}^\star_{(-)}}}
\newcommand{\MModbb}{{\operatorname{Mod}^\star_{-,-}}}
\newcommand{\MModqb}{{\operatorname{Mod}^\star_{?,-}}}

\newcommand{\Quiv}{\operatorname{Quiv}}
\newcommand{\QuivX}{\Quiv_X }

\newcommand{\QuivXV}{\QuivX(\mathcal{V}) }

\newcommand{\ob}{ob}

\newcommand{\cdom}{\mathrm{cdom}}
\newcommand{\CIm}{\operatorname{CIm}}

\newcommand{\EM}{\operatorname{EM}}
\newcommand{\To}{\Rightarrow}

\newcommand{\IPr}{\mathbb{P}\mathrm{r}}

\newcommand{\IPrV}{\IPr_{\mathcal{V}}}

\DeclareMathOperator*\colim{colim}

\newcommand{\oArr}[2]{\overunderset{#1}{#2}{\downarrow}}	

\makeatletter												
\newcommand{\doublehat}[1]{%
	\begingroup%
	\let\macc@kerna\z@%
	\let\macc@kernb\z@%
	\let\macc@nucleus\@empty%
	\hat{\raisebox{.2ex}{\vphantom{\ensuremath{#1}}}\smash{\hat{#1}}}%
	\endgroup%
}
\makeatother

\makeatletter												
\newcommand{\doublewidehat}[1]{%
	\begingroup%
	\let\macc@kerna\z@%
	\let\macc@kernb\z@%
	\let\macc@nucleus\@empty%
	\widehat{\raisebox{.2ex}{\vphantom{\ensuremath{#1}}}\smash{\widehat{#1}}}%
	\endgroup%
}
\makeatother

\usepackage{eucal}                  

\DeclareFontFamily{T1}{cbgreek}{}
\DeclareFontShape{T1}{cbgreek}{m}{n}{<-6>  grmn0500 <6-7> grmn0600 <7-8> grmn0700 <8-9> grmn0800 <9-10> grmn0900 <10-12> grmn1000 <12-17> grmn1200 <17-> grmn1728}{}
\DeclareSymbolFont{quadratics}{T1}{cbgreek}{m}{n}
\DeclareMathSymbol{\qoppa}{\mathord}{quadratics}{19}
\DeclareMathSymbol{\Qoppa}{\mathord}{quadratics}{21}

\DeclareFontFamily{U}{dmjhira}{}
\DeclareFontShape{U}{dmjhira}{m}{n}{ <-> dmjhira }{}

\DeclareRobustCommand{\yo}{\text{\usefont{U}{dmjhira}{m}{n}\symbol{"48}}}

\newcommand{\yoV}{\yo {}^{\!\mathcal{V}}}

\def\cA{\mathcal A}\def\cB{\mathcal B}\def\cC{\mathcal C}\def\cD{\mathcal D}
\def\cE{\mathcal E}
\def\cJ{\mathcal J}\def\cL{\mathcal L}
\def\cM{\mathcal M}\def\cN{\mathcal N}\def\cP{\mathcal P}
\def\cQ{\mathcal Q}\def\cR{\mathcal R}\def\cS{\mathcal S}
\def\cV{\mathcal V}\def\cW{\mathcal W}\def\cX{\mathcal X}
\def\cY{\mathcal Y}

\def\ccC{C}\def\ccD{D}

\newcommand{\rL}{{\mathrm L}}

\newcommand{\rR}{{\mathrm R}}

\newcommand{\HTTsec}[1]{\href{http://www.math.ias.edu/~lurie/papers/HTT.pdf\#section.#1}{\S #1}}

\newcommand{\HAsubsec}[1]{\href{http://www.math.ias.edu/~lurie/papers/HA.pdf\#subsection.#1}{\S #1}}

\newcommand{\HTTthm}[2]{\href{http://www.math.ias.edu/~lurie/papers/HTT.pdf\#theorem.#2}{#1~#2}}
\newcommand{\HAthm}[2]{\href{http://www.math.ias.edu/~lurie/papers/HA.pdf\#theorem.#2}{#1~#2}}

\newcommand{\SAGthm}[2]{\href{http://www.math.ias.edu/~lurie/papers/SAG-rootfile.pdf\#theorem.#2}{#1~#2}}

\newcommand{\HTT}[2]{\cite[\HTTthm{#1}{#2}]{HTT}}
\newcommand{\HA}[2]{\cite[\HAthm{#1}{#2}]{HA}}

\newcommand{\HAss}[4]{\cite[\HAthm{#1}{#2}, \HAthm{#3}{#4}]{HA}}
\newcommand{\SAG}[2]{\cite[\SAGthm{#1}{#2}]{SAG}}

\tikzset{curve/.style={settings={#1},to path={(\tikztostart)
    .. controls ($(\tikztostart)!\pv{pos}!(\tikztotarget)!\pv{height}!270:(\tikztotarget)$)
    and ($(\tikztostart)!1-\pv{pos}!(\tikztotarget)!\pv{height}!270:(\tikztotarget)$)
    .. (\tikztotarget)\tikztonodes}},
    settings/.code={\tikzset{quiver/.cd,#1}
        \def\pv##1{\pgfkeysvalueof{/tikz/quiver/##1}}},
    quiver/.cd,pos/.initial=0.35,height/.initial=0}

\title{Enriched \texorpdfstring{$\infty$}{\unicodeinfty}-categories as Marked Module Categories}
\author{David Reutter}
\address{Fachbereich Mathematik, Universit\"at Hamburg}
\email{david.reutter@uni-hamburg.de}
\urladdr{https://www.davidreutter.com}

\author{Markus Zetto}
\address{Fachbereich Mathematik, Universit\"at Hamburg}
\email{markus.zetto@uni-hamburg.de}
\urladdr{https://www.markus-zetto.com}

\date{\today{}, Hamburg}

\begin{document}
	
	\pagenumbering{gobble}
	\clearpage
	\hypersetup{pageanchor=false}
	\begin{abstract}
	We prove that an enriched $\infty$-category is completely determined by its enriched presheaf category together with a `marking' by the representable presheaves. More precisely, for any presentably monoidal $\infty$-category $\cV$ we construct an equivalence between the category of $\cV$-enriched $\infty$-categories and  a certain full sub-category of the category of presentable $\cV$-module categories equipped with a functor from an $\infty$-groupoid. This effectively allows us to reduce many aspects of enriched $\infty$-category theory to the theory of presentable $\infty$-categories.
	
	As applications, we use Lurie's tensor product of presentable $\infty$-categories to construct a tensor product of enriched $\infty$-categories with many desirable properties --- including compatibility with colimits and appropriate monoidality of presheaf functors --- and compare it to existing tensor products in the literature. 
	We also re-examine and provide a model-independent reformulation of the notion of  univalence (or Rezk-completeness) for enriched $\infty$-categories.

	Our comparison result relies on a monadicity theorem for presentable module categories which may be of independent interest.

	\end{abstract}
	\maketitle
	\hypersetup{pageanchor=true}
	
	\setcounter{tocdepth}{1}
	\tableofcontents

	
	\pagenumbering{roman}
	
	\section{Introduction}
	\label{sec:intro}

Enriched $\infty$-category theory is a homotopy coherent analog of the classical theory of enriched categories. As such, enriched $\infty$-categories naturally occur whenever morphism spaces carry additional structure,  and provide a flexible and powerful tool for an inductive definition of (weak) $n$-categories and $(\infty,n)$-categories~\cite{haugseng2015rectification}, and structured variants thereof.

There exist several definitions of enriched $\infty$-categories \cite{HA, haugseng, hinich}, all known to be equivalent by  \cite{macpherson2019operad,heine}.  The goal of this paper is to give a further equivalent definition axiomatizing enriched $\infty$-categories in terms of their presheaf categories together with a `marking' by the representable presheaves. A key advantage of this approach is that it does not rely on intricate combinatorial or operadic constructions and effectively reduces enriched $\infty$-category theory to the higher algebra of presentable $\infty$-categories, arguably one of the best-developed and most powerful branches of $\infty$-category theory. Presentable $\infty$-categories, defined in \cite{HTT} generalizing the classical notion of \cite{gabriel2006lokal}, are $\infty$-categories which admit small colimits and satisfy a certain set-theoretic size constraint. They are engineered to allow for a simplified adjoint functor theorem and thus provide a setting in which $\infty$-categories and functors can be easily constructed from universal properties.

\subsection*{Overview}

Intuitively speaking, an $\infty$-category $\cC$ enriched in a monoidal $\infty$-category $\cV$  consists of
	\begin{itemize}
		\item an $\infty$-groupoid $X$ of objects;
		\item for any two objects $x, x' \in X$ a morphism object $\Hom_\cC (x, x') \in \mathcal{V}$;
		\item for each $x \in X$ an identity morphism $\operatorname{id}_x : 1_\mathcal{V} \to \Hom_\mathcal{C} (x,x)$;
		\item for each triple of objects $x_0, x_1, x_2$ a composition operation in $\cV$:
		\begin{equation*}
			\Hom_\mathcal{C} (x_{0}, x_1) \otimes \Hom_{\cC}(x_1, x_2) \to \Hom_{\cC}(x_0, x_2) \; ;
		\end{equation*}
		\item higher coherence isomorphisms witnessing associativity and unitality of composition.
	\end{itemize}
	The definition of enriched $\infty$-category of Gepner and Haugseng \cite[Def.\ 4.3.1]{haugseng} directly encodes this data as an algebra over a certain $\infty$-operad.	If $\cV$ is presentably monoidal, i.e.\ admits small colimits which are preserved by the tensor product in each variable and satisfies a certain set theoretic size constraint, we  follow an equivalent but technically easier approach due to Hinich:
	Namely,  for any $\infty$-category $C$, the $\infty$-category $\Fun(C^{\op}, \cV)$ is the free $\infty$-category with colimits and a $\cV$-module structure which is compatible with colimits. In particular,  there is an equivalence
	\[\Fun(C^{\op}\times C, \cV) \simeq \Fun(C, \Fun(C^{\op}, \cV)) \simeq \LinEnd(\Fun(C^{\op},\cV))\]
	with the category of $\cV$-linear colimit-preserving endofunctors of $\Fun(C^{\op},\cV)$. 
	Transporting the composition of endofunctors along this equivalence defines a monoidal structure on $\Fun(C^{\op}\times C, \cV)$. If $C=X$ is an $\infty$-groupoid, the tensor product of $A, B \in \Fun(X^{\op}\times X, \cV)$ explicitly unpacks to the following `matrix multiplication':
	\[ A\otimes B (x,z):= \colim_{y \in X} A(x,y) \otimes B(y,z)
	\] 
	
\begin{defin*}[{\cite[Prop.\ 4.5.3]{hinich}}] \label{intro:enriched} For $\cV$ a presentably monoidal $\infty$-category\footnote{
The restriction to presentable monoidal enrichment is technically  convenient but does not restrict the scope of the above definition: for $V$ a not necessarily presentably monoidal $\infty$-category, $V$-enriched $\infty$-categories coincide with $\infty$-categories enriched in the presentably monoidal presheaf category $\PSh(V)$ for which all hom-objects lie in the full subcategory $V \subseteq \PSh(V)$. We extend several of our results to this setting in \S \ref{sec:small}. 
} and $X$ an $\infty$-groupoid,  a \emph{$\cV$-enriched $\infty$-category with $\infty$-groupoid of objects $X$} is an algebra object in $ \Fun(X^{\op}\times X, \cV) \simeq \LinEnd(\Fun(X^{\op}, \cV))$. We denote the $\infty$-category of $\cV$-enriched $\infty$-categories\footnote{In the rest of the paper, we refer to them as \emph{valent} enriched categories, to distinguish them from the \emph{univalent} enriched categories from \cref{introthm:univalence}.} with $\infty$-groupoid of objects $X$ by $\vcat_X(\cV):= \Alg(\LinEnd(\Fun(X^{\op},\cV)))$. 
\end{defin*}

A \emph{presheaf} on an enriched $\infty$-category with $\infty$-groupoid of objects $X$ is a functor $F:X^{\op} \to \cV$ together with  morphisms $\Hom(x,y) \otimes F(y) \to F(x)$ in $\cV$, natural in $x,y \in X$, and higher coherence isomorphisms. In the spirit of the previous definition, this can be made precise by defining \[\PShV(\cC):= \LMod_{\cC}(\Fun(X^{\op},\cV))\] using the canonical left action of $\LinEnd(\Fun(X^{\op},\cV))$ on $\Fun(X^{\op},\cV)$.
The presheaf category is itself presentable and admits a $\cV$-action which is compatible with colimits: It thus defines an object of the $\infty$-category $\RMod_{\cV}(\PrL)$ of right $\cV$-modules in $\PrL$, the $\infty$-category of presentable categories and colimit-preserving functors. Moreover,  it comes equipped with a canonical Yoneda functor $\yoV_{\cC}: X \to \PShV(\cC)$ sending an object $x$ to the representable presheaf $\Hom_{\cC}(-, x): X^{\op} \to \cV$.

Our first main theorem is that any $\cC\in \vcat_X(\cV)$ is completely determined by its presheaf category together with its marking by representable presheaves, proving Conjecture 1.1 of \cite{berman}:\begin{introtheorem}[{\cref{thm:charessim}.}] \label{intro:maintheoremA}For $X$ an $\infty$-groupoid and $\cV$ a presentably monoidal $\infty$-category, the functor \[\vcat_X(\cV) \to \RMod_{\cV}(\PrL)_{X/} := \widehat{\cat}_{X/} \times_{\widehat{\cat}} \RMod_{\cV}(\PrL) \hspace{1cm} \cC \mapsto (X\to \PShV(\cC))\] is fully faithful. Its essential image consists of those $(X\to \cM) \in \RMod_{\cV}(\PrL)_{X/}$ for which the image of $X$ in $\cM$ consists of $\cV$-atomic objects\footnote{An object $m\in \cM$ in a presentable $\cV$-module category is \emph{$\cV$-atomic} (also called \emph{tiny}) if the internal hom $\iHom_{\cM}(m,-): \cM \to \cV$ preserves colimits and if for all $v\in \cV$ the canonical morphisms $\iHom_{\cM}(m, -)\otimes v \to \iHom_{\cM}(m, -\otimes v)$ is an isomorphism.}   and generates $\cM$ under colimits and the $\cV$-action. We refer to such a $X\to \cM$ as a \emph{marked $\cV$-module}. 
\end{introtheorem}

Allowing arbitrary $\infty$-groupoids of objects $X$, the categories $\vcat_X(\cV)$ assemble into a Cartesian fibration $ \ob: \vcat(\cV) \to \Spaces$ over the $\infty$-category $\Spaces$ of $\infty$-groupoids, assigning to a $\cV$-enriched $\infty$-category $\cC$ its underlying $\infty$-groupoid of objects $\ob \cC$. 
On the other hand, the dependence on $X$ in marked modules may be dropped by replacing $(\RMod_\cV(\PrL))_{X/}$ by the $\infty$-category\footnote{Here, and below we write $\Arr(\cD):= \Fun([1], \cD)$ for the arrow category of an $\infty$-category $\cD$.} $\Arr(\lcat) \times_{\lcat}\RMod_{\cV}(\PrL)$ of functors from a (large) category to a presentable $\cV$-module.

\begin{introtheorem}[{\cref{prop:coCart} and \cref{thm:functorialcomparison}}] \label{introthm:functorial}
The assignment $\cC \mapsto (\ob \cC \to \PShV(\cC))$ induces a fully faithful functor   \[\vcat(\cV) \hookrightarrow   \Arr(\widehat{\cat}) \times_{\widehat{\cat}} \RMod_{\cV}(\PrL)\] with essential image the marked $\cV$-modules. 
In fact, this functor is itself a component of a natural transformation between functors $\Alg(\PrL) \to \lcat$ relating change-of-enrichment with extension of scalars in $\RMod(\PrL)$. 
\end{introtheorem}

Theorems \ref{intro:maintheoremA} and \ref{introthm:functorial} allow us to reduce many definitions and constructions involving enriched $\infty$-categories to the well-developed toolbox of presentable $\infty$-categories.
We demonstrate this by examining two important concepts in enriched $\infty$-category theory.

Firstly, the following variants of statements in \cite{haugseng} become essentially automatic when expressed in terms of marked modules: 
\begin{introtheorem}[{\cref{cor:GHcomparison}, \cref{prop:univalization}, and \cref{thm:spacesenruniv}}]
\label{introthm:univalence}
Let $\cV$ be a presentably monoidal $\infty$-category.
\begin{enumerate}[(1)]
\item A $\cV$-enriched $\infty$-category $\cC$ is \emph{univalent} (or Rezk complete) in the sense of \cite[Def.~5.2.2]{haugseng}  if and only if the functor $\yoV_{\cC}: \ob \cC \to \PShV(\cC)$ is a subcategory inclusion, i.e.\ induces a fully faithful embedding of $\infty$-groupoids $\ob \cC \hookrightarrow \PShV(\cC)^{\simeq}$. We denote the full sub-$\infty$-category on the univalent enriched $\infty$-categories by $\cat(\cV) \subseteq \vcat(\cV)$.
\item The full inclusion $\cat(\cV) \subseteq \vcat(\cV)$ of the univalent enriched categories admits a left adjoint,  \emph{univalization}, which sends $\cC$ to the enriched $\infty$-category associated to the marked module $\mathrm{Im}(\ob \cC \to \PShV(\cC))^{\simeq} \to \PShV(\cC)$.
\item The functor $\cat(\Spaces) \to \cat$ sending a  $\cC \in \cat(\Spaces)$ to the full image of $\ob \cC \to \PSh_{\Spaces}(\cC)$ is an equivalence. 
\end{enumerate}
\end{introtheorem}
In particular, it follows from \cref{introthm:univalence}(1) (see  \cref{lem:PShValmostff}) that the presheaf functor \[\PShV:\cat(\cV) \to \RMod_{\cV}(\PrL)\] is faithful, i.e.\ induces a monomorphism on mapping spaces: A $\cV$-linear colimit-preserving functor $\PShV(\cC)\to \PShV(\cD)$ arises from a $\cV$-enriched functor $\cC \to \cD$ if and only if $\pi_0 \PShV(\cC)^{\simeq} \to \pi_0 \PShV(\cD)^{\simeq}$ preserves the subsets of representable presheaves. 

Secondly, the marked module perspective allows to use Lurie's tensor product on $\PrL$ to construct monoidal structures on $\vcat(\cV)$ with desirable properties:

\begin{introtheorem}[Corollaries \ref{cor:laxsymmetric} and \ref{cor:catpreservescolims}] \label{introthm:lax}
Consider the natural transformation 
\[\vcat(-) \To \Spaces\times_{\RMod_{-}(\cV)} \Arr(\RMod_{-}(\PrL)) : \Alg(\PrL) \to\lcat
\]
given at $\cV \in \Alg(\PrL)$ by the functor $\vcat(\cV) \to \Spaces\times_{\RMod_{\cV}(\PrL)} \Arr( \RMod_{\cV}(\PrL))$ sending   $\cC \in \vcat(\cV)$ to its underlying $\infty$-groupoid of objects $\ob \cC$ together with the unique $\cV$-linear colimit-preserving extension $\Fun((\ob \cC)^{\op}, \cV) \to \PShV(\cC)$ of its Yoneda functor $\ob \cC \to \PShV(\cC)$.  \begin{enumerate}[(1)] \item  The functor $\vcat(-): (\Alg(\PrL), \otimes) \to (\lcat, \times)$ admits a unique lax symmetric monoidal structure making this natural transformation symmetric monoidal. 
\item The functor $\cat(-):(\Alg(\PrL), \otimes) \to (\lcat, \times)$  admits a unique lax symmetric monoidal structure making the univalization natural transformation $\vcat(-) \To \cat(-)$ symmetric monoidal. 
\end{enumerate}
\end{introtheorem}

\begin{introcor}[Corollaries  \ref{constr:internaltensor}, \ref{cor:catVmonoidal} and \ref{cor:finalcomparisonGH} and \cref{thm:preservescolim}    ] \label{introthm:operad}
Let $O$ be an $\infty$-operad and $\cV$ a presentably $O \otimes \E_1$-monoidal $\infty$-category. 
\begin{enumerate}[(1)]
\item  There is a unique $O$-monoidal structure on $\vcat(\cV)$ making the fully faithful functor $\vcat(\cV) \to \Spaces \times_{\RMod_{\cV}(\PrL)} \Arr(\RMod_{\cV}(\PrL))$ $O$-monoidal.

\item There is a unique $O$-monoidal structure on $\cat(\cV)$ making the univalization functor $\vcat(\cV) \to \cat(\cV)$ $O$-monoidal. 
\end{enumerate}
These $O$-monoidal structures are compatible with colimits and agree with the ones constructed in \cite[Cor. 4.3.12]{haugseng} and  \cite[Cor.~5.7.12]{haugseng}, respectively.
\end{introcor}
In particular, if $\cV$ is a presentably $\E_n= \E_{n-1} \otimes \E_1$-monoidal $\infty$-category for $n\geq 1$, then $\vcat(\cV)$ and $\cat(\cV)$ inherit an $\E_{n-1}$-monoidal structure.
The statement that the $O$-monoidal structure on $\vcat(\cV)$ and $\cat(\cV)$ is compatible with colimits is immediate from the above construction but hard to prove in the model of \cite{haugseng}, see \cite{haugseng2023tensor}.

It also follows from  \cref{introthm:operad}  that $\PShV: \vcat(\cV) \to \RMod_{\cV}(\PrL)$ is $O$-monoidal, a statement that was only very recently  proven in \cite[Cor. 5.1]{heineweighted} while our paper was in preparation.


\subsection*{A monadicity theorem for presentable categories}

By  definition, a $\cV$-enriched $\infty$-category $\cC$ with $\infty$-groupoid of objects $X$ is an algebra in the $\infty$-category $\LinEnd(\Fun(X^{\op}, \cV))$ of $\cV$-linear colimit-preserving endofunctors of $\Fun(X^{\op}, \cV)$. In other words it is a \emph{monad} in the $(\infty,2)$-category $\RMod_{\cV}(\PrL)$ on the object $\Fun(X^{\op}, \cV)$. The key idea in our proof of \cref{intro:maintheoremA} is that the unique $\cV$-linear colimit-preserving extension $\Fun(X^{\op}, \cV) \to \PShV(\cC)$ of its associated marked module $X\to \PShV(\cC)$ is the left adjoint in the monadic adjunction (aka Eilenberg-Moore adjunction) associated to the monad $\cC$. \cref{intro:maintheoremA} is therefore a direct consequence of the following monadicity theorem in $\RMod_{\cV}(\PrL)$: 

\begin{introtheorem}[Theorems~\ref{thm:monadicff} and \ref{thm:monadicity}]\label{introthm:monadicity} Let $\cV$ be a presentably monoidal $\infty$-category and $\cP \in \RMod_{\cV}(\PrL)$. Then, the functor 
\[\Alg(\LinEnd(\cP)) \to \RMod_{\cV}(\PrL)_{\cP/} \hspace{1cm}  A \mapsto  \left(\mathrm{Free}:\cP \to \LMod_A(\cP)\right)
\]
is fully faithful with essential image given by those $F: \cP \to \cM$ in $\RMod_{\cV}(\PrL)_{\cP/}$ whose right adjoint $F^\rR: \cM \to \cP$ is $\cV$-linear and colimit-preserving and which furthermore satisfies one of the following equivalent conditions:
\begin{enumerate}[(1)]
\item The induced functor $\LMod_{F^\rR \circ F}( \cP) \to \cM$  is an equivalence. 
\item The adjunction $F \dashv F^\rR$ is monadic in the sense of Barr-Beck-Lurie ~\HA{Def.}{4.7.3.4}; 
\item $F^\rR$ is conservative;
\item The full image of $F$ generates $\cM$ under colimits. 
\end{enumerate}
\end{introtheorem}

\subsection*{Comparison to other models of enriched $\infty$-categories} We briefly summarize existing models for enriched $\infty$-categories and their relation to marked modules:

\begin{enumerate}[(1)]
\item The \emph{operadic definition}: Gepner and Haugseng \cite{haugseng} define $V$-enriched $\infty$-categories with $\infty$-groupoid of objects $X$ as algebras in $V$ over a certain generalized operad $\Delta_X^{\op}$.  Variants of this are also studied in \cite{hinich} and \cite[\S 2]{stefanichpres}, and compared in \cite{macpherson2019operad, heinecomparison, heine, hinich}. Using an operadic Day convolution, Heine constructs in \cite{heine} an operad $\Quiv_X(V)$ with underlying category $\Fun(X\times X, V)$  so that algebras over $\Delta_X^{\op}$ in $V$ become identified with $\E_1$-algebras in $\Quiv_X(V)$. An equivalent direct construction of $\Quiv_X(V)$ appears in \cite[\S 4]{hinich}. A closely related approach in terms of Segal presheaves is developed in  \cite[\S 4.5]{haugseng}, \cite{haugseng2023tensor}. A key advantage of these models is their straight-forward functoriality in the $\infty$-groupoid of objects $X$ and the enriching category $V$; on the other hand, working with enriched categories is technically challenging.

\item The \emph{monadic definition}: In \cite[Prop.\ 4.5.3]{hinich}, Hinich shows that for presentably monoidal $\infty$-categories $\cV$, the operad $\Quiv_X(\cV)$ is equivalent to the monoidal category $\LinEnd(\Fun(X^{\op}, \cV))$ leading to the definition of enriched $\infty$-category from above. While only relying on very little higher algebra --- essentially only on algebras, monoidal categories and the composition of endofunctors --- this definition has the disadvantage that its functoriality in $X$ and $V$ are much harder to establish, though see \cite[Prop. 4.5.5]{hinich}, \cite[\S 6]{heinemonads}.

\item The \emph{laxly tensored definition}: In \HA{Def.}{4.2.1.28}, Lurie defines a notion of enriched $\infty$-category which may roughly be thought of as a category equipped with a lax, or operadic, action by $V$, and which admits internal homs. In \cite[Prop. 6.10]{heine}, Heine shows that this notion agrees with approach (1). 
\
\end{enumerate}
We add a fourth approach, probably closest in spirit to Lurie's definition:
\begin{enumerate}[(4)]
\item The \emph{marked module definition}: This definition encodes the data of an $\infty$-category enriched in a presentably monoidal category $\cV$ as a presentably $\cV$-module together with a marking by a collection of atomically generating objects.
\end{enumerate}

This definition is both evidently functorial in X and V, and at the same time only involves a minimal amount of operadic coherence work, only involving algebras and modules in presentable categories. Our main Theorems \ref{intro:maintheoremA} and \ref{introthm:functorial} compare these with the approaches (1) and (2), a direct comparison with (3) is also possible.

While approaches (1), (3) work natively for enriching in small monoidal categories, and more generally operads, approaches (2) and (4) are primarily suited for enrichment in presentably monoidal $\cV$. In \S \ref{subsec:smallmonoidal}, we show how one may generalize marked modules for enrichment in small monoidal categories and prove that this also coincides with the other available models.

	\subsection{Plan of the paper}
	
	In \S \ref{sec:quiv}, we recall background material on presentable $\infty$-categories and internal homs. We then define $\cV$-enriched $\infty$-categories as algebras in $\LinEnd(\Fun(X^{\op}, \cV))$ and establish basic properties.
	
	What follows is a short digression \S \ref{sec:atomicgen} developing the theory of presentable module $\infty$-categories by introducing atomic objects, atomically generated $\infty$-categories and internally left adjoint functors. 
	This prepares us for \S \ref{sec:markednew}, where we prove \cref{intro:maintheoremA}; an equivalence between monads on $\Fun(X^{\op}, \cV)$ in $\RMod_{\cV}(\PrL)$ and the full subcategory of $\RMod_{\cV}(\PrL)_{X/}$ on the marked modules. This uses our monadicity \cref{introthm:monadicity} for presentable module $\infty$-categories developed in \S \ref{sec:monadicity}. Subsequently \S \ref{sec:functoriality} explains how to make our definition of enriched categories functorial in both the underlying $\infty$-groupoid and enrichment category and comparing it to the one from \cite{haugseng} resulting in \cref{introthm:functorial}.
	
	Next,  \S \ref{sec:univalence} develops the theory of univalent enriched $\infty$-categories in terms of marked modules and establishes various properties, including the ones in \cref{introthm:univalence}. In \S \ref{sec:multiplicativity} we define an external tensor product of enriched $\infty$-categories leading to the lax symmetric monoidal structures on $\vcat(-)$ and $\cat(-)$ from \cref{introthm:lax} and in particular to the $O$-monoidal structures on $\vcat(\cV$) and $\cat(\cV)$ for presentably $O\otimes \E_1$-monoidal categories $\cV$ as in \cref{introthm:operad}.
	
	\S \ref{sec:small} extends our notion of enrichment to any (not necessarily presentably) monoidal $\infty$-category $V$ by embedding it into a presentably monoidal $\infty$-category. Building on this, \S \ref{sec:comptensor} shows that our tensor products agree with those defined in \cite{haugseng, haugseng2023tensor}, by developing techniques that allow us to switch between different set-theoretic universes. Finally \S \ref{sec:groth} recalls and developes some auxiliary statements on the Grothendieck construction and two-sided fibration, and \S \ref{sec:operads} on operads.

	\subsection{Notation and conventions}
	Throughout this paper we freely use language and statements of $\infty$-category theory as developed in \cite{joyal2002quasi}, \cite{HTT}, \cite{HA} and \cite{kerodon}.

	\begin{itemize}
			\item We use the terms `space' and `$\infty$-groupoid' interchangeably and write $\Spaces$ for their $\infty$-category. 
			\item 
			By a category or $2$-category we always refer to an $(\infty, 1)$-category, $(\infty, 2)$-category; and by a $\cV$-category or $\cV$-enriched category, we mean a $\cV$-enriched $\infty$-category. All categorical constructions, like functors, limits and algebra objects, are homotopy coherent.
			
		\item We fix uncountable inaccessible cardinals $\tav < \hat{\tav} < \doublehat{\tav}$, and call them the universes of \emph{small}, \emph{large} and \emph{very large} sets, respectively\footnote{$\tav$ (\emph{tav} or \emph{taw}) is the last letter in the Hebrew alphabet, used by Cantor to denote the ``absolute infinite''.}.

		\item Unless stated otherwise (e.g.\ by calling them presentable) categories are always small, and by having or preserving all colimits we refer to small colimits. The large analog of a construction is then denoted by a hat -- e.g.\ $\cat$ denotes the large category of small categories, while $\widehat{\cat}$ denotes the very large category of large categories.
				\item The mapping space in a category $\ccC$ is denoted $\Map_\ccC$, the morphism object in an enriched category $\cC$ by $\Hom_\cC$, and the internal Hom in a module category $\cM$ by $\iHom_\cM$.
		\item We use $C \hookrightarrow D$ to denote a (not necessarily full) subcategory inclusion, i.e.\ a functor inducing a monomorphism $C^\simeq \hookrightarrow D^\simeq$ on maximal subspaces, and monomorphisms $\Map_C(c, c') \hookrightarrow \Map_D(Fc, Fc')$ for all $c, c' \in C$. Equivalently, these are the monomorphisms in $\cat$ by \kerodon{04W5}.
		\item Given a functor $F: \ccC \to \cat$ we denote its associated coCartesian fibration by $\smallint\nolimits_\ccC F\to \ccC$, similarly for $G: \ccC^{\op} \to \cat$  its associated Cartesian fibration is $\smallint\nolimits^{\ccC} G \to \ccC$.
		\item We write $\Arr(\ccC) := \Fun([1], \ccC)$ for the \emph{arrow category} of a category $\ccC$, where $[1]$ is the walking arrow. Given a property $P$ that morphisms of $\ccC$ can possess, we write $\Arr^P(\ccC) \subseteq \Arr(\ccC)$ for the respective full subcategory of the arrow category. Similarly, we write $\ccC_{/^{P} c} \subseteq \ccC_{/ c}, \ccC_{c/^{P}} \subseteq \ccC_{c/}$ for the respective full subcategories on arrows satisfying $P$.
		\item A category $\ccC$ is called \emph{weakly contractible} if its geometric realization $|\ccC| \in \Spaces$, i.e.\ its localization at all morphisms, is contractible. A functor $f: \ccC \to \ccD$ is called \emph{right (left) cofinal} if precomposing with it preserves colimits (limits).
		\item A functor $F: C \to D$ \emph{reflects $K$-shaped colimits} if for any cone $\bar{p} : K^\triangleright \to C$ such that $F \circ \bar{p}$ is a colimit cone in $D$, already $\bar{p}$ was a colimit cone in $C$. Further $F$ \emph{creates $K$-shaped colimits} if it both preserves and reflects them. Similarly for limits.
		\item We write $F: C \rightleftarrows D : G$ or $F \dashv G$ for functors $F: C \to D$ and $G: D \to C$ such that $F$ is left adjoint to $G$.

		\item By an operad $O$ we usually mean a small symmetric colored $\infty$-operad, unless explicitly stated otherwise. We denote its \emph{category of operations}, the associated fibration over the category $\Fin$ of finite pointed sets, by $O^\otimes \to \Fin$ 
		and its \emph{underlying category} (or \emph{category of colors}), defined as the fiber over $\langle 1 \rangle \in \Fin$, by $\underline{O}$. We refer to the morphisms covering the terminal map $\langle n \rangle \to \langle 1 \rangle$ in $\Fin$ as the \emph{$n$-ary multimorphisms} or \emph{$n$-ary operations}. We write $\Op$ for the category of operads, and $\OpAss := \Op_{/\Ass}$ for that of non-symmetric operads.  
	\end{itemize}

	\subsection{Acknowledgements}
	We thank Shay Ben-Moshe, Rune Haugseng and Maxime Ramzi   for valuable  discussions and comments relevant to the material of this paper.

	DR and MZ were supported by the Deutsche Forschungsgemeinschaft under the Emmy Noether program – 493608176, and DR also under the Collaborative Research Center (SFB) 1624 “Higher structures, moduli spaces and integrability” – 506632645.
	
		\pagenumbering{arabic}

	
		\section{Enriched categories}
	\label{sec:quiv}
	
	In this section, we first recall background on presentable categories and internal homs, and then define valent enriched categories as outlined in the introduction. We also define the enriched presheaf category and Yoneda embedding.

	\subsection{Recollections on presentable categories}

	A convenient setting for both classical and $\infty$-category theory is provided by the notion of a presentable category. A \emph{presentable category} is a (large but) locally small category which has all small colimits (i.e.\ is \emph{cocomplete}) and is accessible, i.e.\ generated under $\kappa$-filtered colimits by a small set of $\kappa$-compact objects for some regular cardinal $\kappa$; see~\HTT{\S}{5.5.1} for details. The category $\cS$ of spaces is presentable~\HTT{Ex.}{5.5.1.8}, and so is the functor category $\Fun(C, \cD)$ for $C$ a small category and $\cD$ a presentable category, which includes the presheaf category $\PSh(C):= \Fun(C^{\op}, \cS)$ of a small category $C$.
	A main application of the notion of presentable category is the \emph{adjoint functor theorem} \HTT{Cor.}{5.5.2.9}: A functor from a locally small category to a presentable category preserves small colimits (i.e.\ is \emph{cocontinuous})  if and only if it is a left adjoint. 
	
	\begin{notat} We let $\Catcolim$ denote the subcategory of the (very large) category $\lcat$ of large categories on those which admit small colimits and small-colimit preserving functors.  We follow~\HTT{Def.}{5.5.3.1} and let $\PrL$ denote the full subcategory of $\Catcolim$ on the presentable categories. 	\end{notat}
	 For presentable categories $\cC, \cD$, we let $\Fun^\rL(\cC, \cD)$ denote the full subcategory of the functor category $\Fun(\cC,\cD)$ on the functors which preserve colimits.

	Constructed in~\mbox{\HA{Prop.}{4.8.1.15}} is a symmetric monoidal structure on $\Pr$, making the presheaf functor $\PSh: \cat \to \PrL$ symmetric monoidal for the Cartesian symmetric monoidal structure on $\cat$. Its unit is the presentable category $\cS$ of spaces and it is universally characterized by the property that for presentable categories $\cC_1, \cC_2, \cD \in \PrL$ the category $\Fun^\rL(\cC_1 \otimes \cC_2, \cD)$ is equivalent to the full subcategory of the functor category $\Fun(\cC_1 \times \cC_2,\cD)$ on those functors which preserve colimits separately in both variables. By~\HA{Prop.}{4.8.1.17}, the tensor product $\cC_1 \otimes \cC_2$ is equivalent to the category $\Fun^\rL(\cC_1, \cC_2^{\op})^{\op}$. In particular, given a small category $C\in \cat$ and a presentable category $\cD \in \PrL$, there is an equivalence 
	\begin{equation*}
	\PSh(C) \otimes \cD \simeq \Fun^\rL(\PSh(C), \cD^{\op})^{\op} \simeq \Fun(C, \cD^{\op})^{\op} \simeq \Fun(C^{\op}, \cD),
	\end{equation*}
	where the second equivalence follows from the Yoneda lemma. 
	
	A morphism $F: \cC \to \cD$ in $\PrL$ is called a \emph{localization} if it admits a fully faithful right adjoint. By  \HTT{Prop.}{5.5.4.2} $(3)$ this exhibits $\cD$ as the accessible localization of $\cC$ at a small set $S$ of morphisms. By~\cite[Prop.\ 2.13]{monadictower}, such localization form the left class in a factorization system on $\PrL$ whose right class are those functors in $\PrL$ which are conservative. In particular, localizations are closed under colimits in the arrow category $\Arr(\PrL) = \Fun([1], \PrL)$. 
	\begin{lemma} \label{lem:localizationtensor}If $F: \cC \to \cD$ in $\PrL$ is a localization and $\cE \in \PrL$, then $\cE \otimes F: \cE \otimes \cC \to \cE \otimes \cD$ is also a localization. 
	\end{lemma}
\begin{proof}
Since localizations and conservative functors form a factorization system on $\PrL$, the assertion is equivalent to the statement that any conservative left adjoint $\cX \to \cY$ induces a conservative $\Fun^{\mathrm{L}}(\cE, \cX) \to \Fun^{\mathrm{L}}(\cE, \cY)$ which follows since invertibility of a natural transformation can be checked component-wise.
\end{proof}

%

		\subsection{Recollections on presentable module categories}

	We refer to an algebra object $\cV \in \Alg(\PrL)$ as a \emph{presentably monoidal category}; this unpacks to a monoidal category $\cV$ whose underlying category is presentable and so that the tensor product functor $\cV \times \cV \to \cV$ preserves colimits separately in both variables. As the unit in $\PrL$, the category $\Spaces$ of spaces inherits a unique presentably symmetric monoidal structure (namely the Cartesian product), making it an initial object in $\Alg(\PrL)$, see~\HA{Prop.}{3.2.1.8}. For a $\cV \in \Alg(\PrL)$, we denote by $\operatorname{Free}^{\cV}$ the unique left adjoint monoidal functor $\cS \to \cV$, namely the left Kan extension of $1_{\cV}: * \to \cV$ against $* \to \cS$.  Its right adjoint is given by $\Map_{\cV}(1_{\cV}, -) : \cV \to \Spaces$.

	Given a $\cV \in \Alg(\PrL)$, we refer to a module object $\cM \in \RMod_{\cV}(\PrL)$ as a \emph{presentable right $\cV$-module category} (and analogously for left modules); this unpacks to a right module category $\cM$ over $\cV$ whose underlying category is presentable and so that the action functor $ \cM  \times \cV \to \cM$ preserves colimits separately in both variables. 
	
		\begin{notat}
	For $\cV \in \Alg(\PrL)$, we write $\PrV:= \RMod_{\cV}(\PrL)$. 
	\end{notat}

	\begin{ex}\label{ex:algebrastoPrV} If $\cV \in \Alg(\PrL)$ and $a \in \Alg(\cV)$, it follows from~\HA{Cor.}{4.2.3.7} that $\LMod_a(\cV) \in \RMod_{\cV}(\PrL)$ with right $\cV$-action  given by sending a left $a$-module ${}_am$ and an object $v$ in $\cV$ to the left $a$-module ${}_a m\otimes v$.  We will use this repeatedly throughout. 
	\end{ex}

	By the defining universal property, for $\cC, \cD \in \PrL$ there is a functor $\cC \times \cD \to \cC \otimes \cD$ and similarly for $\cM \in \RMod_{\cV}(\PrL)$ and $\cN \in \LMod_{\cV}(\PrL)$ there is a functor $\cM \times \cN \to \cM \otimes_{\cV} \cN$. We refer to objects in the image of these functors as \emph{pure tensors}. 
	
	\begin{lemma}
		\label{lem:puretensors}
		Let $\cC, \cD \in \PrL$. Then, every object of $\cC \otimes \cD$ is a small colimit of pure tensors as defined above. More generally,  if $\cM \in \RMod_\cV(\PrL), \cN \in \LMod_\cV(\PrL)$ then every object of $\cM \otimes_\cV \cN$ is a small colimit of pure tensors.
	\end{lemma}
	\begin{proof}
		Write $\cC = \PSh(\ccC_0)[S^{-1}], \cD = \PSh(\ccD_0)[T^{-1}]$ as accessible localizations of presheaf categories at small sets of morphisms $S$ and $T$, which is possible by combining \HTT{Thm.}{5.5.1.1} $(5)$ and \HTT{Prop.}{5.5.4.2} $(3)$. Then by \cref{lem:localizationtensor} and decomposing $F \otimes G = (F \otimes \operatorname{Id}_{\cD}) \circ (\operatorname{Id}_{\cC} \otimes G)$, we obtain an accessible localization $\PSh(\ccC_0 \times \ccD_0) = \PSh(\ccC_0) \otimes \PSh(\ccD_0) \to \ccC \otimes \ccD$. Every object in $\PSh(\ccC_0 \times \ccD_0)$ is a colimit of objects in $\ccC_0 \times \ccD_0$, which are reflected to pure tensors in $\cC \otimes \cD$, so we are finished since localizations are surjective (as their counit is an isomorphism).
		
		For the relative statement, recall
		\[ \cM \otimes_\cV \cN := \colim_{\Delta^{\op}} \left(\dots \cM \otimes \cV \otimes \cN \rightrightarrows \cM \otimes \cN\right) \]
		and choose a regular cardinal $\kappa$ such that $\cV \in \Alg(\Pr_\kappa), \cM \in \RMod_\cV(\Pr_\kappa), \cN \in \LMod_\cV(\Pr_\kappa)$, where $\Pr_{\kappa}$ denotes the subcategory of $\Pr$ on the $\kappa$-compactly generated categories and functors preserving $\kappa$-compact objects with its tensor product from \HA{Prop.}{4.8.1.15}. Then the full subcategory $\cV^\kappa$ on $\kappa$-compact objects is a monoidal subcategory of $\cV$, and $\cM^\kappa$, $\cN^\kappa$ are $\cV^{\kappa}$-modules. We obtain a monoidal localization functor $\PSh(\cV^\kappa) \to \cV$ and $\PSh(\cV^{\kappa})$-linear localization functors $\PSh(\cM^\kappa) \to \cM$, $\PSh(\cN^\kappa) \to \cN$. Combining \cref{lem:localizationtensor} and that localizations are closed under colimits, we obtain a localization $\PSh(\cM^\kappa \otimes_{\cV^\kappa} \cN^\kappa) \to \cM \otimes_\cV \cN$ where $\cM^\kappa \otimes_{\cV^\kappa} \cN^\kappa$ denotes the relative tensor product with respect to the Cartesian product on $\cat$. We are finished after noticing that the map $\cM^\kappa \times \cN^\kappa \to \cM^\kappa \otimes_{\cV^\kappa} \cN^\kappa$ is surjective, so every object of $\PSh(\cM^\kappa \otimes_{\cV^\kappa} \cN^\kappa)$ is a small colimit of objects in the image of $\cM^\kappa \times \cN^\kappa$.
			\end{proof}

	\subsection{Recollections on internal homs}

	
	\begin{defin}[{\HA{Def.}{4.2.1.28}}]
		\label{reminder:iHom}
		Let $V \in \Alg(\cat)$ be a monoidal category and $M \in \LMod_V(\cat)$ a left $V$-module category. An \emph{internal hom} $\iHom(m, m') \in V$ between objects $m, m'\in M$ is a representing object for the presheaf $\Map_M(- \otimes m, m'): V^{\op} \to \Spaces$. A $V$-module category $M$ is \emph{closed} if an internal hom exists between any pair of objects $m,m'$.  A \emph{closed monoidal category} $V$ is a monoidal category which is closed as a left module over itself. Analogous definitions apply for right $V$-module categories. Passing to a larger universe leads to analogous notions for large categories. 
		\end{defin}
		\begin{obs} 
		A $V$-module category $M$ is closed if and only if the functor \[M^{\op} \times M \to \PSh(V),~m, m' \mapsto \Map_M(- \otimes m, m')\] factors through the full subcategory $V \hookrightarrow \PSh(V)$, establishing functoriality of $\iHom_M(-,-): M^{\op} \times M \to V$. 
		\end{obs}

We now collect a few basic observations about internal homs that we will implicitly use throughout the paper: 

	\begin{obs}
		\label{obs:iHomunderlying}
		Mapping out of the monoidal unit $1_V$ of $V$ recovers the underlying mapping spaces: 
				\begin{equation*}
					\Map_{V} (1_V, \iHom_{M}(m, n)) \simeq \Map_{M}(1_V \otimes m, n) \simeq \Map_M(m,n)\, .
				\end{equation*}
	\end{obs}

	\begin{obs}
	\label{obs:functorialityhom}\label{obs:adjunctioninthom} 
		If $V\in \Alg(\cat)$ is a monoidal category,  $F:M \to N$ is a morphism in $\LMod_{V}(\cat)$ between closed module categories $M$ and $N$, and $m, m' \in M$, then the universal property of internal homs induces a morphism in $V$:
	\[ \iHom_M(m, m') \to \iHom_N(Fm, Fm').
	\]
	If  $F$ admits a right adjoint $F^\rR$ (which inherits a lax $V$-linear structure by \HA{Cor.}{7.3.2.7}), then the composite 
			\[ \iHom_M(m, F^\rR n) \to \iHom_N(Fm, FF^\rR n) \to \iHom_N (Fm, n)
		\]
		is an isomorphism. 
		\end{obs}

	\begin{obs} If $F:V \to W$ is a monoidal functor and $M$ a left $W$ module, denote its restriction to a $V$-module by $F^*M$. Then, if the functor $F$ admits a right adjoint $F^\rR$ and $m,n \in M$ admit an internal hom $\iHom_{M}(m, n) \in W$, it follows from the universal property that $m,n$ also admit an internal hom $\iHom_{F^*M}(m,n) \in V$ which is equivalent to	\[\iHom_{F^*M}(m, n) \simeq F^\rR \iHom_{M}(m,n).
	\] 
	\end{obs}

		\begin{ex}[{\HA{Prop.}{4.2.1.33}}] \label{ex:presentableihom}
 If $\cV \in \Alg(\PrL)$ is a presentably monoidal category, then it follows from the adjoint functor theorem that $\cV$ is closed monoidal and similarly that any presentable module category $\cM \in \LMod_{\cV}(\PrL)$ is closed. \end{ex}

\begin{obs} \label{obs:reducetopresentable} Let $V \in \Alg(\cat) $ be a monoidal category and $M$ a closed left $V$-module category. Then, $\PSh(M)$ is a presentable left  $\PSh(V)$-module category and for $m, m' \in M$ with associated representable presheaves $\yo(m), \yo(m') \in \PSh(M)$ it follows immediately from the definition that
\[\iHom_{\PSh(M)} ( \yo(m),  \yo(m')) \simeq \yo_{\iHom_M(m, m')} \in \PSh(V)~.\]
Thus, computations of internal homs can always be reduced to \cref{ex:presentableihom}.
\end{obs}

	\begin{ex} The category $\cat$ of categories equipped with its Cartesian monoidal structure is presentably monoidal by \HA{Lemma}{4.8.4.2}, in particular closed. The internal hom between $C, D \in \cat$ is the functor category  $\Fun(C, D)$.
	\end{ex}
	
	It follows from \HA{Prop.}{4.3.2.5} that if $a \in \Alg(V)$ is an algebra in a monoidal category $V \in \Alg(\cat)$, then the category $\RMod_a(V)$ of right $a$-modules is a left $V$-module category, i.e.\ in $\LMod_V(\cat)$ with action intuitively given by sending $v\in V$ and $m_a \in \RMod_a(V)$ to the module $v \otimes m_a \in \RMod_a(V)$. In particular, the forgetful functor $\RMod_a(V) \to V$ is a left $V$-module functor, i.e.\ a morphism in $\LMod_V(\cat)$.

	\begin{ex} \label{ex:FunV}
	If $V \in \Alg(\cat) $ is a monoidal category, then the (large) left $\cat$-module category $\RMod_V(\cat)$ is in $\LMod_{\cat}(\PrL)$ by \HA{Prop.}{4.8.3.22} and \HA{Cor.}{4.2.3.7},  and hence is a closed  $\cat$-module category. We denote the internal hom between $M, N \in \RMod_V(\cat)$ by  $\Fun_V(M, N) \in \cat$ and refer to it as the \emph{category of $V$-linear functors}. Applying \cref{obs:functorialityhom} to the forgetful functor $\RMod_V(\cat) \to \cat$ provides a forgetful functor $\Fun_V(M, N) \to \Fun(M, N)$ which we will use implicitly throughout. 
	\end{ex}
	
	\begin{obs}
	\label{obs:equivalentFunV}
	By the Grothendieck construction, a monoidal category is equivalently a coCartesian fibration $V^{\otimes} \to \Ass^{\otimes}$ of operads, and a  pair $(V,M)$ of a monoidal category $V$ and a right $V$-module category $M$ is a coCartesian fibration $M^{\otimes} \to \RM^{\otimes}$ of operads.  In these terms, $\RMod_V(\cat)$ is equivalent to the category $\mathrm{Op}_{/^{\mathrm{coCart}} \RM^{\otimes}} \times_{\mathrm{Op}_{/^{\mathrm{coCart}} \Ass^{\otimes}}} \{V^{\otimes}\}$ and $\Fun_V(M, N)$ is equivalent to  the category of $V$-linear functors from~\HA{Definition}{4.6.2.7}, i.e.\ the full subcategory of $\Fun_{/\RM^\otimes} (M^\otimes, N^\otimes) \times_{\Fun_{/\Ass^{\otimes}}(V^{\otimes}, V^{\otimes})}\{ \id\}$ on those functors $M^{\otimes} \to N^{\otimes}$ that preserve coCartesian morphisms.  If we equivalently regard $M, N$ as locally coCartesian fibrations $M^\oast, N^\oast \to V^\oast$ as in \HA{Notation}{4.2.2.17}, then by \HA{Lemma}{4.8.4.12} $\Fun_V(M, N)$ is also equivalent to the full subcategory of $\Fun_{/V^\oast}(M^\oast, N^\oast)$ on the functors preserving locally coCartesian morphisms,  cf.\ \cite[Lemma 4.1.6]{soergel} and \cite[Rem.\ 3.68, Prop.\ 3.89]{heine}. In the following, we will use all these perspectives on $\Fun_V(M, N)$ interchangeably. 
	\end{obs}
	\begin{rem} If $V$ is $\E_2$-monoidal, then the relative tensor product over $V$ makes the category $\RMod_V(\cat)$ presentably monoidal by~\HA{Cor.}{4.8.5.20} and hence closed monoidal. In this case, the category $\Fun_V(M, N)$ inherits a $V$-module structure making it the internal hom in $\RMod_V(\cat)$. 
	\end{rem}

	For $\cV \in \Alg(\PrL)$ a presentably monoidal category and $\cM, \cN \in \PrV:= \RMod_{\cV}(\PrL)$, we let $\LinFun(\cM, \cN)$ denote the full subcategory of $\Fun_V(\cM, \cN)$ on those $V$-linear functors whose underlying functor $\cM \to \cN$ is colimit-preserving.
%
	\begin{prop}[{cf. \cite[Lemma 4.1.6]{soergel}}]
		\label{prop:LinFuniHom}
		Let $\cV \in \Alg(\PrL)$ be a presentably monoidal category and $\cM, \cN \in \PrV$. Then, $\LinFun(\cM, \cN)$ is presentable and defines an internal hom between $\cM$ and $\cN$ for the left $\PrL$-module structure on $\PrV:=\RMod_{\cV}(\PrL)$ induced by the tensor product of presentable categories. 
	\end{prop}
	\begin{proof}
%
%

	By \HA{Prop.}{4.3.2.5}, $\RMod_{\cV}(\PrL)$ is a left $\PrL$-module category. We now show that $\LinFun(-,-)$ defines an internal hom. 
	By definition, $\LinFun(\cM, \cN)$ is a full subcategory of $\Fun_{\cV}(\cM, \cN) \in \widehat{\cat}$ which is an internal hom between $\cM$ and $\cN$ for the $\widehat{\cat}$-module structure on $\RMod_{\cV}(\widehat{\cat})$ induced by the cartesian product (see~\cref{ex:FunV}). By~\HA{Rem.}{4.8.4.14}, the full inclusion $\LinFun(\cM, \cN) \to \Fun_{\cV}(\cM, \cN)$ is closed under colimits. Let $\Catcolim$ denote the category of large categories with colimits and colimit-preserving functors and note that $\PrL \subseteq \Catcolim$ is a full symmetric monoidal subcategory~\HA{Prop.}{4.8.1.15}.	 Unwinding definitions, it  follows that for $\cP \in \Catcolim$, the full subspace $\Map^{\mathrm{colim}}(\cP, \LinFun(\cM, \cN)) \subseteq \Map(\cP, \LinFun(\cM, \cN))$ on the colimit-preserving functors is the full subspace of $\Map_{\RMod_{\cV}(\widehat{\cat})}(\cP \times \cM, \cN)$ on those functors which preserve colimits separately in both variables and thus we have produced an equivalence \[\Map^{\mathrm{colim}}(\cP, \LinFun(\cM, \cN)) \simeq \Map_{\RMod_{\cV}(\Catcolim)}( \cP \otimes \cM , \cN).\] 
	Thus, $\LinFun(\cM ,\cN)$ is an internal hom between $\cM$ and $\cN$ for the left $\Catcolim$-action on $\RMod_{\cV}(\Catcolim)$.
We are thus finished if we show that $\LinFun(\cM, \cN)$ is presentable. Indeed, it follows from the above equivalence that for $\cC \in \PrL$, there is an equivalence  $\LinFun(\cC \otimes \cV, \cN) \simeq \Fun^{\mathrm{L}}(\cC, \cN)$. But we can resolve any $\cM$ in $\RMod_{\cV}(\Pr)$ by a bar construction of free modules, meaning that $\LinFun(\cM, \cN)$ is a limit in $\Catcolim$ of categories of the form $\LinFun(\cM \otimes \cV^{\otimes n}, \cN) \simeq \Fun^{\mathrm{L}}(\cM \otimes \cV^{\otimes n-1}, \cN)$, which are all presentable by~\HTT{Prop.}{5.5.3.8}. Since the full inclusion $\PrL \subseteq \Catcolim$ is closed under limits (since for both the further subcategory inclusion into $\widehat{\cat}$ is closed under limits, see~\HTT{Prop.}{5.5.3.13} and~\HTT{Cor.}{5.3.6.10}), we deduce that $\LinFun(\cM, \cN)$ is presentable. 
	\end{proof}

	\begin{rem}
		Since $\PrV$ is a closed $\Pr$-module category, we could regard it as a $\Pr$-enriched category (in particular a $2$-category) $\IPrV$ with Hom-categories given by $\LinFun(\cM, \cN)$ using the correspondence between tensoring and enrichment from \cite[Prop. 6.10 (5)]{heine}. While we will not formally use this perspective, it will be useful to keep in mind since we will effectively work with adjunctions and monadic $1$-morphisms internal to this $2$-category.
	\end{rem}

	\begin{prop}
		\label{prop:reltensoradj}
		Let $\cQ, \cV \in \Alg(\PrL)$ and $\cP \in \BMod{\cQ}{\cV}(\PrL)$. The functor $- \otimes_{\cQ} \cP : \Pr_{\cQ} \to \PrV$ admits a right adjoint $\LinFun(\cP, -) : \PrV \to \Pr_{\cQ}$, whose composition with the forgetful functor $\Pr_{\cQ} \to \Pr$ agrees with the internal hom functor from \cref{prop:LinFuniHom}. In particular, for $\cL \in \RMod_{\cQ}(\PrL)$ and $\cM \in \RMod_{\cV}(\PrL)$ the evaluation $\LinFun(\cP, \cM) \otimes_{\cQ} \cP \to \cM$ induces an equivalence
		\[
		\LinFun(\cL \otimes_{\cQ} \cP, \cM) \simeq \Fun^{\rL}_{\cQ} \left(\cL , \LinFun(\cP, \cM) \right) \, .
		\]
	\end{prop}
	\begin{proof}
		Note that $- \otimes_{\cQ} \cP$ is $\PrL$-linear, and $\LinFun$ and $\Fun^{\rL}_{\cQ}$ are the respective internal homs by \cref{prop:LinFuniHom}, so the second assertion follows from the first and \cref{obs:adjunctioninthom}.  Note that $- \otimes_{\cQ} \cP: \RMod_{\cQ}(\Catcolim) \to \RMod_{\cV}(\Catcolim)$ preserves large colimits, and since both sides are presentable after passing to a larger universe it admits a right adjoint $R$. Moreover, using the appropriate internal homs and \cref{obs:adjunctioninthom}, the underlying object of $R ( \cM )$ can be recovered as $\Fun^{\mathrm{colim}}_{\cQ}(\cQ, R \cM) \simeq \LinFun(\cQ \otimes_\cQ \cP, \cM) \simeq \LinFun( \cP, \cM )$. In particular the result is presentable by \cref{prop:LinFuniHom}, and thereby factors through $\Pr_{\cQ}$.
	\end{proof}

We end this section with a recollection of induced algebraic structures on internal ends and homs, and their functoriality:
	
	\begin{reminder}
		\label{reminder:endo}
		Let $V \in \Alg(\cat)$ be a monoidal category,  $M \in \LMod_V(\cat)$ a left $V$-module category and $m \in M$. If the internal end $\iEnd_M (m) := \iHom_M (m, m) \in V$ exists, then by \HA{Cor.}{4.7.1.40} the `composition morphism' \[\iEnd_M(m) \otimes \iEnd_M(m) \to \iEnd_M(m)\] (adjunct to the composition of evaluations $\iEnd_M(m) \otimes \iEnd_M(m) \otimes m \to \iEnd_M(m) \otimes m \to m$)  is the multiplication of an algebra structure on $\iEnd_M(m)$. Indeed, equipped with this algebra structure, $\iEnd_M(m)$ is the terminal object of $\LMod(M) \times_M \{m\}$, i.e.\  the universal algebra in $V$ with a left action on $m$. In particular, for any algebra morphism $a\to \iEnd_M(m)$ in $V$ the associated morphism $a\otimes m \to m$ in $M$ enhances to a left $a$-module structure on $m$ and this construction assembles into an equivalence 
		\[\Alg(V)_{/\iEnd_M(m)} \simeq \LMod(M) \times_M \{m\}
		\]
		compatible with the respective projections to $\Alg(V)$.
	\end{reminder}

	\begin{ex}\label{lem:endounit}
		 For $V\in \Alg(\cat)$, $a\in \Alg(V)$ we write  $a_a \in \RMod_a(V)$ for $a$ considered as a free right $a$-module over itself. The $a$-bimodule structure on $a$ induces a left $a$-module structure on $a_a \in \RMod_a(V)$. Then, it follows from the defining universal properties that the algebra morphism $a\to \iEnd_{\RMod_a(V)}(a_a)$ in $V$ induced by the universal property of~\cref{reminder:endo}  is an isomorphism. 
	\end{ex}
		
	\begin{obs}
		\label{obs:endofunctoriality}
		Let $V\in \Alg(\cat)$ a monoidal category and $F:M \to N$ a morphism in $\LMod_C(\cat)$. For $m \in M$, $C$-linearity of $F$ implies that the defining $\iEnd_M(m)$-action on $m \in M$ induces a $\iEnd_M(m)$-action on $F m $ in $N$. Thus, by the defining universal property, there is an induced algebra homomorphism in $C$
		\[ \iEnd_M(m) \to \iEnd_N(Fm) 
		\]
		whose underlying morphism agrees with the one from~\cref{obs:functorialityhom}. 
	\end{obs}
		


	\begin{obs}\label{obs:boringbimodule}
	Let $V\in \Alg(\cat)$ be a monoidal category, $a\in \Alg(V)$, and consider $\RMod_a(V)$ with its induced left $V$-action. Then, by \cref{reminder:endo} any $x \in \RMod_a(V)$ carries a universal left $\iEnd_{\RMod_a(V)}(x)$-action and hence enhances to an object in \[\LMod_{\iEnd_{\RMod_a(V)}(x)} \left( \RMod_a(V) \right) = \BMod{\iEnd_{\RMod_a(V)}(x)}{a}(V),\] i.e.\ to a $\iEnd_{\RMod_a(V)}(x)$--$a$ bimodule in $V$. 
	\end{obs}

	Intuitively, for a closed module category $M$ over a monoidal category, composition defines an $\iEnd_M(m')$--$\iEnd_M(m)$-bimodule structure on the internal hom $\iHom_M(m,m')$. The following makes this precise:
	\begin{prop}\label{prop:bimoduleihom}
	Let $V\in \Alg(\cat)$ a monoidal category, $M \in \LMod_V(\cat)$ a closed module category and $m,m' \in M$. Then, the morphism 
	\[
	\iEnd_M(m') \otimes \iHom_M(m,m') \otimes \iEnd_M(m) \to \iHom_M(m,m')
	\]
	adjunct to the composition of evaluation maps $\iEnd_M(m') \otimes \iHom_M(m,m') \otimes \iEnd_M(m) \otimes m \to \iEnd_M(m') \otimes \iHom_M(m,m') \otimes m \to \iEnd_M(m') \otimes m' \to m'$ enhances to a $\iEnd_M(m')$--$\iEnd_M(m)$-bimodule structure on $\iHom(m, m')$ for the algebra structures on $\iEnd_M(m')$ and $\iEnd_M(m)$ from~\cref{reminder:endo}.	\end{prop}
	\begin{proof}
	By \cref{obs:reducetopresentable}, we may assume that $\cV \in \Alg(\PrL)$ and $\cM \in \LMod_{\cV}(\PrL)$. Consider the functors \[   \cV \stackrel{\iEnd_M(m') \otimes -}{\longrightarrow} \LMod_{\iEnd_M(m')}(\cV)  \stackrel{- \otimes \iEnd_M(m)}{\longrightarrow} \RMod_{\iEnd_M(m)} \left( \LMod_{\iEnd_M(m')}(\cV)\right) \]\[ \simeq\BMod{\iEnd_M(m')}{\iEnd_M(m)}(\cV) \stackrel{- \otimes_{\iEnd_M(m)} m}{\longrightarrow} \LMod_{\iEnd_M(m')}(\cM)\]
	where the former two are free module functors and the latter is as in~\HA{Ex.}{4.4.2.12} using that $m$ is a (universal) left $\iEnd_M(m)$-module. The composite is equivalent to the functor \[\cV \stackrel{-\otimes m}{\longrightarrow} \cM \stackrel{\iEnd_M(m') \otimes -}{\longrightarrow} \LMod_{\iEnd_M(m')}(\cM).\]
	All these functors are cocontinuous by~\HA{Cor.}{4.4.2.16} and all categories presentable by~\HA{Cor.}{4.3.3.10}, hence all functors admit a right adjoint by the adjoint functor theorem. In particular, since the right adjoints of free module functors are forgetful functors~\HA{Cor.}{4.3.3.13}, it follows that the right adjoint of $\BMod{\iEnd_M(m')}{\iEnd_M(m)}(\cV) \to \LMod_{\iEnd(m')}(\cM)$ defines a functor $\cR: \LMod_{\iEnd_M(m')}(\cM) \to \BMod{\iEnd_M(m')}{\iEnd_M(m)}(\cV)$ whose composite with the forgetful functor to $\cV$ is given by  the composite \[\LMod_{\iEnd_M(m')}(\cM) \stackrel{\text{forget}}{\longrightarrow} \cM \stackrel{\iHom_M(m,-)}{\longrightarrow} \cV.\]
	In particular, applied to the defining module $m' \in \LMod_{\iEnd_M(m')}(\cM)$, this defines a bimodule $\cR(m') \in \BMod{\iEnd_M(m')}{\iEnd_M(m)}(\cV)$ whose underlying object in $\cV$ is given by $\iHom_M(m,m')$. 	
	Unwinding this proof, the left and right action of $\iEnd_M(m)$ and $\iEnd_M(m')$ on $\iHom_M(m,m')$ are as claimed.
	\end{proof}

	\subsection{Enriched categories}
	
	In this section, we recall a definition of enriched $\infty$-category from \cite{hinich}. By~\cite[Prop.\ 4.5.3]{hinich},\cite{macpherson2019operad} this definition agrees with the definition in \cite{haugseng}.
	
	For a presentably monoidal category  $\cV \in \Alg(\PrL)$, we will define a $\cV$-enriched $\infty$-category as a pair of a space $X$ (of objects) together with an algebra in the category $\Fun(X \times X, \cV)$ equipped with a certain `matrix multiplication' monoidal structure. We first construct this monoidal structure.

If $C\in \cat$ is a small category and $\cV\in \Alg(\PrL)$ is a presentably monoidal category, then $\PSh(C) \otimes \cV$ may be thought of as the \emph{free presentable $\cV$-module category} on $C$, as follows: 
	\begin{lemma}
		\label{lem:freetensored}
		Let $C\in \cat$ and $\cV \in \Alg(\PrL)$ with unit $\cS \to \cV$. Then,  precomposing with the composite functor 
		\[C \hookrightarrow \PSh(C) \simeq \PSh(C) \otimes \cS \to \PSh(C) \otimes \cV\]
		induces an equivalence 
				\[\LinFun(\PSh(C) \otimes \cV, \cM) \simeq \Fun(C, \cM)\, .
		\] 
	\end{lemma}
	\begin{proof}
		This follows immediately from the universal properties of $\PSh(X) \otimes \cV$ as the free right $\cV$-module on $\PSh(X)$, and of the presheaf category.
	\end{proof}

	Given a functor $A:C \to \cM$, the corresponding functor $\Fun(C^{\op}, \cV) \simeq \PSh(C) \otimes \cV \to\cM$ is by definition the unique colimit-preserving $\cV$-module functor extending $A$. The following gives an explicit formula for this functor using the \emph{$\infty$-categorical coend}\footnote{The \emph{coend} of a functor $F: C\times C^{\op} \to D$, denoted $\oint^{c\in C} F(c,c) \in D$, is the colimit of the composite $\mathrm{Tw}^r(C) \to C \times C^{\op} \to D$. 
Here, $\mathrm{Tw}^r(C) \to C \times C^{\op} = (C^{\op} \times C)^{\op}$ denotes the \emph{twisted arrow category} of $C$, i.e.\ the right fibration corresponding to the functor $\Map_C(-,-): C^{\op} \times C \to \Spaces$. 
}, developed in~\cite{glasman2016spectrum, NikolausLax, haugseng2022co}.
	
	\begin{lemma}
\label{lem:coendformula}
Let $C\in \cat$, $\cV \in \Alg(\PrL)$ and $\cM \in \PrV$. The equivalence 
\[\LinFun(\Fun(C^{\op}, \cV), \cM) \simeq \Fun(C, \cM)\]
from~\cref{lem:freetensored} combined with the equivalence $\PSh(C) \otimes V \simeq \Fun(C^{\op}, V)$ sends a functor $A\in \Fun(C, \cM)$ to the functor $\widehat{A}\in \LinFun(\Fun(C^{\op}, \cV), \cM)$ explicitly given by
\begin{equation*}
\Fun(C^{\op}, \cV) \ni F \mapsto \oint^{c\in C} A(c) \otimes F(c) \in \cM.
\end{equation*}
If $C=X \in \Spaces$ is a space, this formula reduces to a colimit over $X$:
\begin{equation*}
\Fun(X, \cV) \ni F \mapsto \colim_{x \in X} A(x) \otimes F(x) \in \cM.
\end{equation*}
\end{lemma}
\begin{proof}
Since colimits in functor categories are computed pointwise and since colimits commute with colimits and tensoring, the functor in the first expression preserves colimits. As a composite of right $V$-module functors it itself has the structure of a right $V$-module functor. Thus, it defines an object of $\LinFun(\Fun(C^{\op}, \cV), \cM)$. Unwinding \cref{lem:freetensored}, the equivalence $\LinFun(\Fun(C^{\op}, \cV), \cM) \to \Fun(C, \cM)$ is given by precomposing with $C\to \PSh(C) \simeq \PSh(C) \otimes \Spaces \to  \PSh(C) \otimes \cV \simeq \Fun(C^{\op}, \cV)$, i.e.\ given by $c \mapsto 1_{\cV} \otimes \Map_{C}(-, c) \in \Fun(C^{\op}, \cV)$.   Thus, under this equivalence the functor in our claim is sent to $\oint^{c \in C} A(c) \otimes \Map_{C}(-,c)$ which coincides by the coYoneda lemma~\cite[Rem.~4.2]{haugseng2022co} with $A$.
If $X$ is a space, the coend reduces to an ordinary colimit since $\Tw(X) \simeq X$ by  \kerodons{048H}{048L}.
\end{proof}

	\begin{cor}
	\label{cor:profunctorvscocont}
	For $C, D \in \cat$ and $\cV \in \Alg(\PrL)$, there is an equivalence of categories 
	\begin{align*} \LinFun( \PSh(C) \otimes \cV, \PSh(D) \otimes \cV)&  \simeq \Fun(C, \PSh(D) \otimes \cV) \simeq \Fun(C, \Fun(D^{\op}, \cV))\\& \simeq \Fun(C \times D^{\op}, \cV) \simeq \Fun(D^{\op} \times C, \cV) \, .
	\end{align*}
	\end{cor}
	\begin{proof} Combine~\cref{lem:freetensored} with the equivalence $\PSh(D) \otimes \cV \simeq \Fun(D^{\op}, \cV)$.
	\end{proof}

		\begin{rem}
	A \emph{$\cV$-valued profunctor} (or \emph{$\cV$-bimodule}) between  categories $C$ and $D$  is by definition a functor $D^{\op} \times C \to \cV$. Thus, \cref{cor:profunctorvscocont} is an $\infty$-categorical version of the classical correspondence between profunctors and cocontinuous module functors. In~\cref{rem:profunctor}, we will generalize these to $\cV$-enriched profunctors, allowing for $C,D$ to be $\cV$-enriched categories and $D^{\op} \times C \to \cV$ to be a $\cV$-enriched functor.
	\end{rem}
	
	Under the correspondence from~\cref{cor:profunctorvscocont}, the composition of cocontinuous $\cV$-module functors turns into a `matrix multiplication' of profunctors:
	
	\begin{cor}\label{cor:profunctorcomp}
	For $\cV \in \Alg(\PrL)$, $A: E^{\op} \times D \to \cV$ and $B:D^{\op} \times C\to \cV$  with image under the equivalence from~\cref{cor:profunctorvscocont} denoted $\widehat{A}: \PSh(D) \otimes \cV \to \PSh(E) \otimes \cV$ and $\widehat{B}: \PSh(C) \otimes \cV \to \PSh(D) \otimes \cV$, the composition $\widehat{A} \circ \widehat{B}$ corresponds under~\cref{cor:profunctorvscocont} to the functor $E^{\op} \times C \to \cV$ given by 
	\[ 
	(e, c) \mapsto \oint^{d\in D} A(e, d) \otimes B(d, c).
		\]
		If $D=X$ is a space, this becomes an ordinary colimit over $X$:
			\[ 
	(e, c) \mapsto \colim_{d \in D} A(e,d) \otimes B(d, c).
		\]
	\end{cor}
	\begin{proof} 
	Immediate from~\cref{lem:coendformula}. 
	\end{proof}

	\begin{defin}
		\label{defin:quivend}
		For $C\in \cat$ and $\cV \in \Alg(\PrL)$, we define the \emph{quiver monoidal category}
		\[\LinEnd(\Fun(C^{\op}, \cV)) \simeq \LinEnd(\PSh(C) \otimes \cV) \simeq \Fun(C^{\op} \times C, \cV) \in \Alg(\PrL)
		\]
		which by~\cref{prop:LinFuniHom} is the internal end of $\Fun(C^{\op}, \cV) \simeq \PSh(C) \otimes \cV \in \PrV$ and thus by \cref{reminder:endo} an algebra in $\PrL$ with monoidal structure given by composition of endo-functors. 	
	\end{defin}
	\begin{rem}
		\label{rem:quivspaces}
		Under the equivalence \[\LinEnd(\PSh(C) \otimes \cV)   \simeq \Fun(C^{\op}\times C, \cV) \simeq \PSh(C \times C^{\op}) \otimes \cV \simeq \End^\rL(\PSh(C)) \otimes \cV\] the composition monoidal structure from~\cref{defin:quivend} agrees with the induced monoidal structure as a product of algebras $ \End^\rL(\PSh(C))$ and $\mathcal{V}$ in $\Alg(\PrL)$ by \cite[Prop.\ 3.10]{redshift}.
	\end{rem}
	

	\begin{obs}
		\label{prop:matrixproduct} 	
		For $C \in \cat$ and $\cV \in \Alg(\PrL)$, \cref{cor:profunctorcomp} yields an explicit formula for the monoidal structure on $\LinEnd(\PSh(\cC) \otimes \cV) \simeq \Fun(C^{\op} \times C, \cV)$, agreeing with the anticipated `matrix multiplication' from \S\ref{sec:intro}: 
		Indeed, for $A, B \in  \Fun(C^{\op} \times C, \cV)$ we have
		\[
		(A \otimes B) (c, c') := \oint^{d\in C} A(c, d) \otimes B(d, c').
		\] 
		In particular, if $C = X \in \Spaces$ is a space this simplifies to
		\[
		(A\otimes B)(x, x') = \colim_{y \in X} A(x, y) \otimes B(y, x') \, .
		\]
	\end{obs}


	\begin{defin}[{\cite[Prop.\ 4.5.3]{hinich}}]\label{def:valentVcat}
	Let $\cV\in \Alg(\PrL)$ and $X\in \Spaces$. The category of \emph{valent\footnote{The term `valent' is used for distinction with the `univalent' or `Rezk complete' $\cV$-enriched categories from \cref{def:univalence}. In a valent $\cV$-enriched category, the space of objects may not coincide with the maximal subgroupoid of the underlying category.} $\cV$-enriched categories with space of objects $X$} is the category \[ \vcatXV:=\Alg(\LinEnd(\PSh(X) \otimes \cV)) \;  ,\]where $\LinEnd(\PSh(X) \otimes \cV)$ comes equipped with the composition monoidal structure from \cref{defin:quivend}. The \emph{graph} of a valent $\cV$-enriched category $\cC \in \vcatXV$ is its underlying object in $\Fun(X \times X, \cV) \simeq \LinEnd(\PSh(X) \otimes \cV).$
	\end{defin}

	\begin{prop}\label{obs:pCatXpres} For $X\in \Spaces, \cV \in \Alg(\PrL)$, the category $\vcat_X^{\cV}$ of valent $\cV$-enriched categories is presentable.
	\end{prop}
	\begin{proof} Since $\LinEnd(\PSh(X) \otimes \cV)$ is presentably monoidal, this immediately follows from \HA{Cor.}{3.2.3.5}.
	\end{proof}

	\begin{ex}
		\label{ex:quivpoint}
		If $X = *$ is the point, $\LinEnd(\PSh(*) \otimes \cV)= \Fun(* \times *, \cV) = \cV$, and its monoidal structure agrees with the original one on $\cV$ by \cref{lem:endounit}.  Hence, $\vcat_*(\cV) \simeq \Alg(\cV)$;  valent $\cV$-categories with space of objects $*$ are equivalently algebras in $\cV$; $\vcat_{*}(\cV) \simeq \Alg(\cV)$.
	\end{ex}



\subsection{Change of enrichment}
Given a morphism $f:\cV \to \cW$ in $\Alg(\PrL)$, we now construct an induced adjunction $\vcatXV \rightleftarrows \vcatXW$ which sends an enriched category $\cC$ to the enriched category $f_!\cC$ with same space of objects and 
\[\Hom_{f_! \cC}(x, x') = f (\Hom_{\cC}(x,x')) \in \cW.
\]
In \S \ref{sec:functoriality}, we will make this fully functorial and compare it with the literature. 

	\begin{constr}
		\label{constr:changeofenrquiv}
		For $X\in \Spaces$ and $f: \cV \to \cW$ a morphism in $\Alg(\PrL)$, the extension-of-scalars functor $- \otimes_\cV \cW : \PrV \to \PrW$ is a left $\PrL$-module functor and hence induces by \cref{obs:endofunctoriality} a morphism in $\Alg(\PrL)$, i.e.\  a colimit-preserving monoidal functor
		\[ f_!: \LinEnd(\PSh(X) \otimes \cV) \to \LinEndW(\PSh(X) \otimes \cW).
		\] Unwinding this construction, on underlying categories of graphs 
			\[
		\Fun(X \times X, \cV) \to \Fun(X \times X, \cW)
		\]
		 this functor acts by postcomposition with $f$. 
		 Hence, it admits an (automatically laxly monoidal) right adjoint $f^\rR_!$ given by postcomposition with $f^\rR$. We therefore obtain an adjunction of \emph{change-of-enrichment functors} between the respective categories of algebras
		\[
		f_! : \vcatXV = \Alg(\LinEnd(\PSh(X) \otimes \cV)) \rightleftarrows \Alg(\LinEndW(\PSh(X) \otimes \cW)) = \vcat_X(\cW) : f^\rR_! \; .
		\]
	\end{constr}

	\begin{prop} 		\label{obs:changeofenrff}Let $X\in \Spaces$ and $f: \cV \to \cW$ in $\Alg(\PrL)$. Then, if $f$ is fully faithful, so is the induced functor $f_! : \vcatXV \to \vcatXW$. Its image is the full subcategory of $\vcatXW$ on those $\cW$-categories whose graph $X \times X \to \cW$ factors through $\cV \subseteq \cW$. 
Similarly, if the right adjoint $f^\rR$ is fully faithful, then so is the induced right adjoint $f_!^\rR: \vcatXW \to \vcatXV$ and its image is given by those $\cV$-categories whose underlying graph $X \times X \to \cV$ factors through $f_!^\rR: \cW \hookrightarrow \cV$. 
	\end{prop}
	\begin{proof}
	If $f$ is fully faithful, the unit of $f\dashv f^\rR$ is an isomorphism. Since the adjunction $\Fun(X \times X, \cV) \rightleftarrows \Fun(X \times X, \cW)$ is given by postcomposition, it follows that the unit of that adjunction is also invertible. By conservativity of the forgetful functor $\vcatXV = \Alg(\LinEnd(\PSh(X) \otimes \cV)) \to \LinEnd(\PSh(X) \otimes \cV) \simeq \Fun(X \times X, \cV)$, it follows that the unit of the adjunction $\vcatXV \rightleftarrows \vcatXW$ is an isomorphism, and hence that the left adjoint is fully faithful. The proof for the right adjoint is analogous. 
	\end{proof}

	\subsection{Enriched presheaves and the Yoneda embedding}
	\label{sec:presheaves}

	By \cref{def:valentVcat}, a $\cV$-enriched category $\cC$ with space of objects $X$  is an algebra object in the presentably monoidal category $\LinEnd(\PSh(X) \otimes \cV)$. Since this is the endomorphism algebra of the object $\PSh(X) \otimes \cV \in \PrV$ and hence acts on it (universally) from the left, we may consider the category of left $\cC$-modules in $\PSh(X) \otimes \cV$:

	\begin{defin}\label{def:presheaf} For $X \in \Spaces, \cV \in \Alg(\PrL)$ and $\cC \in \vcat_{X}^{\cV} := \Alg(\LinEnd(\PSh(X) \otimes \cV))$, we define its \emph{enriched presheaf category} as 
	\[\PSh_{\cV}(\cC):= \LMod_{\cC}( \PSh(X) \otimes \cV). 
	\]
	We define the \emph{enriched Yoneda embedding of $\cC$} to be the composite 
	\[\yoV_\cC : X \overset{\yo}{\hookrightarrow} \PSh(X) \overset{\id \otimes 1_{\cV}}{\to} \PSh(X) \otimes \cV  \overset{\mathrm{Free}}{\to} \LMod_{\cC}(\PSh(X) \otimes \cV) =: \PShV(\cC),
	\]
	where $\yo$ denotes the ordinary Yoneda embedding of $X$ and $\mathrm{Free}$ the free $\cC$-module functor.  An object of $\PShV(\mathcal{C})$ is called \emph{representable} if it is in the image of $\yoV_{\cC}$. 
	For $x\in X$, we define the \emph{evaluation functor} $\ev_x: \PShV(\cC) \to \cV$ as the composite
	\[\PShV(\cC) := \LMod_{\cC}(\PSh(X) \otimes \cV)  \overset{\mathrm{Forget}}{\to} \PSh(X) \otimes \cC \simeq  \Fun(X^{\op}, \cV) \overset{\ev_x}{\to} \cV.
	\]
	\end{defin}

	\begin{obs}\label{obs:presheafunpacked}
	By definition, an enriched presheaf of a $\cV$-enriched category $\cC$ with space of objects $X$ is therefore an object of $\LMod_{\cC}(\PSh(X) \otimes \cV)\simeq \LMod_{\cC}(\Fun(X^{\op}, \cV))$, i.e.\ a functor $F:X^{\op} \to \cV$ together with a left action of the algebra object $\Hom_\cC \in \Fun(X^{\op} \times X, \cV)$ via the universal action of $\Fun(X^{\op} \times X, \cV) = \LinEnd(\Fun(X^{\op}, \cV))$ on $\Fun(X^{\op}, \cV)$. By~\cref{cor:profunctorcomp}, this amounts to morphisms in $\cV$
	\[
		\Hom_\cC (x, x') \otimes F(x') \to F(x)
	\]
	natural in $x, x' \in X$ which are coherently compatible with the composition in $\cC$. 
	
	In these terms,  for an $x\in X$ the evaluation functor $\ev_x: \PShV(\cC) \to \cV$ sends $F$ to $F(x)$ and the representable presheaf $\yoV_{\cC}(x)$ unpacks to the functor $\Hom_{\cC}(-, x) : X^{\op} \to \cV$.
	\end{obs}
	
	\begin{rem}
		\label{rem:profunctor}
		Given enriched categories $\cC \in \vcatXV, \cD \in \vcatYV$, by \cref{prop:bimoduleihom} the internal hom $\LinFun(\PSh(X) \otimes \cV , \PSh(Y) \otimes \cV) \simeq \Fun(X \times Y^{\op}, \cV)$ in $\PrV$ enhances to an $\LinEnd(\PSh(Y) \otimes \cV)$-$\LinEnd(\PSh(X) \otimes \cV)$-bimodule. We refer to objects of $P \in \BMod{\cD}{\cC}(\Fun(X \times Y^{\op}, \cV))$ as \emph{enriched profunctors} from $\cC$ to $\cD$, since by~\cref{cor:profunctorcomp} such a bimodule structure amounts to morphisms
		\[
		\Hom_\cD (y, y') \otimes P(y', x') \otimes \Hom_\cC (x', x)  \to P(y,x)
		\]
		natural in $x, x' \in X, y, y' \in Y$ which are coherently compatible with the composition in $\cC, \cD$. Compare with the classical (c.f.\ \cite{benabou1973distributeurs}) notion of enriched profunctors as enriched functors $\cC \otimes \cD^{\op} \to \cV$.
	\end{rem}
		
	\begin{prop}\label{prop:presentablepresheaf}
	For $X\in \Spaces, \cV \in \Alg(\PrL)$ and $\cC \in \vcat_X^{\cV}$, the enriched presheaf category $\PSh_{\cV}(\cC)$ is a presentable right $\cV$-module category, i.e.\ an object of $\PrV:=\RMod_{\cV}(\PrL)$. Moreover, the functor $\yoV_{\cC}:X \to \PSh_{\cV}(\cC)$ extends via \cref{lem:freetensored} (uniquely) to the cocontinuous $\cV$-module functor $\mathrm{Free}: \PSh(X) \otimes \cV \to \LMod_{\cC}(\PSh(X) \otimes \cV) =: \PSh_{\cV}(\cC)$. 
	\end{prop}
	\begin{proof}
	\cref{obs:boringbimodule} enhances $\PSh(X) \otimes \cV \in \PrV$ to an object of$\BMod{\LinEnd(\PSh(X) \otimes \cV)}{ \cV}(\PrL)$, which is equivalent to $\LMod_{\LinEnd(\PSh(X) \otimes \cV)}(\PrV)$ by~\HA{Thm.}{4.3.2.7}. Thus,  the category  $\PSh_{\cV}(\cC):= \LMod_{\cC}(\PSh(X) \otimes \cV)$ retains a right $\cV$-module structure, which by~\HA{Cor.}{4.2.3.7} is a presentable module category. Moreover, $\mathrm{Free}: \PSh(X) \otimes \cV \to \LMod_{\cC}(\PSh(X) \otimes \cV)$ is a $\cV$-linear left adjoint and hence a morphism in $\PrV$. Thus, by \cref{lem:freetensored}  it is the unique extension of its restriction to $X$ which is $\yoV_{\cC}$. 
	\end{proof}

	With this definition, one may easily prove the following enriched variant of the Yoneda lemma:

	\begin{prop}\label{prop:Yonedafancy}
	For $\cV \in \Alg(\PrL)$, $X\in \Spaces, \cC \in \vcatXV$ and  $x\in X$, the functor $\ev_x: \PShV(\cC) \to \cV$ is right adjoint to the functor $\yoV_{\cC}(x)\otimes -: \cV \to \PShV(\cC)$ and is thus equivalent to $\iHom_{\PShV(\cC)}(\yoV_{\cC}(x),-)$. 
	\end{prop}
	\begin{proof}
	By definition, the left adjoint of $\ev_x$ is the functor \[ \cV \overset{\yo_x \otimes \id_{\cV}}{\to} \PSh(X) \otimes \cV \overset{\mathrm{Free}}{\to}  \LMod_{\cC}(\PSh(X) \otimes \cV) =: \PShV(\cC) \]. Since $\mathrm{Free}$ is a $\cV$-module functor, this agrees with $\mathrm{Free}(\yo_x) \otimes - =: \yoV_{\cC}(x) \otimes-$. 
	\end{proof}
	
	\begin{obs}\label{prop:Yoneda}\label{cor:PShVgraph}
	It immediately follows from~\cref{prop:Yonedafancy} that for $x\in X$ and $F\in \PShV(\cC)$ we have 
	\[\iHom_{\PShV(\mathcal{C})}(\yoV_{\cC}(x), F) \simeq \operatorname{ev}_x(F) \in \mathcal{V}\; .\]
	In particular, it follows from the unpacking in~\cref{obs:presheafunpacked} that for $x,x' \in X$
\[\iHom_{\PShV(\mathcal{C})}(\yoV_\cC(x), \yoV_\cC(x')) \simeq \Hom_{\mathcal{C}}(x,x') \in \mathcal{V} \; .\]
	\end{obs}

Our main  \cref{thm:charessim} may be understood as a vast generalization of~\cref{cor:PShVgraph}, showing that the full structure of the enriched category $\cV$ can be recovered from its enriched presheaf category $\PShV(\cC)$ together with its Yoneda embedding $\yoV_{\cC}:X \to \PShV(\cC)$.
	
	\begin{warning}
		Generally, the functor $\yoV_{\cC}: X \to \PShV(\cC)$ is \emph{not} a subcategory inclusion. In particular, two representable presheaves $\yoV_\cC(x) \simeq \yoV_\cC(y)$ can be isomorphic in $\PShV(\cC)$ even though $x$ and $y$ do not lie in the same connected component of $X$. Enriched categories for which $X \to \PShV(\cC)$ is a subcategory inclusion, i.e.\ exhibit $X$ as the maximal subgroupoid of the full image of $X$ in $\PShV(\cC)$,  are called \emph{univalent} (or `complete' in the terminology of~\cite{haugseng}) and will be discussed in \S \ref{sec:univalence}.
	\end{warning}

	\section{Atomic objects and atomically generated categories}
	\label{sec:atomicgen}
	
	Our main theorem expresses the data of a $\cV$-enriched category $\cC$ in terms of its presheaf category $\PShV(\cC)$. In this section, we study the properties satisfied by such presheaf categories. 
	Many of the following statements, as well as their proofs, are direct analogs of similar results for presentable categories; and variants of many of them can e.g.\ be found in \cite{ben2024naturality}, \cite{ramzi2024dualizable}, \cite{redshift}.

\subsection{Generating subcategories and colimit-dominant functors}

Recall the following terminology, which we will use throughout this paper:

	\begin{defin}\label{def:generating}
	For $\cC \in \PrL$, we say that a full subcategory $\cC_0 \subseteq \cC$ \emph{generates $\cC$ under colimits} if $\cC$ is the smallest full subcategory of $\cC$ which contains $\cC_0$ and is closed under small colimits. 
	For $\cV \in \Alg(\PrL)$ and  $\cM \in \PrV$, we say that a full subcategory $\mathcal{M}_0 \subseteq \mathcal{M}$  \emph{generates $\cM$ under colimits and tensoring} if $\cM$ is the smallest full subcategory of $\cM$ which contains $\cM_0$ and is closed under small colimits and tensoring with objects of $\cV$. 
	\end{defin}
It will follow from \cref{prop:conservativegen} that a full subcategory $\cM_0 \subseteq\cM$ generates $\cM$ under colimits and tensoring if and only if the functors $\{\iHom_{\cM}(m_0, -): \cM \to \cV\}_{m_0 \in \cM_0}$ are jointly conservative.

\begin{notat} We let $\operatorname{Im}(F:C \to D)$ denote the \emph{full image} of  a functor $F: C \to D$ in $\cat$.
\end{notat}

\begin{defin} The \emph{closed image} $\CIm(F) \subseteq \cM$ of a morphism $F: \cN \to\cM$ in $\PrL$ is the smallest full subcategory of $\cM$ which contains the full image of $F$ and is closed under colimits. 
\end{defin}
If clear from context, we sometimes also write $\CIm(\cN \to \cM)$ or just $\CIm(\cN)$.

\begin{lemma}
	\label{lem:imageispres}
	The closed image $\CIm(F)$ of a morphism $F: \cN \to \cM$ in $\PrL$ is presentable and the inclusion $\CIm(F) \subseteq \cM$ is closed under colimits, i.e.\ defines a morphism in $\PrL$.
\end{lemma}
\begin{proof}
By definition, $\CIm(F)$ is closed under colimits in $\cM$. We show that $\CIm(F)$ is $\kappa$-compactly generated for a regular cardinal $\kappa$. Indeed, let $\kappa$ be a regular cardinal such that $\cN$ is $\kappa$-compactly generated and $F$ preserves $\kappa$-compact objects. Since $F$ preserves colimits, the image $F(\cN^{\kcpt})$ of the full subcategory $\cN^{\kcpt} \subseteq \cN$ on the $\kappa$-compacts generates $\CIm(F)$ under colimits. But by assumption on $\kappa$ the objects in $F(\cM^{\kcpt})$ are $\kappa$-compact in $\cM$ and hence also in $\CIm(F)$. 
\end{proof}

\begin{warning}The full image $\operatorname{Im}(F)$ is generally not closed under arbitrary colimits.  Consider for example the unique functor $\Spaces \to \Ab$ in $\PrL$ sending the point to $\Z$, but which does not contain $\Z_2 = \operatorname{coker}(\Z \overset{\cdot 2}{\to} \Z)$ in its image. \end{warning}
	However,  $\operatorname{Im}(F) \subseteq \cM$ is always closed under coproducts, since by surjectivity of $\cN \to \operatorname{Im}(F)$ for any set $X$ any diagram $X\to \operatorname{Im}(F)$ can be lifted to a diagram $X\to \cN$ and its colimit can be computed there. 


\begin{defin}\label{defin:colimitdominant}
A morphism $F: \cN \to \cM$ in $\PrL$ is called \emph{colimit-dominant} if $\CIm(F) = \cM$, i.e.\ if its full image generates $\cM$ under colimits. A morphism in $\PrV$ is called colimit-dominant if its underlying morphism in $\PrL$ is. 
\end{defin}

\begin{prop}[{\cite[Prop.\ 2.9]{monadictower}}]
	\label{lem:colimit-dominant-spaces}
	A morphism $F: \cN \to \cM$ in $\PrL$ is colimit-dominant if and only if its right adjoint $F^\rR$ is conservative. \end{prop}

\begin{cor}
	\label{obs:cancellationcolimdom}
	The composition of colimit-dominant functors in $\PrL$ is colimit-dominant. Moreover, if $F: \cN \to \cM$ and $G: \cM \to \cL$ are morphisms in $\PrL$ and $G\circ F$ is colimit-dominant, then $G$ is colimit-dominant. 
\end{cor}
\begin{proof}
This follows from the respective properties of conservative functors.
\end{proof}

\begin{lemma}\label{lem:colimdomtensor}
Let $F: \cN \to \cM$ in $\PrL$ be colimit-dominant and $\cL \in \PrL$. Then, $F\otimes \cL : \cN \otimes \cL \to \cM \otimes \cL$ is colimit-dominant. 
\end{lemma}
\begin{proof}
In~\cite[Cor.\ 2.10]{monadictower} it is shown that colimit-dominant functors are the left class of a factorization system (in the  sense of ~\HTT{\S}{5.2.8}) on $\PrL$ with right class the fully faithful left adjoints. (In~\cref{lem:prfact} below, we lift this to a factorization system on $\PrV$.) Since $\PrL$ is closed monoidal, the statement of the lemma is thus equivalent to the assertion that for every fully faithful left adjoint $G: \cC \to \cD$ the induced morphism $\Fun^{\mathrm{L}}(\cL, \cC ) \to \Fun^{\mathrm{L}}(\cL, \cD)$ is fully faithful which is clear as both categories are full subcategories of functor categories and post-composition with fully faithful functors remains fully faithful. 
\end{proof}

We now extend these notions to $\PrV$ for a $\cV \in \Alg(\PrL)$. 

\begin{lemma}
	\label{lem:imagetens}
	Let $\cM \in \PrV$ and $\cM_0 \subseteq \cM$ be a full subcategory which is closed under $\cV$-tensoring. If $\hat{\cM}_0$ is the smallest full subcategory of $\cM$ that is closed under colimits and contains $\cM_0$, then $\hat{\cM}_0$ is also closed under $\cV$-tensoring. In particular, $\cM_0$ generates $\cM$ under colimits and tensoring if and only if it generates $\cM$ under colimits. 
\end{lemma}
\begin{proof}
	Let $\widetilde{\cM}$ denote the full subcategory of $\hat{\cM}_0$ on those $m$ such that $m \otimes v \in \hat{\cM}_0$ for all $v\in \cV$. This contains $\cM_0$ and is closed under colimits, hence is all of $\hat{\cM}_0$. 
\end{proof}

\begin{prop} \label{cor:CImV}Let $F: \cN \to \cM$ be a morphism in $\PrV$. Then, $\CIm(F) \subseteq \cM$ is the smallest full subcategory of $\cM$ which contains the image of $F$ and is closed under colimits and tensoring. In particular, there is a factorization $\cN \to \CIm(F) \to \cM$ in $\PrV$. 
\end{prop}
\begin{proof}Immediate from~\cref{lem:imagetens} since the full image of $F$ is closed under tensoring. 
\end{proof}

it follows from~\cref{cor:CImV} that $\CIm(F)$ is also closed under tensoring, justifying \cref{defin:colimitdominant} for morphisms in $\PrV$.

\begin{cor}\label{cor:imagevsPsh}
Let $\cC \in \cat$ be a small category, $\cM \in \PrV$ and $F: \cC \to \cM$ a functor with induced morphism $\PSh(\cC) \otimes \cV \to \cM$ in $\PrV$. Then, $\CIm(\PSh(\cC) \otimes \cV \to \cM)$ is the smallest full subcategory of $\cM$ which contains the image $F(\cC)$ and is closed under colimits and $\cV$-tensoring. 
\end{cor}
\begin{proof}
Note that $\CIm(\PSh(\cC) \otimes \cV \to \cM) \subseteq \cM$ contains the image $F(\cC)$ and is closed under colimits and $\cV$-tensoring. Conversely, if $\widetilde{\cM}$ is a full subcategory of $\cM$ which is closed under colimits and $\cV$-tensoring and contains the image of $F$, then $\widetilde{\cM} \hookrightarrow \cM$ is a morphism in $\PrV$ and $\PSh(\cC) \otimes \cV \to \cM$ factors through $\widetilde{\cM}$. Hence, $\CIm(\PSh(\cC) \otimes \cM \to \cV) \subseteq \widetilde{\cM}$. 
\end{proof}

\begin{cor}
	\label{prop:conservativegen}
	Let $\cC \in \cat$ be a small category, $\cM \in \PrV$ and $F: \cC \to \cM$ a functor. Then, the following are equivalent: 
	\begin{enumerate}[(1)]
		\item The full image $F(\cC) \subseteq \cM$ generates $\cM$ under colimits and tensoring;
		\item The induced functor $\PSh(\cC) \otimes \cV \to \cM$ in $\PrV$ is colimit-dominant;
		\item The functor $\iHom_{\mathcal{M}}(F -, -): \mathcal{M} \to \Fun(\cC^{\op}, \mathcal{V})$ is conservative;
		\item The family of functors $\{\iHom_{\cM}(F(c), -) : \mathcal{M} \to \mathcal{V}\}_{c \in \cC}$ is jointly conservative.
	\end{enumerate}
\end{cor}
\begin{proof} The equivalence between (1) and (2) follows directly from~\cref{cor:imagevsPsh}. Applying \cref{lem:colimit-dominant-spaces}, this is equivalent to the assertion that the right adjoint $\cM \to \PSh(\cC) \otimes \cV \simeq \Fun(\cC^{\op}, \cV)$ is conservative, i.e.\ to (3). Statement (3) is evidently equivalent to (4). 
\end{proof}
\begin{obs}
It follows from~\cref{prop:conservativegen} that a small set $\cJ$ of objects of $\cM$  generates $\cM$ under colimits and $\cV$-tensoring if and only if the functors $\{\iHom_{\cM}(x, -): \cM \to \cV\}_{x\in \cJ}$ are jointly conservative. 
\end{obs}

\begin{cor}
	\label{cor:colimit-dominant}
	The following conditions on a morphism $F:\cN \to \cM$ in $\PrV$ are equivalent:
	\begin{enumerate}[(1)]
		\item The full image $F(\cN) \subseteq \cM$ generates $\cM$ under colimits and tensoring.
		\item The full image $F(\cN)$ generates $\cM$ under colimits, i.e.\ $F$ is colimit-dominant.
		\item The right adjoint $F^\rR: \cM \to \cN$ is conservative.
	\end{enumerate}
\end{cor}
\begin{proof}
The first two conditions are equivalent by \cref{cor:CImV}, the third and the second are equivalent by \cref{lem:colimit-dominant-spaces}. \end{proof}

In fact, the factorization of \cref{cor:CImV} gives rise to a factorization system as defined in~\HTT{\S}{5.2.8} in $\PrV$, generalizing the corresponding factorization system in $\PrL$ from~\cite[Cor.\ 2.10]{monadictower}. \begin{prop}
	\label{lem:prfact}
	The classes of fully faithful functors and colimit-dominant functors  form a factorization system in $\PrV$. The factorization of a morphism $F: \cN \to \cM$ in $\PrV$ is given by $
	\cN \to \CIm(F) \to \cM \; .
	$
\end{prop}
\begin{proof}
	By \cref{cor:CImV}, any morphism in $\PrV$ factors through its closed image into a fully faithful functor followed by a colimit-dominant functor. It thus suffices to show that the classes of fully faithful and colimit-dominant functors are orthogonal and closed under retracts.
	We will therefore show that any diagram in $\PrV$
	\[
	\begin{tikzcd}
		\cN \arrow[d] \arrow[r] & \cN' \arrow[d, hook] \\ \cM \arrow[ur, dashed] \arrow[r] & \cM'
	\end{tikzcd}
	\]
	where the left vertical map is colimit-dominant and the right is fully faithful has  a contractible space of dashed lifts. Since $\cN' \to \cM'$ is fully faithful and hence in particular a monomorphism in $\widehat{\cat}$ and thus $\PrV$ the space of lifts is either empty or contractible, i.e.\ it suffices to show that $\cM \to \cM'$ factors through the full subcategory $\cN' \subseteq \cM'$. Let $\hat{\cM}$ denote the full subcategory of $\cM$ on those objects whose image under $\cM \to \cM'$ lands in the full subcategory $\cN' \subseteq \cM'$. 
	Since all morphisms are in $\PrV$, the full subcategory $\hat{\cM}$ is closed under colimits and tensoring. Moreover, by commutativity of the diagram, the full image of $\cN \to \cM$ is contained in $\hat{\cM}$. Thus, $\cM= \CIm(\cN \to \cM) \subseteq \hat{\cM}$ and hence $\hat{\cM} = \cM$ and therefore $\cM \to \cM'$ factors through $\cN'$.
	
	Given a section-retraction pair $r\circ s = \id$ in $\PrV$, then $r$ is surjective on objects and in particular colimit-dominant. It then follows from~\cref{obs:cancellationcolimdom} that colimit-dominant functors are closed under retracts in the arrow category of $\PrV$. That fully faithful functors are closed under retracts follows by passing to mapping spaces and noting that isomorphisms are closed under retracts in $\Spaces$.
\end{proof}
\begin{obs}
	\label{obs:cdomstab}
	By \HTT{Prop.}{5.2.8.6}, the class of colimit-dominant morphisms is therefore closed under retracts, pushouts along morphisms in $\PrV$ and colimits in $\Arr(\PrV):= \Fun([1], \PrV) $. By \HTT{Lem.}{5.2.8.19} the full inclusion $\Arr^{\mathrm{cdom}}(\PrV) \subseteq \Arr(\PrV)$ of the colimit-dominant functors admits a right adjoint sending $F: \cN \to \cM$ to its corestriction $F: \cN \to \CIm(F)$. 
\end{obs}

\begin{cor}\label{cor:cdomreltens} Let $\cV, \cW \in \Alg(\PrL)$ and  $\cB \in \BMod{\cV}{\cW}(\PrL)$. Given a colimit-dominant $F: \cM \to \cN$ in $\PrV$, then $F\otimes_{\cV} \cB : \cM \otimes_{\cV} \cB \to \cN \otimes_{\cV} \cB$ is colimit-dominant. 
\end{cor}
\begin{proof}
By expressing the relative tensor product via the bar construction as a colimit of tensors in $\PrL$, this follows from~\cref{lem:colimdomtensor} and the fact that colimit-dominant functors are closed under colimits in $\Arr(\PrV)$ from \cref{obs:cdomstab}. 
\end{proof}
\begin{cor}\label{cor:generatingbasechange} Let $\cV \to \cW$ be a morphism in $\Alg(\PrL)$, $\cM \in \RMod_{\cV}(\PrL)$, $C \in \cat$ a small category, $F: C\to \cM$ a functor and assume that the image $F(C)$ generates $\cM$ under colimits and $\cV$-tensoring. Then, the image of $C \to \cM \to \cM \otimes_{\cV} \cW$ generates $\cM \otimes_{\cV} \cW$ under colimits and $\cW$-tensoring. 
\end{cor}
\begin{proof} By \cref{prop:conservativegen} (2), equivalently the induced functor $\PSh(C) \otimes \cV \to \cM$ is colimit-dominant. By \cref{cor:cdomreltens}, this implies that also the functor $\PSh(C) \otimes \cW \simeq \PSh(C) \otimes \cV\otimes_{\cV} \cW \to \cM \otimes_{\cV} \cW$ is colimit-dominant. 
\end{proof}
	\subsection{Atomic objects}
	
	If $C$ is a small category, then any representable presheaf is a \emph{completely compact object} in $\PSh(C)$, i.e.\ an object $x\in \PSh(C)$ such that $\Map_{\PSh(C)}(x,-): \PSh(C) \to \Spaces$ preserves small colimits. This almost gives a characterization of representable presheaves internal to $\PSh(C)$: By~\HTT{Prop.}{5.1.6.8}, completely compact objects in $\PSh(C)$ are retracts of representable presheaves and thus the full subcategory of $\PSh(C)$ on the completely compact objects recovers the idempotent completion of $C$. We now study the generalization of these notions to enriched categories.

	\begin{defin}
		An object $m \in \mathcal{M}$ in a presentable $\cV$-module category $\cM \in \PrV$ is called $\cV$-\emph{atomic}\footnote{Atomic objects are sometimes also called `tiny', see e.g \cite{brandenburg2014reflexivity}.} if the functor $\iHom_{\mathcal{M}}(m, -): \mathcal{M} \to \mathcal{V}$
		preserves small colimits and preserves $\cV$-tensoring, in the sense that for all $v \in \cV$ the canonical map
		\begin{equation*}
			\iHom_{\mathcal{M}}(m, -) \otimes v \to \iHom_{\mathcal{M}}(m, - \otimes v)
		\end{equation*}
		adjoint to the evaluation $m \otimes \iHom_{\mathcal{M}}(m, -) \otimes v \to - \otimes v$ is an isomorphism. We denote by $\mathcal{M}^{\mathrm{at}}$ the full subcategory of $\cM$ on the atomic objects.
	\end{defin}
If clear from context, we will simply say atomic instead of $\cV$-atomic. 
	
	\begin{rem}
 		Any object $m \in \cM$ induces a morphism $m\otimes -: \cV \to \cM$ in $\Pr_{\cV} = \RMod_{\cV}(\PrL)$ whose right adjoint $\iHom_{\cM}(m,-): \cM \to \cV$ inherits by ~\HA{Example}{7.3.2.8} a lax $\cV$-linear structure. Atomicity of $m$ precisely asks this lax $\cV$-linear functor $\iHom_{\cM}(m,-)$ to be strongly $\cV$-linear and moreover colimit preserving, i.e.\ to itself be a morphism in $\PrV$. 
	\end{rem}

	
	\begin{lemma}
		\label{prop:atomicdual}
		For $\cV \in \Alg(\PrL)$ and $A\in \Alg(\cV)$, a left $A$-module $M$ is atomic in $\LMod_A(\cV) \in \RMod_{\cV}(\PrL)$ (see~\cref{ex:algebrastoPrV})  if and only if it is right dualizable, i.e.\ there exists a right module $M^{\vee} \in \RMod_A(\cV)$ such that the functor		\[
					M^\vee \otimes_A - : \LMod_A(\mathcal{V}) \to \mathcal{V}
		\]
		is right adjoint to $M \otimes -$.
	\end{lemma}
	\begin{proof}
		If $M$ is dualizable, it follows by definition that $M^{\vee} \otimes_A - \simeq \iHom_{\LMod_A(\cV)}(M,-)$, and hence preserves $\cV$-tensoring and colimits. Conversely, if $M$ is atomic, 
	\[ \iHom_{\LMod_A(\cV)} (M, N) \simeq \iHom_{\LMod_A(\cV)}(M, A \otimes_A N) \simeq \iHom_{\LMod_A(\cV)}(M, A) \otimes_A N \]
	where the last equivalence uses atomicity to commute $\iHom_{\LMod_A(\cV)}(M,-)$ with bar constructions. This exhibits $\iHom_{\LMod_A(\cV)}(M,-)$ as the right dual of $M$. 
	\end{proof}
	
	Important examples of presentably symmetric monoidal categories are given by \emph{modes} \HA{before Prop.}{4.8.2.11}, \cite[Def. 5.2.4]{carmeli2021ambidexterity}; by definition, these are unital idempotent objects in $\PrL$, i.e.\ presentable categories $\cE$ equipped with a left adjoint functor $\Spaces \to \cE$ (equivalently, a choice of object $1_{\cE} \in \cE$)  such that the induced functor $\cE \simeq \cE \otimes \Spaces \to \cE \otimes \cE$ is an equivalence. By \HA{Prop.}{4.8.2.9}, this induces a commutative algebra structure on $\cE$ with unit $1_{\cE}$  and for which the forgetful functor $\Mod_E(\PrL) \to \PrL$  is fully faithful and hence exhibits $\Mod_{\cE}(\PrL)$ as a reflective localization of $\PrL$ equivalent to the full image of the functor $\cE \otimes -: \PrL \to \PrL$, see~\HA{Prop.}{4.8.2.10}.

\begin{prop}[{\cite[Rem. 2.5]{redshift}}]\label{prop:mode} Let $\cE$ be a mode, i.e.\ a unital idempotent in $\PrL$ and $\cM\in \Pr_{\cE} \subseteq \PrL$. Then, $m \in \cM$ is atomic if and only if $\iHom_{\cM}(m,-): \cM \to \cE$ preserves colimits.  
\end{prop}	
\begin{proof}
By definition, $m$ is atomic if $\iHom_{\cM}(m,-): \cM \to \cE$ is a morphism in $\Pr_{\cE}$. But since $\Pr_{\cE}$ is a full subcategory of $\PrL$, it suffices that $\iHom_{\cM}(m,-): \cM \to \cE$ preserves colimits. 
\end{proof}
In other words, if $\cE$ is a mode, $m \in \cM \in \Pr_{\cE}$ and $\iHom_{\cM}(m,-)$ preserves colimits, it is automatically compatible with tensoring.

\begin{ex}
Using \cref{prop:mode},  we can collect many examples of atomic objects: 
\begin{itemize}
\item  $\cV=\Spaces$, the initial algebra in $\PrL$. Then, an $m \in \cM \in \PrL$ is atomic if and only if it is completely compact~ \HTT{Def.}{5.1.6.2}, i.e.\ $\Map_{\cM}(m,-): \cM \to \Spaces$ preserves colimits. 
\item Let $\cV= \Spaces_{\leq m}$ be the mode of $m$-truncated spaces, classifying by  \HA{Prop.}{4.8.2.15}  the category $\Pr_{\Spaces_{\leq m}}$ of presentable $(m+1, 1)$-categories. Since for $\cM \in \Pr_{\Spaces_{\leq m}}$, the inner hom is given by the mapping space, it follows that $m\in \cM$ is atomic if  $\Map_{\cM}(m,-): \cM \to \Spaces_{\leq m}$ preserves colimits. (Note that since $\Spaces_{\leq m} \to \Spaces$ does not preserve colimits, this condition is weaker than being completely compact.)
\item Let $\cV=\Spaces_*$ be the mode of pointed spaces, classifying by~\HA{Prop.}{4.8.2.11} the category $\Pr_{\mathrm{zero}}:= \Pr_{\Spaces_*}$ of pointed presentable categories (also known as categories with a zero object). Then, for $\cM \in \Pr_{\mathrm{zero}}$, an object $m \in \cM$ is atomic if and only if the mapping space functor $\Map_{\cM}(m,-): \cM \to \Spaces$ preserves weakly contractible colimits, i.e.\ colimits whose indexing category has contractible geometric realization. Indeed, as a right adjoint, $\Map_{\cM}(m,-): \cM \to \Spaces_*$ preserves the terminal and hence by pointedness the initial object. Since colimits are generated by initial objects and weakly contractible colimits\footnote{Given a colimit diagram $p: K \to C$, if $C$ admits an initial object $\emptyset$ we can always uniquely extend to a diagram $p': K^\triangleleft \to C$ sending the cone point to $\emptyset$. Its colimit agrees with that of $p$, but $K^\triangleleft$ is weakly contractible.}, it follows that $m$ is atomic if and only if $\Map_{\cM}(m,-): \cM \to \Spaces_*$ preserves weakly contractible colimits. Since $\Spaces_* \to \Spaces$ is conservative and preserves weakly contractible colimits by \cref{lem:weaklycontrslice}, the claimed statement follows. 
\item Let $\cV= \Sp$ be the mode of spectra, classifying by \HA{Prop.}{4.8.2.18}  the category $\Pr_{\mathrm{st}}:= \Pr_{\Sp}$ of stable presentable categories. Then, for $\cM \in \Pr_{\mathrm{st}}$ an object $m \in \cM$ is atomic if it is compact, see~\cite[Prop.\ 2.8]{redshift}.
\item Let $\cV=\Spcn$ be the mode of connective spectra, classifying by~\SAG{Cor.}{C.4.1.3} the category $\Pr_{\mathrm{add}}=\Pr_{\Spcn}$ of additive presentable categories. Then, for $\cM \in \Pr_{\mathrm{add}}$ an object $m \in \cM$ is atomic if it is \emph{compact projective}, i.e.\ if $\Map_{\cM}(m, -): \cM \to \Spaces$ preserves sifted colimits. Indeed, $\iHom_{\cM}(m,-): \cM \to \Spcn$ preserves finite products and hence by additivity finite coproducts. Thus, $m$ is atomic if and only if $\iHom_{\cM}(m,-): \cM \to \Spcn$ preserves sifted colimits which is equivalent to $\Map_{\cM}(m,-): \cM \to \Spaces$ preserving sifted colimits since $\Omega^{\infty}: \Spcn\to \Spaces$ preserves sifted colimits ~\HA{Prop.}{1.4.3.9}  and is conservative.  \end{itemize}
\end{ex}

	\begin{prop}[{\cite[Prop.\ 5.9]{ben2024naturality}}]
		\label{prop:atomicarecpt}
 		If $\kappa$ is a regular cardinal such that the unit $1_\mathcal{V} \in \mathcal{V}$ is $\kappa$-compact, then every atomic object in $\mathcal{M} \in \PrV$ is also $\kappa$-compact. In particular, the full subcategory $\mathcal{M}^{\mathrm{at}} \subseteq \mathcal{M}$ on the atomic objects is always small.
	\end{prop}
	\begin{proof}
		For $m \in \mathcal{M}^{\mathrm{at}}$, the functor
		\begin{equation*}
			\Map_{\mathcal{M}}(m , -) = \Map_{\mathcal{V}}(1_{\mathcal{V}}, \iHom_{\mathcal{M}}(m, -))
		\end{equation*}
		is the composition of two functors that preserve $\kappa$-filtered colimits, making $m$ a $\kappa$-compact object as well. So $\mathcal{M}^{\mathrm{at}} \subseteq \mathcal{M}^{\kcpt}$ the full subcategory of $\kappa$-compact objects, which is small by \HTT{Prop.}{5.4.2.2}.
	\end{proof}

	\begin{defin}
		If the full subcategory $\mathcal{M}^{\mathrm{at}} \subseteq \mathcal{M}$ generates $\cM$ under colimits and tensoring, we call $\mathcal{M}$ \emph{atomically generated}\footnote{Atomically generated categories are also called \emph{molecular}~\cite{ben2024naturality} and \emph{tinily-spanned}~\cite{brandenburg2014reflexivity}.}. We let $\PrVag$ denote the full subcategory of $\PrV$ on the atomically-generated categories.
	\end{defin}
	
	\begin{rem}\label{rem:atomicgenerator}
	By \cref{prop:atomicarecpt}, $\cM$ is atomically generated if and only if there exists a small set of atomic objects which generate $\cM$ under colimits and tensoring. We call such a set a \emph{set of atomic generators}. \end{rem}

	\subsection{Internally left adjoint functors}

	Recall that an object $m \in \cM$ is defined to be atomic if the functor $m \otimes - : \cV \to \cM$ in $\PrV$ admits a right adjoint $\iHom_\cM(m, -)$ in $\PrV$. This can immediately be generalized to other functors $F: \cN\to \cM$ in $\PrV$:
	
	\begin{defin}
		\label{defin:iL}
	A functor $F: \cN \to \cM$ in $\PrV$ is \emph{internally left adjoint} if its right adjoint $F^\rR: \cM \to \cN$ preserves colimits and if for $m\in \cM, v\in \cV$ the morphism 
	\[F^\rR(m) \otimes v \to F^\rR(m \otimes v) 
	\]
	adjoint to $F(F^\rR m \otimes v) \simeq FF^\rR(m) \otimes v \stackrel{\ev_F \otimes \id_v}{\longrightarrow} m \otimes v$ is an isomorphism. 
	
	In this case, we say that $F \dashv F^\rR$ is an \emph{adjunction internal to $\PrV$} and denote by $\Pr^{\mathrm{iL}}_{\mathcal{V}}$ the subcategory of $\PrV$ on all objects and the internally left adjoint functors.
	\end{defin}


	\begin{rem} A functor $F: \cN \to \cM$ in $\PrV$ has, as a functor, by definition a right adjoint $F^\rR: \cM \to \cN$ which carries by \HA{Cor.}{7.3.2.7} a lax $\cV$-linear structure. Therefore, an $F: \cN \to \cM$ in $\PrV$ is internally left adjoint precisely if this right adjoint $F^\rR$ preserves colimits (and hence admits a further right adjoint) and this lax linear structure is strong, i.e.\ if $F^\rR$ itself defines a morphism in $\PrV$. The terminology alludes to the fact that the internally left adjoints are precisely the adjunctions internal to the 2-category $\IPrV$, see  \cite[\S~4]{ben2024naturality} for a more thorough discussion.
		\end{rem}

	\begin{obs}
		\label{obs:iLandatomic}
		The unique morphism $\cV \to \cM$ in $\PrV$ sending $1_\cV \mapsto m$ (and therefore $v \mapsto m \otimes v$) is internally left adjoint if and only if $m$ is atomic.
	\end{obs}
	
	\begin{obs}
		\label{obs:iLpresatomic}
		Internally left adjoint functors in $\PrV$ are closed under composition. In particular, combined with \cref{obs:iLandatomic}, internally left adjoint functors preserve atomic objects; and we will next show a partial converse.
	\end{obs}

	\begin{prop}
		\label{prop:iLpreserve}
		Given a morphism $F: \cN \to \cM$ in $\PrV$ and assume that $\cN$ is atomically generated. Then, the following are equivalent: 
		\begin{enumerate}[(1)]
		\item $F$ is internally left adjoint;
		\item $F$ preserves atomic objects;
		\item If $\cN_0 \subseteq \cN^{\mathrm{at}}$ is a  full subcategory which generates $\cN$ under colimits and $\cV$-tensoring, then $F(\cN_0) \subseteq \cM^{\mathrm{at}}$. 
		\end{enumerate}
		\end{prop}
		\begin{proof}
		Internally left adjoint functors preserve atomic objects by \cref{obs:iLpresatomic}, thus (1) implies (2). Clearly, (2) implies (3). We now prove that (3) implies (1). Consider the full subcategory $\widetilde{\cN} \subseteq \cN$ on those objects $n \in \cN$ such that for all diagrams $I \to \cM$, and all objects $m \in \cM$ and $v\in \cV$ the morphisms
		\[ \iHom_{\cN} (n, \colim_i F^\rR m_i) \to  \iHom_{\cN}(n, F^\rR \colim_i m_i)  
		\]
		\[\iHom_{\cN}(n, F^\rR m \otimes v) \to \iHom_{\cN}(n, F^\rR (m \otimes v))
		\]
		are isomorphisms in $\cV$. Then, since $\cN_0 \subseteq \cN^{\mathrm{at}}$ and $F(\cN_0) \subseteq \cM^{\mathrm{at}}$ it follows that $\cN_0 \subseteq \widetilde{\cN}$. Moreover, $\widetilde{\cN}$ is closed under colimits and tensoring with objects in $\cV$. Thus, since $\cN_0$ is $\cV$-generating, it follows that $\widetilde{\cN} = \cN$. By Yoneda, it then follows that $F^\rR$ preserves colimits and tensoring and thus $F$ is internally left adjoint. 
		\end{proof}

		\begin{lemma}
		\label{lem:fflemma}
		A morphism $F: \cN \to \cM$ in $\PrV$ is fully faithful if and only if for all $m \in \cM$, the induced morphism 
		\begin{equation*}
			F: \iHom_{\mathcal{N}}(m, m') \overset{}{\to} \iHom_{\mathcal{M}} (F(m), F(m'))
		\end{equation*}
		 is an isomorphism in $\mathcal{V}$. 
		
		If $\cN$ is atomically generated by $\cN_0 \subseteq \cN$ and $F$ is internally left adjoint, then it suffices to verify this condition for $m, m' \in \cN_0$.
	\end{lemma}
	\begin{proof}
		The first statement follows immediately from the universal property of $\iHom$. 
		For the second statement, let $n_0 \in \cN_0$ and note that the full subcategory of those $n \in \cN$ for which $\iHom_{\cN}(n_0, n) \to \iHom_{\cM}(Fn_0, Fn)$ is an isomorphism contains $\cN_0$ by assumption and is closed under colimits and $\cV$-tensoring since $m_0$ and $Fm_0$ are atomics since $F$ preserves atomics by \cref{prop:iLpreserve}. Since $\cN$ is atomically generated by $\cN_0$, this full subcategory is all of $\cN$ and hence $\iHom_{\cN}(n_0, n) \to \iHom_{\cM}(Fn_0, Fn')$ is an isomorphism for all $n_0 \in \cN_0$ and $n\in \cN$. Now for a fixed $n\in \cN$ consider the full subcategory of those $\widetilde{n} \in \cN$ for which $\iHom_{\cN}(\widetilde{n}, n) \to \iHom_{\cM}(F\widetilde{n}, Fn)$ is an isomorphism. By the above, this contains $\cN_0$ and is closed under colimits and tensoring since $F\in \PrV$.
	\end{proof}

\begin{prop}\label{lem:iLunder}
Given a commuting diagram in $\PrV$
\[\begin{tikzcd}
&\arrow[dl, "\text{iL, cdom}"'] \arrow[dr, " \text{iL}"] \cX& \\
\cN \arrow[rr, "F"] && \cM
\end{tikzcd}
\]
where the morphisms are internal left adjoint and colimit-dominant as indicated. Then, $F$ is also internally left adjoint.
\end{prop}
\begin{proof}
Denote the colimit-dominant internal left adjoint by $G:\cX \to \cN$. Passing to right adjoints, $G^\rR$ is conservative by \cref{lem:colimit-dominant-spaces} and both $G^\rR$ and  $G^\rR \circ F^\rR$ preserve colimits and tensoring. Thus, also $F^\rR$ preserves colimits and tensoring.\end{proof}

\begin{cor}\label{cor:prilfact} The factorization system of colimit-dominant and fully faithful functors on $\PrV$ from \cref{lem:prfact} restricts to one on the subcategory $\PrViL \to \PrV$. 
\end{cor}
\begin{proof}
Given a square of internal left adjoint morphisms in $\PrV$ with left vertical morphism colimit-dominant and right vertical morphisms fully faithful, it follows from \cref{lem:iLunder} that the unique filler in $\PrV$ guaranteed by \cref{lem:prfact} is in fact in $\PrViL$. 

It therefore suffices to show that the morphisms in the factorization $\cA \to \CIm(F) \to \cB$ of an internal left adjoint $F: \cA \to \cB$ are again internal left adjoints. First, the right adjoint $\CIm(F) \to \cA$ is given by the composite $\CIm(F) \hookrightarrow \cB \overset{F^\rR}{\to} \cA$, where the inclusion $\CIm(F) \subseteq \cB$ is in $\PrV$, i.e.\ preserves colimits and tensoring. Thus, $\CIm(F) \to \cA$ preserves colimits and tensoring since $F^\rR$ does. 
Hence, $\cA \to \CIm(F)$ is internal left adjoint and colimit-dominant. It then follows from \cref{lem:iLunder} that also $\CIm(F) \to \cB$ is internal left adjoint. 
 \end{proof}
 \begin{obs}
	\label{obs:cdomiLstab}
	Exactly as in \cref{obs:cdomstab}, it follows that the class of colimit-dominant internal left adjoints is closed under retracts, pushout along morphisms in $\PrViL$ and colimits in $\Arr(\PrViL)$. Moreover, the full inclusion $\Arr^{\cdom}(\PrViL) \to \Arr(\PrViL)$ admits a right adjoint sending an internal left adjoint to its corestriction to its closed image. 
\end{obs}

\begin{lemma}\label{lem:colimdomiLgen}
Given an internal left adjoint and colimit-dominant morphism $F:\cN \to \cM$ in $\PrV$ and assume that $\cN$ is atomically generated. Then, $\cM$ is atomically generated by the image $F(\cN^{\mathrm{at}})$.
\end{lemma}
\begin{proof}
The full subcategory $F(\cN^{\mathrm{at}})$ consists of atomics since $F$ is internal left adjoint and generates $\cM$ under colimits and $\cV$-tensoring since $F$ is colimit-dominant.
\end{proof}

Internal left adjoints are precisely the adjoints internal to the $2$-category $\IPrV$. A way to express this without using $2$-categorical machinery is as follows.

\begin{defin} A commutative square of categories 
\[\begin{tikzcd}
A\arrow[r, "F"] \arrow[d, "H"] & B \arrow[d, "K"] \\
C \arrow[r, "G"] & D
\end{tikzcd}
\] 
is called \emph{horizontally right adjointable} if  $F$ and $G$ have right adjoints $F^\rR$ and $G^\rR$ and the induced morphism $ H \circ F^\rR \To G^\rR \circ K$  is an isomorphism and is called \emph{vertically left adjointable} if $H$ and $K$ have left adjoints $H^\rL$ and $K^\rL$ and the induced morphism $  K^\rL \circ G \To  F\circ H^\rL$ is an isomorphism.
\end{defin}

\begin{prop}\label{prop:adjointable}
A functor $F:\cN \to \cM$  in $\PrV$ is internally left adjoint if and only if either of the following equivalent conditions holds:
\begin{enumerate}[(1)]
\item For all $G: \cP \to \cQ$ in $\PrV$ the square of categories
\[\begin{tikzcd}
\LinFun(\cQ, \cN) \arrow[d, "-\circ G"] \arrow[r, "F\circ -"] & \LinFun(\cQ, \cM) \arrow[d, "-\circ G"] \\ \LinFun(\cP, \cN) \arrow[r, "F\circ-"] & \LinFun(\cP, \cM) 
\end{tikzcd}
\]
is horizontally right adjointable. 
\item For all $G: \cP \to \cQ$ in $\PrV$ the square of categories
\[\begin{tikzcd}
\LinFun(\cM, \cP) \arrow[d, "-\circ F"] \arrow[r, "G\circ -"] & \LinFun(\cM, \cQ) \arrow[d, "-\circ F"] \\ \LinFun(\cN, \cP) \arrow[r, "G\circ-"] & \LinFun(\cN, \cQ) 
\end{tikzcd}
\]
is vertically left adjointable. 
\end{enumerate}
\end{prop}
\begin{proof}
\newcommand{\LaxFun}{\Fun^{\mathrm{lax}}_{\cV}}
We prove the equivalence with (1), the equivalence with (2) is proven completely analogously:
For $\cR, \cS \in \PrV$, let $\LaxFun(\cR, \cS)$ denote the category of lax $\cV$-linear functors\footnote{If we regard $\cR^\otimes, \cS^\otimes \to \RM^\otimes$ as coCartesian fibrations of operads, $\LaxFun(\cR, \cS)$ may be defined as the full subcategory of $\Fun_{/\RM^\otimes}(\cR^\otimes, \cS^\otimes) \times_{\Fun_{/\Ass^\otimes}(\cV^\otimes, \cV^\otimes)} \{ \id_{\cV^\otimes} \}$ on those functors that preserve coCartesian lifts of inert morphisms; while (strong) $\cV$-linear functors preserve all coCartesian lifts.} so that $\LinFun(\cR, \cS) \subseteq \LaxFun(\cR, \cS)$ is a full subcategory. 
Since $F: \cN \to \cM$ is in $\PrV$ and hence in $\PrL$, it has a right adjoint $F^\rR$ which by \HA{Ex.}{7.3.2.8} is lax $\cV$-linear. Thus, for any $G$ as above, the square
\[\begin{tikzcd}
\LaxFun(\cQ, \cN) \arrow[d, "-\circ G"] \arrow[r, "F\circ -"] & \LaxFun(\cQ, \cM) \arrow[d, "-\circ G"] \\ \LaxFun(\cP, \cN) \arrow[r, "F\circ-"] & \LaxFun(\cP, \cM) 
\end{tikzcd}
\]
is horizontally right adjointable with horizontal right adjoints given by postcomposing with $F^\rR$. If $F$ is internally left adjoint, then by definition for all  $\cR$ in $\PrV$ the above horizontal right adjoints $F^\rR \circ - : \LaxFun(\cR, \cM) \to \LaxFun(\cR, \cN)$ factor through the full subcategories $\LinFun(\cR, \cM) \to \LinFun(\cR, \cN)$ and therefore make the original square in (1) right adjointable. 

For the converse, it suffices to assume that the square 
\[\begin{tikzcd}
\LinFun(\cM, \cN) \arrow[d, "-\circ F"] \arrow[r, "F\circ -"] & \LinFun(\cM, \cM) \arrow[d, "-\circ F"] \\ \LinFun(\cN, \cN) \arrow[r, "F\circ-"] & \LinFun(\cN, \cM) 
\end{tikzcd}
\]
is horizontally right adjointable. The right adjoint of $\LinFun(\cM, \cN) \to \LinFun(\cM, \cM)$ applied to $\id_{\cM}$ defines a  $H \in \LinFun(\cM, \cN)$ together with a morphism $\eta: F\circ H \To \id_{\cM}$ in $\LinFun(\cM, \cM)$. Using that the square is horizontally right adjointable, it follows that (the underlying natural transformation of) $\eta$ exhibits the functor $H$ as a right adjoint to $F$. Since $H$ is moreover in $\PrV$ and $\eta$ is a morphism in $\LinFun(\cM,\cM)$, it follows that $F$ is internally left adjoint with internal right adjoint $H$. 
\end{proof}

\begin{cor}\label{cor:iLreltens}
\label{lem:propertiestensorbetter}
		Let $\cV, \cW \in \Alg(\PrL)$ and $F: \cN \to \cM$ be an internal left adjoint in $\PrV$. 
		\begin{enumerate}[(1)]
		\item If $G: \cN' \to \cM'$ is an internal left adjoint in $\PrW$, then $F\otimes G$ is an internal left adjoint in $\Pr_{\cV \otimes \cW}$ with right adjoint  $F^\rR \otimes G^\rR$.
		\item If $\cB \in \BMod{\cV}{\cW}( \PrL)$, then $F\otimes_{\cV} B: \cN \otimes_{\cV} \cB \to \cM \otimes_{\cV} \cB$ is an internal left adjoint in $\PrW$ with right adjoint $F^\rR \otimes_{\cV} \cB$. 
		\end{enumerate}
		\end{cor}
		\begin{proof} Since internal left adjoints compose, the first statement follows from the statement for $G=\id_{\cN'}$ which is a special case of the second statement. The second statement follows from the characterization of internal left adjoints from  \cref{prop:adjointable}(2) together with the universal property of the relative tensor product from \cref{prop:reltensoradj}.
		\end{proof}

This may be used to show that atomics are preserved under basechange:
\begin{cor}\label{cor:atomicbasechange}
Let $\cV \to \cW$ a morphism in $\Alg(\PrL)$ and $m \in \cM \in \PrV$ a $\cV$-atomic object. Then, $m \otimes 1_{\cW} \in \cM \otimes_{\cV} \cW \in \Pr_{\cW}$ is $\cW$-atomic.
\end{cor}
\begin{proof}
This follows immediately from~\cref{cor:iLreltens} since $\cV$-atomics in $\cM \in \PrV$ are precisely internal left adjoints $\cV \to \cM$. \end{proof}

\begin{cor}\label{cor:internalhombasechange}
For $f: \cV \to \cW$ a morphism in $\Alg(\PrL)$, $\cM \in \PrV$ and $m, n \in \cM$ with $m\in \cM$ $\cV$-atomic, the induced morphism
\[  f(\iHom_{\cM}(m,n)) \to \iHom_{\cM \otimes_{\cV} \cW}(m\otimes 1_{\cW}, n \otimes 1_{\cW} ) 
\]is an isomorphism in $\cW$. 
\end{cor}
\begin{proof}By definition, the functor $\iHom_{\cM \otimes_{\cV} \cW} (m \otimes 1_{\cW}, -) : \cM \otimes_{\cV} \cW \to \cW$ is right adjoint to the functor $m \otimes - : \cW \to \cM \otimes_{\cV} \cW$. By atomicity of $m \in \cM$ and \cref{cor:iLreltens} this functor is internal left adjoint and its right adjoint is given by $\iHom_{\cM}(m, -) \otimes_{\cV} \cW : \cM \otimes_{\cV}\cW \to \cV \otimes_{\cV} \cW \simeq \cW$ which after restriction along $\cM \to \cM \otimes_{\cV}\cW$ is $f(\iHom_{\cM}(m,-)): \cM \to \cW$.
\end{proof}

We will also use the following fact from \cite{ramzi2024dualizable}:
\begin{prop}[{\cite[Corollarly 1.32]{ramzi2024dualizable}}]
		\label{lem:iLclosedcolim}
		The subcategory $\PrViL \hookrightarrow \PrV$ is closed under colimits.
	\end{prop}

	\begin{obs}
		\label{obs:iLstability}
		In particular, the pushout of an internally left adjoint functor in $\PrV$ along an internal left adjoint is still internally left adjoint, since we obtain a pushout square in $\PrViL$. Also $\Arr(\PrViL) \subseteq \Arr(\PrV)$ is closed under colimits as those are calculated pointwise.
	\end{obs}

	\section{Monadicity internal to \texorpdfstring{$\PrV$}{presentable module categories}}
	\label{sec:monadicity}

	A \emph{monad on an object $c$} in a $2$-category $\C$ is an algebra $T \in \Alg(\operatorname{End}_{\C}(c))$ in the monoidal category of endomorphisms of $c$. In particular, a valent $\cV$-category as in \cref{def:valentVcat} is precisely a monad in the $2$-category $\IPrV$. Informally, an \emph{Eilenberg-Moore object} for $T$ is an object $\LMod_{T}(c) \in \C$ equipped with an adjunction $c \rightleftarrows \LMod_{T}(c)$ internal to $\C$ which induces for any $c' \in \C$ an equivalence 
		\[
	\LMod_T \left( \Hom_{\C} (c', c) \right) \simeq \Hom_{\C}(c', \LMod_T (c) ) \, ,
	\]
	i.e.\ it represents the functor $\LMod_T \Hom_{\C}(-,c)$ thereby internalizing the construction of module categories into $\C$. Any adjunction in $\C$ arising this way is called \emph{monadic adjunction}, and any right adjoint in a monadic adjunction is called \emph{monadic $1$-morphism} in $\C$. We refer to \cite{heinemonads} or \cite[\S 7]{stefanich} for a general treatment; here we will only develop the theory of monads and monadic adjunctions working in the closed monoidal $1$-category $\PrV$.

	\subsection{The Eilenberg-Moore functor}

	\begin{constr} \label{cons:EM}Let $\cQ, \cV \in \Alg(\PrL)$ and $\cP \in \BMod{\cQ}{\cV}(\PrL) \simeq \LMod_{\cQ}(\PrV)$. We now construct a functor 
	\[	
	\EM: \Alg(\cQ) \to (\PrV)_{\cP /}
	\] 
	sending an algebra $T \in \Alg(\cQ)$ to the `free-module' functor $(\cP \to \LMod_T(\cP))$. 
	
	Indeed, the forgetful functor $\LMod(\cP) \to \Alg(\cQ)$ is a Cartesian and coCartesian fibration, with Cartesian transport given by restriction and coCartesian transport by extension of scalars; see \HA{Cor.}{4.2.3.2} and \HA{Lemma}{4.8.3.15} respectively. Regarded as a coCartesian fibration, it classifies a functor $\LMod_{(-)}(\cP): \Alg(\cQ) \to \widehat{\cat}$ sending $T \mapsto \LMod_T(\cP)$. It factors through the subcategory $\PrL$ since extension of scalars preserves colimits and $\LMod_T(\cP)$ is presentable.
	
	We now lift this functor $\LMod_{-}(\cP): \Alg(\cQ) \to \PrL$ through $\PrV$:	Combining \HA{Prop.}{4.3.2.5} and \HA{Prop.}{4.3.2.6}, we learn that taking sections on $\LM^\otimes$ of the composition $\cP^\otimes \to \BM^\otimes \to \LM^\otimes \times \RM^\otimes$ exhibiting $\cP$ as a left $\cQ$- and right $\cV$-module, we obtain a right $\cV$-module $\LMod(\cP)_{\cV}$, or equivalently a coCartesian fibration $\LMod(\cP)^{\otimes}\to \RM^{\otimes}$ whose pullback along $\Ass^{\otimes} \to \RM^{\otimes} $ is $\cV^{\otimes}$. As in \HA{Prop.}{4.8.3.22} we can parametrize by the algebras in $\cQ$, yielding a coCartesian fibration $\LMod(\cP)^\otimes \to \Alg(\cQ) \times \RM^\otimes$ which unstraightens to a functor $\LMod_{-}(\cP): \Alg(\cQ) \to \RMod_{\cV}(\widehat{\cat})$. Since the forgetful functor $\LMod_T(\cP) \to \cP$ creates colimits, this functor factors through $\PrV$.

	Finally, $\Alg(\cQ)$ has an initial object $1_{\cQ}$ which the functor $\Alg(\cQ) \to \PrV$ sends to $\LMod_{1_{\cQ}}(\cP) = \cP$, so we obtain an induced functor $\Alg(\cQ) \simeq \Alg(\cQ)_{1_{\cQ}/} \to (\PrV)_{\cP/}$.
	\end{constr}

The functor $\EM$ from~\cref{cons:EM} in fact lands in the full subcategory of $(\PrV)_{\cP/}$ on the internal left adjoints and colimit-dominant functors:
	\begin{prop}
		\label{prop:EMisiL}
		For any $\cQ, \cV \in \Alg(\PrL), \cP \in \BMod{\cQ}{\cV}(\PrL)$ and $T\in \Alg(\cQ)$, the free-module functor $\cP \to \LMod_T(\cP)$ is internal left adjoint in $\PrV$ and colimit-dominant. 
	\end{prop}
	\begin{proof}
	The right adjoint is the forgetful functor $\LMod_T(\cP) \to \cP$ which is conservative, hence by~\cref{lem:colimit-dominant-spaces} $\cP \to \LMod_T(\cP)$ is colimit-dominant. Moreover, the forgetful functor $\LMod_T(\cP) \to \cP$ preserves colimits by  \HA{Cor.}{4.2.3.5} and its lax $\cV$-module structure is strong by the definition of the right $\cV$-action on $\LMod_T(\cP)$. 
	\end{proof}

	\begin{obs}\label{obs:EMmonad}
	For $\cV \in \Alg(\PrL)$ and $\cP \in \PrV$, recall from \cref{prop:LinFuniHom} that $\LinEnd(\cP) \in \Alg(\PrL)$ is the internal end of $\cP$ for the left $\PrL$ action on $\PrV$. By \cref{reminder:endo}, this enhances $\cP$ to an object of $\LMod_{\LinEnd(\cP)}(\PrV) \simeq \BMod{\LinEnd(\cP)}{\cV}(\PrL)$. 
	We may thus apply \cref{cons:EM} to it and obtain a functor \[
	\EM: \Alg(\LinEnd(\cP)) \to (\PrV)_{\cP/}~.
	\]
\end{obs}
	\begin{defin}
		Let $\cV \in \Alg(\PrL)$ and $\cP \in \PrV$. An \emph{internal monad on $\cP$ in $\PrV$} is a $T \in \Alg(\LinEnd(\cP))$, its \emph{Eilenberg-Moore object} is $\LMod_T(\cP) \in \PrV$ and its \emph{Eilenberg-Moore adjunction} is the free-forget adjunction $\cP \rightleftarrows \LMod_T(\cP)$.\end{defin}

		Given a monad, the functor from \cref{obs:EMmonad} therefore assigns its Eilenberg-Moore adjunction, which by \cref{prop:EMisiL} is an internal adjunction in $\PrV$. Conversely, given any such internal adjunction we may recover a monad as follows:
	\begin{lemma}[{cf.~\HA{Lem.}{4.7.3.1}}]
		\label{obs:adjendo}
		Given an internal adjunction $F: \mathcal{P} \rightleftarrows \mathcal{M} :F^\rR$ in $\PrV$, the transformation
		\[ F \circ (F^\rR \circ F) \simeq (F\circ F^\rR) \circ F \to F
		\]
		exhibits $F^\rR \circ F\in \LinEnd(\cP)$ as the internal endomorphism object of $F\in \LinFun(\cP, \cM)$ with respect to the right $\LinEnd(\cP)$-action on $\LinFun(\cP, \cM)$ from \cref{prop:bimoduleihom}. In particular, $F^\rR \circ F$ inherits the structure of a monad on $\cP$ in $\PrV$ by \cref{reminder:endo}.
		\end{lemma}
		\begin{proof}The proof is entirely analogous to \HA{Lem.}{4.7.3.1}: For each $H \in \LinEnd(\cP)$ the composition
		\[ \Map_{\LinEnd(\cP)}(H, F^\rR \circ F) \to \Map_{\LinFun(\cP, \cM)}(F \circ H , F \circ F^\rR \circ F ) \to \Map_{\LinFun(\cP, \cM)}(F \circ H , F) \]
		is an equivalence since $F, F^\rR$ are adjoint.
\end{proof}
\begin{rem}
An analogous argument shows that $F^\rR \circ F$ is also isomorphic to the internal endomorphism object of $F^\rR \in \LinFun(\cM, \cP)$ with respect to the left $\LinEnd(\cP)$ action on $\LinFun(\cM, \cP)$. \end{rem}

	Using these observations, we now show that the functor from \cref{obs:EMmonad} is in fact fully faithful and admits a right adjoint.
	
	\begin{theorem}
		\label{thm:monadicff}		Let $\cV \in \Alg(\PrL)$ and $\cP \in \PrV$. The functor \[\EM: \Alg(\LinEnd(\cP)) \to (\PrV)_{\cP /}\]from \cref{obs:EMmonad} is fully faithful and admits a right adjoint which sends $(F:\cP \to \cM) \in (\PrV)_{\cP/}$ to the endomorphism object $\iEnd_{\LinFun(\cP, \cM)}(F) \in \LinEnd(\cP)$ in the presentable right $\LinEnd(\cP)$-module category $\LinFun(\cP, \cM)$ (with action induced by composition as defined in \cref{prop:bimoduleihom}).
	\end{theorem}
	\begin{proof}
		We first show that taking internal ends defines a right adjoint for $\EM$. Given $T \in \Alg(\LinEnd(\cP))$ and $(\cP \to \cM) \in (\PrV)_{\cP/}$, the category of $\cV$-linear functors over $\cP$ fits into the pullback diagram
		\[
		\begin{tikzcd}
			\Fun^{\rL}_{\cV, \cP/}\left(\LMod_T(\cP), \cM \right) \arrow[r] \arrow[d] \arrow[dr, phantom, "\scalebox{1}{$\lrcorner$}" , very near start, color=black] & \LinFun\left(\LMod_T(\cP), \cM \right) \arrow[d] \\
			\{ \cP \to \cM \} \arrow[r] & \LinFun(\cP , \cM)
		\end{tikzcd}
		\]
		where the right vertical map is induced by precomposition with the free functor $\cP \to \LMod_T(\cP)$. But we can rewrite the upper right entry as
		\begin{align*}
			&\LinFun\left(\LMod_T(\cP), \cM \right) \simeq \LinFun(\LMod_T(\LinEnd(\cP)) \otimes_{\LinEnd(\cP)} \cP, \cM) \simeq \\ &\quad \simeq \Fun^{\rL}_{\LinEnd(\cP)} \left(\LMod_T(\LinEnd(\cP)), \LinFun(\cP, \cM) \right) \simeq \RMod_T (\LinFun(\cP, \cM))
		\end{align*}
		using \HA{Thm.}{4.8.4.6}, \cref{prop:reltensoradj} and \HA{Thm.}{4.8.4.1}. Now recall from \cref{reminder:endo} that algebra homomorphisms that go from $T$ into the endomorphism object $\iEnd_{\LinFun(\cP, \cM)}(\cP \to \cM) \in \LinEnd(\cP)$ agree with the very same pullback \[\RMod_T(\LinFun(\cP, \cM)) \times_{\LinFun(\cP, \cM)} \{\cP \to \cM\}.\] Thus, the functor $\Alg(\LinEnd(\cP)) \to (\PrV)_{\cP/}$ has a pointwise right adjoint\footnote{A \emph{pointwise right adjoint of a functor $F: \cC \to \cD$ at an object $d\in \cD$} is an object in $\cC$ which represents the functor $\Map_{\cD}(F- , d): \cC^{\op} \to \Spaces$. A functor $F$ has a right adjoint if and only if it has a pointwise right adjoint at every object $d\in \cD$; the right adjoint in this case is given by the induced factorization $\cD \to \PSh(\cC)$ through the full subcategory $\cC$.   } and  therefore a right adjoint. 
		
		Now, it suffices to show that the unit $T \to \iEnd_{\LinFun(\cP, \cM)}(\cP \to \LMod_T(\cP)) $ is an isomorphism for any algebra $T$. In other words, the canonical $T$-action by precomposition exhibits $T$ as the endomorphism object of $\cP \to \LMod_T(\cP)$. This is precisely the statement of \cref{obs:adjendo} using \cref{prop:EMisiL}  that $\cP \to \LMod_T(\cP)$ is an internal left adjoint.
	\end{proof}
	
	\begin{rem} For $\cP = \cV$ this recovers \HA{Thm.}{4.8.5.5}. 	\end{rem}


	\subsection{A monadicity theorem for presentable module categories}
	
	We now characterize the image of the fully faithful functor $\Alg(\LinEnd(\cP)) \to (\PrV)_{\cP/}$ thereby proving a variant of the Barr-Beck-Lurie monadicity theorem for presentable module categories.

Given an internal adjunction $F: \cP \rightleftarrows \cM: F^\rR$ in $\PrV$, the composition $F^\rR \circ F$ is by~\cref{obs:adjendo} the endomorphism object of $F$ and thus inherits the structure of a monad. Moreover, the counit of the adjunction from~\cref{thm:monadicff} defines a morphism $\LMod_{F^\rR \circ F}(\cP) \to \cM$ in $(\PrV)_{\cP/}$.
\begin{defin}\label{defin:internalmonadic} An internal adjunction $F: \cP \rightleftarrows \cM: F^\rR$ in $\PrV$ is called \emph{internally monadic} if this induced functor $\LMod_{F^\rR \circ F} (\cP) \to \cM$ is an equivalence. 
\end{defin}

	\begin{obs}
		\label{obs:counitcomparison}
		By \cref{thm:monadicff}, given an $F: \cP \to \cM$ in $\PrV$ the counit \[\LMod_{\iEnd_{\LinFun(\cP, \cM)}(F)}(\cP) \to \cM\] is an equivalence if and only if $F$ is in the image of $\EM: \Alg(\LinEnd(\cP)) \to (\PrV)_{\cP/}$. Thus, the image consists precisely of those internal left adjoints $F: \cP \to \cM$ which are part of an internal monadic adjunction. 
	\end{obs}


The following proposition is a presentable version of~\HA{Cor.}{4.7.3.16} and will be key to our monadicity theorem. 
	\begin{prop} \label{prop:keymonadicity}Given a diagram in $\PrL$ 
	\[\begin{tikzcd}& \cA \arrow[dl, "F"'] \arrow[dr, "F'"] \\ \cB \arrow[rr, "H"] && \cC
	\end{tikzcd}
	\]
	for which $F, F'$ are internal left adjoints and $F$ is colimit-dominant, and assume that the induced transformation \[F^\rR F \To F^\rR H^\rR H F  \simeq (F')^\rR F'\] is an isomorphism.   Then, $H$ is internal left adjoint and fully faithful. 
	If $F'$ is also colimit-dominant, then $H$ is an equivalence. 
	\end{prop}
	\begin{proof}
By ~\cref{lem:colimit-dominant-spaces}, $F^\rR$ is conservative and hence $F \To H^\rR H F \simeq H^\rR F'$ is an equivalence.	By~\cref{lem:iLunder}, $H$ is an internal left adjoint and hence $H^\rR$ and thus $H^\rR H$ preserve colimits. Since $F$ is colimit-dominant, this implies that the counit $\id \To H^\rR H$ is also an equivalence and hence that $H$ is fully faithful. 
	If $F'$ is also colimit-dominant, it follows from the uniqueness of factorization in the colimit-dominant/fully faithful factorization system on $\PrL$ from ~\cref{lem:prfact} that $H$ is an equivalence. 
	\end{proof}

	\begin{theorem}\label{thm:monadicity}
	Let $\cV \in \Alg(\PrL)$ and $\cP \in \PrV$. 	An object $(F:\cP \to \cM)$ in $(\PrV)_{\cP/}$ is in the image of the fully faithful functor $\EM: \Alg(\LinEnd(\cP)) \hookrightarrow (\PrV)_{\cP/}$ from \cref{thm:monadicff} if and only if $F$ is internally left adjoint in $\PrV$ and one of the following equivalent conditions holds:
	\begin{enumerate}[(1)]
	\item The internal adjunction $ F \dashv F^\rR$ is internally monadic in $\PrV$ as in \cref{defin:internalmonadic};
	\item The adjunction $F\dashv F^\rR$ is monadic in the sense of Barr-Beck-Lurie~\HA{Def.}{4.7.3.4}; 
	\item $F^\rR$ is conservative;
	\item $F$ is colimit-dominant.
	\end{enumerate}
	\end{theorem}
	\begin{proof}
	By \cref{obs:counitcomparison}, $F$ is in the image of $\EM$ if and only if it is internally left adjoint and (1) holds. We will now show that (1) is equivalent to (2), (3), and (4). 
	By the Barr-Beck-Lurie monadicity theorem \HA{Thm.}{4.7.3.5}, (2) is equivalent to (3). By \cref{cor:colimit-dominant}, (3) is equivalent to (4). Moreover, if $F\dashv F^\rR$ is internally monadic in $\PrV$, then $F: \cP \to \cM$ is equivalent to the composite $ \cP \to \LMod_T(\cP) \simeq \cM$ and hence is colimit-dominant, proving that (1) implies (4). We will now show that (4) implies (1). 
	By construction, the morphism $\LMod_{F^\rR \circ F} (\cP) \to \cM$ is in $(\PrV)_{\cP/}$; hence we have a commuting diagram in $\PrV$ and thus in $\PrL$:
		\[\begin{tikzcd}
		& \cP \arrow[dl] \arrow[dr, "F"] \\ 
		\LMod_{F^\rR \circ F} (\cP)  \arrow[rr]&& \cM 
	\end{tikzcd}\; .
	\]
	By assumption $F$ is internally left adjoint and colimit-dominant, and so is the other diagonal functor by \cref{prop:EMisiL}. Moreover, the composite of the free- and forget functor $\cP \to \LMod_{F^\rR F} (\cP) \to \cP$ is equivalent to the underlying functor of the monad which by \cref{obs:adjendo} is  $F^\rR \circ F$. Thus, the conditions of \cref{prop:keymonadicity} are satisfied and $\LMod_{F^\rR \circ F} \to \cM$ is an equivalence, hence $F \dashv F^\rR$ is internally monadic in the sense of \cref{defin:internalmonadic}.	\end{proof}
\begin{rem}
The equivalences $(1) \Leftrightarrow (3) \Leftrightarrow(4)$ in our \cref{thm:monadicity} do not depend on Lurie's monadicity theorem \HA{Thm.}{4.7.3.5}. Instead our proof uses \cref{prop:keymonadicity} as a simpler presentable variant of \HA{Cor.}{4.7.3.16}.
\end{rem}

	\section{Marked modules}
	\label{sec:markednew}

	We will now use our monadicity theorem to relate enriched categories with marked modules and prove our main result \cref{thm:charessim}.
	
	We start with the following preparatory statements:
	\begin{lemma}
		\label{lem:PShVXtg}
		Let $\cV \in \Alg(\PrL)$ and $X\in \Spaces$. Then, the objects in the image of $X\to \PSh(X) \simeq \PSh(X) \otimes \Spaces \to \PSh(X) \otimes \cV \in \PrV$ are $\cV$-atomic. Moreover, they generate $\PSh(X) \otimes \cV$ under colimits and $\cV$-tensoring. In particular, $\PSh(X) \otimes \cV$ is atomically generated. 
	\end{lemma}
	\begin{proof}
	Recall that the representable presheaves in $\PSh(X)$ are $\Spaces$-atomic, i.e.\ completely compact. Thus, by~\cref{cor:atomicbasechange}, the objects in the image of $X \to \PSh(X) \otimes \cV$ are $\cV$-atomic. Moreover, by \cref{cor:imagevsPsh} the full image of $X\to \PSh(X) \otimes \cV$ generates under colimits and $\cV$-tensoring.
	\end{proof}

		\begin{cor}\label{cor:iLfactoring}
		Let $\cV \in \Alg(\PrL), \cM \in \PrV$ and $C \in \cat$ a small category. Then, for a functor $\iota: C \to \cM$ the following are equivalent: 
		\begin{enumerate}[(1)]
		\item $\iota$ factors through the full subcategory $\cM^{\mathrm{at}}$;
		\item The induced functor $\PSh(C) \otimes \cV \to \cM$ in $\PrV$ is internally left adjoint.
		\end{enumerate}
		\end{cor}
		\begin{proof}
		Since the full image of $C \to \PSh(C) \otimes \cV$ consists of atomics and generates under colimits and $\cV$-tensoring by~\cref{lem:PShVXtg}, the statement follows from~\cref{prop:iLpreserve}(3).
		\end{proof}

	\begin{notat}\label{not:PrVX} Let $C \in \cat$ be a small category. Then, we write $(\PrV)_{C/}$ for the pullback $\widehat{\cat}_{C/} \times_{\widehat{\cat}} \PrV$, where $\PrV \to \widehat{\cat}$ is the forgetful functor.  Informally, an object of $(\PrV)_{C/}$ is an $\cM \in \PrV$ equipped with a functor $C \to \cM$.
	\end{notat}
	
	\begin{lemma}\label{obs:PrVX}
	Let $\cV \in \Alg(\PrL)$, $C \in \cat$, then restricting along $C \to \PSh(C) \otimes \cV$ induces an equivalence  \[ (\PrV)_{\cP(C) \otimes \cV/}  \to (\PrV)_{C/}.
	\]
	\end{lemma}
	\begin{proof}
	Let $\freecoc: \widehat{\cat} \rightleftarrows \Catcolim$ denote the free cocompletion under small colimits which by \HTT{Cor.}{5.3.6.10} is left adjoint to the subcategory inclusion $\Catcolim \to \widehat{\cat}$. For \emph{small} categories $C$, $\freecoc(C) \simeq \Fun(C^{\op}, \Spaces) = \PSh(C)$ as explained in \HTT{Ex.}{5.3.6.6}. Composing with the free-forgetful adjunction we obtain an adjunction  $\widehat{\cat} \rightleftarrows \RMod_{\cV}(\Catcolim): \mathrm{Forget}$\ and thus for large categories $D$ by \cref{cor:sliceadjoints} an equivalence 
	\[ \RMod_{\cV}(\Catcolim)_{D/} \simeq \RMod_{\cV}(\Catcolim)_{\freecoc(D) \otimes \cV/}.
	\]
	For small categories $C$ this restricts to the desired equivalence between full subcategories. 
	\end{proof}

	We are now ready to state and prove the first main theorem of this paper:
		\begin{theorem}
		\label{thm:charessim}
		Let $\cV \in \Alg(\PrL)$ and $X\in \Spaces$. The composite
		\[\vcatXV := \Alg(\LinEnd(\PSh(X) \otimes \cV)) \stackrel{\EM}{\longrightarrow} (\PrV)_{\PSh(X) \otimes \cV/}  \overset{\text{\cref{obs:PrVX}}}{\simeq}(\PrV)_{X/}
		\]
		sends a valent $\cV$-category $\cC$ to $(\yo {}^{\cV}_{\cC}: X \to \PShV(\cC)) \in (\PrV)_{X/}$, i.e.\ its enriched presheaf category marked by the representable presheaves. This functor is fully faithful, has a right adjoint and its image consists of those functors $\yoV: X \to \cM$ for which either of the following equivalent conditions holds:
		\begin{enumerate}[(1)]
		\item The image of $\yoV$ is contained in the atomics and generates $\cM$ under colimits and $\cV$-tensoring; 
		\item The image of $\yoV$ is contained in the atomics and the functors $\{\iHom_{\cM}(\iota x, -): \cM \to \cV\}_{x\in X}$ are jointly conservative;
		\item The induced functor $\PSh(X) \otimes \cV \to \cM$ is an internal left adjoint in $\PrV$ which is colimit-dominant (equivalently its right adjoint is conservative).
		\end{enumerate}
		\end{theorem}
		\begin{proof}
		By definition, the composite sends a monad $\cC \in \Alg(\LinEnd(\PSh(X) \otimes \cV))$ to the free module functor $\PSh(X) \otimes \cV \to \LMod_\cC(\PSh(X) \otimes \cV)$ in $\PrV$. By \cref{def:presheaf} and \cref{prop:presentablepresheaf}, this is precisely $\PSh(X) \otimes \cV \to \PShV(\cC)$ corresponding under the equivalence $(\PrV)_{\PSh(X) \otimes \cV/} \simeq (\PrV)_{X/}$ to $\yoV: X \to \PShV(\cC)$. 
		Fully faithfulness follows from \cref{thm:monadicff} and the characterization of the image from~\cref{thm:monadicity} and Corollaries \ref{cor:iLfactoring} and \ref{prop:conservativegen}.
		\end{proof}
		
		\begin{defin} \label{def:ModXV}We refer to a $\cM \in \PrV$ together with a small space $X \in \Spaces$ and a functor $X\to \cM$ satisfying the equivalent conditions of~\cref{thm:charessim} as a \emph{marked module}. We let $\MModXV \subseteq (\PrV)_{X/}$ denote the full subcategory on the marked modules. 
		\end{defin}
		
		\begin{defin} \label{def:PShVfunctor}
			Using \cref{cons:EM} we define the functor \[\PShV := \LMod_{(-)}(\PSh(X) \otimes \cV): \Alg(\LinEnd(\PSh(X) \otimes \cV)) \to \PrV\] sending a $\cV$-category $\cC$ to its enriched presheaf category $\PShV(\cC)$ (see \cref{def:presheaf}). Under the equivalence from \cref{thm:charessim}, it corresponds to the slice projection $\MModXV \subseteq (\PrV)_{X/} \to \PrV$.
		\end{defin}

	\begin{obs} \label{cor:PShVtg}An immediate consequence of~\cref{thm:charessim} is that for any $\cC \in \vcatXV$ the presheaf category $\PShV(\cC)$ is atomically generated by the image of the Yoneda embedding $\yoV:X \to \PShV(\cC)$. 
	\end{obs}

		\begin{rem}
		In~\cite[Prop.\ 6.3.1]{hinich} given an $\cM \in \PrV$, a small space $X$ and a functor $X\to \cM$, Hinich constructs a $\cV$-enriched category as the internal end of the induced functor $(\PSh(X) \otimes \cV \to \cM) \in \LinFun(\PSh(X)\otimes \cV, \cM)$ with respect to its right $\LinEnd(\PSh(X) \otimes \cV)$-action. By~\cref{thm:monadicity} this precisely agrees with the right adjoint of the fully faithful comparison functor $\vcatXV \to (\PrV)_{X/}$ from~\cref{thm:charessim}.
		\end{rem}

		\begin{rem}
		\label{rem:tgarepshv}
		By~\cref{thm:charessim}, every atomically generated $\cM \in \PrV$ is the enriched presheaf category of some valent $\cV$-category $\cC$. There is in fact a distinguished choice for such $\cC$, namely that corresponding to the marked module $\cM^{\text{at,}\simeq} \to \cM$ where $\cM^{\text{at,}\simeq}$ is the small (by \cref{prop:atomicarecpt}) subspace of $\cM$ spanned by the atomic objects. This satisfies condition $(1)$ by definition. Enriched categories $\cC$ arising this way are called \emph{Cauchy-complete}.
	\end{rem}

	\begin{cor} \label{cor:fullyfaithfulatomic}
	If $\cV \to \cW$ is a fully faithful morphism in $\Alg(\PrL)$ and $\cM \in \PrV$ is $\cV$-atomically generated, then $\cM \simeq \cM \otimes_{\cV} \cV  \to \cM \otimes_{\cV} \cW$ is fully faithful. 
	\end{cor}
	\begin{proof} By \cref{rem:tgarepshv}, there is an $X\in \Spaces$ and a $\cC \in \vcatXV = \Alg(\LinEnd(\PSh(X) \otimes \cV))$ so that $\cM \simeq \PShV(\cC) := \LMod_{\cC}(\PSh(X) \otimes \cV)$. Since $ \Fun(X^{\op}, \cV) \simeq \PSh(X) \otimes \cV \to \PSh(X) \otimes  \cV\otimes_{\cV} \cW \simeq \PSh(X) \otimes \cW \simeq \Fun(X^{\op}, \cW) $ is given by postcomposition and hence is fully faithful, it immediately follows that also \[\LMod_{\cC}(\PSh(X) \otimes \cV) \to \LMod_{\cC}(\PSh(X) \otimes \cV) \otimes_{\cW} \cW \simeq \LMod_{\cC}(\PSh(X) \otimes \cW)\] is fully faithful, the last equivalence using {{\HA{Thm.}{4.8.4.6}}}.	\end{proof}
	
	\begin{warning}
	The condition that $\cM$ is atomically generated is necessary in \cref{cor:fullyfaithfulatomic}. For a counterexample, take $\cV \to \cW$ the full inclusion $\Spcn \to \Sp$ of connective spectra into spectra and let $\cM= \Ab$ be the category of abelian groups with $\Spcn$-action induced by the morphism $\pi_0 : \Spcn \to \Ab$  in $\CAlg(\PrL)$. Then, \[ \Ab \to {\Ab} \otimes_{\Spcn} {\Sp} \simeq {\Ab} \otimes {\Sp} \simeq *\]
	is not fully faithful. Compare \cite{haine2021nonabelian}.	\end{warning}

As an immediate corollary of \cref{thm:charessim} we obtain the following variant of \HA{Prop.}{4.8.5.8}:
\begin{cor}
Let $\cV \in \Alg(\PrL)$. Then, the functor 
\[\Alg(\cV) \to (\PrV)_{*/}
\]
sending an algebra $A$ to the category $\LMod_A(\cV)$ pointed by the free $A$-module $A$ is fully faithful. Its image consists precisely of those $\cM \in \PrV$ with pointing $m \in \cM$ for which $m$ is an atomic generator.
\end{cor}
\begin{proof}
Apply \cref{thm:charessim} in the case $X=*$. 
\end{proof}

	\section{Functoriality of marked modules}
	\label{sec:functoriality}
	In our first main \cref{thm:charessim}, we constructed  for a fixed space $X$ and $\cV \in \Alg(\PrL)$ an equivalence between the category $\vcatXV:= \Alg(\LinEnd(\PSh(X) \otimes \cV))$ of valent $\cV$-enriched categories with space of objects $X$ and the full subcategory $\MModXV \subseteq (\PrV)_{X/}$  of marked modules, i.e.\ those functors $X \to \cM$ whose image consists of atomic objects that generates under colimits and $\cV$-tensoring. 
	
	One advantage of the latter perspective on enriched categories is that the functoriality of the full subcategory $\MModXV$ in $X$ and $\cV$ is much easier to establish than the functoriality of $\Alg(\LinEnd(\PSh(X) \otimes \cV))$ defined in terms of endomorphism algebras (cf.~\cite[Prop. 4.5.5]{hinich}). In  \cite{haugseng} this is remedied by explicitly constructing $\vcatXV$ as algebras over an operad in an evidently functorial way, however adding combinatorial complications.
	In this section, we explain how to make this full subcategory $\MModXV \subseteq (\PrV)_{X/} $ functorial in $X$ and $\cV$ and establish a number of basic properties. In~\S \ref{subsec:heinecomparison}, we will show that this functoriality indeed agrees with the one arising from \cite{haugseng}, making use of tools developed in \cite{heine}.

	\subsection{Functoriality in spaces}

	\begin{defin}
		\label{defin:MModV}
		Let $\Arr(\widehat{\cat}) \times_{\widehat{\cat}} \PrV$ denote the pullback of the target projection $\Arr(\widehat{\cat}) \to \widehat{\cat}$ along the forgetful functor $\PrV \to \widehat{\cat}$. 
		The category of \emph{marked presentable $\mathcal{V}$-module categories}, or in short \emph{marked $\cV$-modules} is the full subcategory		\[\MModV \subseteq \Arr(\widehat{\cat}) \times_{\widehat{\cat}} \PrV
		\]
		on those functors $ X \to \cM$ for which $X$ is a small space and the image of $X \to \cM$ consists of atomics which generate $\cM$ under colimits and $\cV$-tensoring. By definition, the functor $\MModV \to \Arr(\widehat{\cat}) \stackrel{\mathrm{src}}{\to} \widehat{\cat}$ factors through the full subcategory $\Spaces \subseteq \widehat{\cat}$ of small spaces. We refer to the induced functor $ \ob: \MModV \to \Spaces$ as the \emph{space of objects} of a marked $\cV$-module. 
		\end{defin}
		
		By definition, the fiber of $\ob:\MModV \to \Spaces$ at a space $X$ agrees with the category $\MModXV$ from \cref{def:ModXV} which by \cref{thm:charessim} is equivalent to $\vcatXV$.

	\begin{ex}
		\label{defin:enrfunctor}
		Given $\cV$-categories $\cC \in \vcatXV, \cD \in \vcat_Y(\cV)$ identified via \cref{thm:charessim} with their Yoneda functors  $(X \to \PShV(\cC)) \in \MModXV$ and $(Y \to \PShV(\cD))\in \MModYV$, a morphism between them in $\MModV$ amounts to a commuting diagram 
		\[\begin{tikzcd}X \arrow[r] \arrow[d, dashed] & \PShV(\cC) \arrow[d, dashed] \\ Y \arrow[r] & \PShV(\cD)
		\end{tikzcd}
		\]
		where the right vertical morphism is in $\PrV$. In particular, there is a map of underlying spaces of objects $f: X\to Y$ and for every $x,y\in X$ by \cref{prop:Yoneda} and \cref{obs:adjunctioninthom} a morphism in $\cV$: \[ \Hom_{\cC}(x,x') \simeq \iHom_{\PShV \cC} (\yo {}_{\cC}^\cV(x), \yo {}_{\cC}^{\cV}(y)) \to \iHom_{\PShV \cD}(\yo {}^{\cV}_{\cC}(fx), \yo {}^{\cV}_{\cD}(fy)) \simeq \Hom_{\cD}(fx, fy)\]In fact, we will show in~\cref{thm:functorialcomparison} that the category $\MModV$ together with its functor $\MModV \to \Spaces$ is equivalent to the category $\vcatV$ of valent $\cV$-enriched categories and $\cV$-enriched functors in the sense of \cite{haugseng}, \cite{heine}, \cite{hinich}.
		\end{ex}
		In the remainder, it will be convenient to express $\MModV$ in terms of the free $\cV$-module $\PSh(X) \otimes \cV$:
		\begin{notat}
		Let $\Arr^{\text{iL,cdom}}(\PrV)$ denote the full subcategory of $\Arr(\PrV) $ on the internal left adjoint and colimit-dominant functors and $\Arr^{\mathrm{cdom}}(\PrViL)$ the full subcategory of $\Arr(\PrViL)$ on the colimit-dominant functors; compare~\cref{defin:colimitdominant} and~\cref{defin:iL}.
		\end{notat}
		
		\begin{prop} \label{prop:MModVequiv}
		The category $\MModV$ is the pullback 
				\[
			\MModV\simeq  \Spaces \times_{\PrV} \Arr^{\text{iL,cdom}} \left( \PrV \right)
		\]
		of the source projection $\Arr^{\text{iL,cdom}} \left( \PrV \right) \overset{\mathrm{src}}{\to} \PrV$ along the functor $\Spaces \to \PrV$ sending $X$ to the free presentable $\cV$-module category $\PSh(X) \otimes \cV$ on $X$.
		
		Moreover, the factorization of $\Spaces \to \PrV$ through the subcategory inclusion $ \PrViL \to \PrV$ induces an equivalence
		\[ \Spaces \times_{\PrViL} \Arr^{\mathrm{cdom}} \left( \PrViL \right) \overset{\simeq}{\to} \Spaces \times_{\PrV} \Arr^{\text{iL,cdom}} \left( \PrV \right) \simeq \MModV \, , \]
		\end{prop}
		\begin{proof}
		Composing the free cocompletion under small colimits adjunction $\widehat{\cat} \rightleftarrows \Catcolim$ from \HTT{Cor.}{5.3.6.10} with the free-forgetful adjunction $\Catcolim \rightleftarrows \RMod_{\cV}(\Catcolim)$ results in an adjunction $\widehat{\cat} \rightleftarrows \RMod_{\cV}(\Catcolim)$. \cref{lem:adjunctionunstr} then yields an equivalence 
		\[\widehat{\cat} \times_{\RMod_{\cV}(\Catcolim)} \Arr(\RMod_{\cV}(\Catcolim)) \simeq \Arr(\widehat{\cat}) \times_{\widehat{\cat}} \RMod_{\cV}(\Catcolim)\; .
		\]
		Restricting to the relevant full subcategories proves the first statement. 
		The second statement follows from \cref{lem:iLunder}.
		\end{proof}
		
		\begin{cor}\label{cor:MModVclosed} The full subcategory $\MModV \subseteq \Spaces\times_{\PrV} \Arr(\PrV)$ is closed under colimits. 
		\end{cor}
		\begin{proof} By \cref{cor:prilfact} colimit-dominant and fully faithful functors form a factorization system on $\PrViL$, so it follows as in \cref{obs:cdomiLstab} that the full inclusion $\Arr^{\cdom}(\PrViL) \hookrightarrow \Arr(\PrViL)$ has a right adjoint and hence so has the full inclusion 
		\[ \MModV \simeq \Spaces\times_{\PrViL} \Arr^{\cdom}(\PrViL) \hookrightarrow \Spaces \times_{\PrViL} \Arr(\PrViL)\,.
		\]
	 Using \cref{lem:iLclosedcolim}  that the subcategory $\PrViL \to \PrV$ is closed under colimits, so is the induced subcategory inclusion $\Spaces \times_{\PrViL} \Arr(\PrViL) \to \Spaces\times_{\PrV}\Arr(\PrV)$. 
		\end{proof}

	We will now prove that $\ob: \MModV \to \Spaces$ is a Cartesian and coCartesian fibration. This will directly follow from the following:

		\begin{cor}\label{cor:coCartMModV}The functor $\ob: \MModV \to \Spaces$ is a Cartesian and coCartesian fibration, classifying a functor 
		\[\MModbV : \Spaces \to \Pr \; .
		\] 	\end{cor}
	\begin{proof}
	Since $\PrViL$ has colimits  (preserved by the inclusion into $\PrV$, see  \cref{obs:iLstability}) and since colimit-dominant and fully faithful functors form a factorization system on $\PrViL$  by \cref{cor:prilfact}, it follows from \cref{prop:ArrLCartesian} that $\Arr^{\cdom}(\PrViL) \to \PrViL$ is a Cartesian and coCartesian fibration.
	By \cref{prop:MModVequiv}, $\MModV \to \Spaces$ is therefore a pullback of a (co)Cartesian fibration and hence itself one.
	Since the fibers $\MModXV \simeq \vcatXV$ are presentable by \cref{obs:pCatXpres}, it follows that the classified functor lands in $\Pr$.
	\end{proof}
	
		\begin{ex}	\label{ex:carttransportgraph}
	Unpacked, given a map of spaces $g:X \to Y$, the induced right adjoint $g^! : \MModYV \to \MModXV$ sends a marked module $Y\to \cM$ to the marked module $X\to \CIm(X \to Y \to \cM)$. 
	In particular, using the equivalence $\vcatXV \simeq \MModXV$ from \cref{thm:charessim} and \cref{cor:PShVgraph},  the induced functor $\vcatYV \to \vcatXV$ sends a valent $\cV$-enriched category $\cD$ with space of objects $Y$ to the valent $\cV$-enriched category $g^! \cD$ with space of objects $X$ and $\Hom_{g^!\cD}(x , x') = \Hom_{\cD}(g(x), g(x'))$.
	\end{ex}

	\begin{obs}\label{obs:coCartMModVDesc}Using the description of Cartesian and coCartesian morphisms from \cref{prop:ArrLCartesian}, the proof of~\cref{cor:coCartMModV} implies that a morphism in $\MModV$ consisting of a map $g: X \to Y$ in $\Spaces$ and a commutative square in $\PrV$
		\[
		\begin{tikzcd}
			\PSh(X) \otimes \mathcal{V} \arrow[r] \arrow[d] & \PSh(Y) \otimes \mathcal{V} \arrow[d] \\
			\cM \arrow[r] & \cN
		\end{tikzcd}
		\]
		\begin{itemize}
			\item is Cartesian iff the functor $\cM \to \cN$ is fully faithful,
			\item is coCartesian iff this square is a pushout square in $\PrV$ (or equivalently, $\PrL$).
		\end{itemize}
	\end{obs}

	\begin{cor}
		\label{prop:CatVcolimSpace} 
		The category $\MModV$ has limits and colimits, and the functor $\ob: \MModV \to \Spaces$ preserves them.
	\end{cor}
	\begin{proof}
	By \cref{cor:coCartMModV}, $\ob$ is a Cartesian and coCartesian fibration classifying a functor $\Spaces \to \Pr$. In particular, the transport maps preserve limits and colimits in the fibers of $\ob$ respectively, so our claim follows from \HTT{Cor.}{4.3.1.11} and \HTT{Prop.}{4.3.1.5} since the basis $\Spaces$ admits limits and colimits.  	\end{proof}
	\begin{rem}
	By the previous proof, colimits in $\MModV$ are calculated by first taking them in $\Spaces$ obtaining a space $X$, then performing coCartesian transports along the colimit cone to push the diagram over $X$, and finally calculating the colimit of the transports in the fiber $\vcatXV$ which admits limits and colimits as it is presentable by \cref{obs:pCatXpres}. Dually for limits and Cartesian transport.
\end{rem}


\subsection{Functoriality in the enriching category}

We now  construct a coCartesian fibration $\MMod \to \Alg(\PrL)$ with fibers $\MModV$, exhibiting the latter as functorial in $\mathcal{V}$. 

In \cref{defin:MModV} for a fixed $\cV \in \Alg(\PrL)$, we defined $\MModV$ as the full subcategory of $\Arr(\widehat{\cat}) \times_{\widehat{\cat}} \PrV$ on those functors $C \to \cM$ where $C$ is a small space and the functor atomically generates.

Recall the coCartesian fibration $p: \RMod(\PrL) \to \Alg(\PrL)$ which by  \HA{Lemma}{4.8.3.15}  straightens to the functor $ \RMod_{-}(\PrL)=: \Pr_{-}: \Alg(\PrL) \to \widehat{\cat}$ sending algebra homomorphisms to the induced extension-of-scalars functors.

\begin{defin} \label{defin:MMod} We define $\MMod$ as the full subcategory of $\Arr(\widehat{\cat})\times_{\widehat{\cat}} \RMod(\PrL)$ on those triples of $\cV \in \Alg(\PrL), \cM \in \PrV$ and functor $C\to \cM$ for which $C$ is a small space and whose image lands amongst the $\cV$-atomics and generates $\cM$ under colimits and tensoring. 
\end{defin}

\begin{lemma}
	\label{lem:arrlcattwosided}
	The functor $\Arr(\lcat) \times_{\lcat} \RMod(\Pr) \overset{\mathrm{src} \times p}{\longrightarrow} {\lcat \times \Alg(\Pr)}$ is a two-sided fibration (c.f.\ \cref{reminder:twosidedfib}), classifying a functor $\lcat \times \Alg(\Pr) \to \llcat, (C, \cV) \mapsto (\PrV)_{C/}$.
\end{lemma}
\begin{proof}
	By \cref{ex:arrowtwosided} the arrow category $(\mathrm{src}, \mathrm{trg}): \Arr(\lcat) \to \lcat \times \lcat$ is a two-sided fibration, so by \cref{obs:twosidedfib} the pullback $\Arr(\lcat) \times_{\lcat} \RMod(\Pr) \to \lcat \times \RMod(\Pr)$ is one as well. Postcomposing with the coCartesian fibration $p: \RMod(\Pr) \to \Alg(\Pr)$ proves the claim.
	\end{proof}
	
	\begin{notat}\label{notat:intPrVC}
	For better readability, we will sometimes follow \cref{reminder:twosidedfib} and write $\int^C_{\cV} (\PrV)_{C/} \to \lcat \times \Alg(\PrL)$ for the two-sided fibration from \cref{lem:arrlcattwosided}. We write $\int^X_{\cV} (\PrV)_{X/} \to \Spaces \times \Alg(\PrL)$ for its pullback along the inclusion $\Spaces \hookrightarrow \lcat$ (classifying the restriction of the functor $(C,\cV) \mapsto (\PrV)_{C/}$ to $\Spaces \hookrightarrow \lcat$). 
	\end{notat}

\begin{prop}\label{prop:coCart} The composite $\MMod \hookrightarrow \Arr(\widehat{\cat}) \times_{\widehat{\cat}} \RMod(\PrL) \to \RMod(\PrL) \to \Alg(\PrL)$ is a coCartesian fibration and the inclusion $\MMod \to \Arr(\widehat{\cat}) \times_{\widehat{\cat}} \RMod(\PrL)$ preserves coCartesian morphisms. 
\end{prop}
\begin{proof} By \cref{lem:arrlcattwosided} the functor $\Arr(\widehat{\cat})\times_{\widehat{\cat}} \RMod(\PrL) \to \Alg(\PrL)$ is a coCartesian fibration. Explicitly, given a morphism $\cV \to \cW$ in $\Alg(\PrL)$, the induced coCartesian transport sends a $C\to \cM_{\cV} \in \Arr(\widehat{\cat}) \times_{\widehat{\cat}} \RMod(\PrL)$ to $C \to \cM \otimes_{\cV} \cW_{\cW}$. In particular, if $C\to \cM_{\cV}$ is a marked module for $\cV$, i.e.\ if its image lands in the $\cV$-atomics and generates under colimits and $\cV$-tensoring, then also $C\to \cM \otimes_{\cV} \cW$ is a marked module for $\cW$ by \cref{cor:atomicbasechange} and \cref{cor:generatingbasechange}. Thus, the full subcategory $\MMod \subseteq \Arr(\widehat{\cat}) \times_{\widehat{\cat}} \RMod(\PrL)$ is closed under coCartesian transport and hence the composite $\MMod \to \Alg(\PrL)$ is also a coCartesian fibration by \cref{lem:coCartrefl}.
\end{proof}

\begin{ex}\label{ex:basechangerelative}
It follows directly from the proof of~\cref{prop:coCart} that the induced straightened functor $\MModb: \Alg(\PrL) \to \widehat{\cat}$ sends a morphism $f: \cV \to \cW$ in $\Alg(\PrL)$ to the functor $\MModV \to \MModW$ sending a marked $\cV$-module $X \to \cM$ to the marked $\cW$-module $X\to  \cM \to \cM \otimes_{\cV} \cW$. 
\end{ex}

\begin{obs}\label{obs:MModsubfunctor}
Let $\Arr(\lcat) \times_{\lcat} \Pr_{-}: \Alg(\PrL) \to \llcat$ denote the straightening of the coCartesian fibration $\Arr(\lcat) \times_{\lcat} \RMod(\PrL) \to \Alg(\PrL)$. Then, \cref{prop:coCart} is equivalent to the assertion  that $\MMod \to \Alg(\PrL)$ is the unstraightening of a full subfunctor $\MModb : \Alg(\PrL) \to \lcat$ given at a $\cV \in \Alg(\PrL)$ by the full subcategory $\MModV \subseteq \Arr(\lcat) \times_{\lcat} \PrV$ from \cref{defin:MModV}.
\end{obs}

We now show that for a morphism $f: \cV \to \cW$ the induced change-of-enrichment $\MModV \to \MModW$ sending $(X\to \cM) \mapsto (X\to \cM \otimes_{\cV} \cW)$ acts as expected on enriched categories. 

\begin{constr}	\label{obs:EMfunct}

Let ${}_{g} F_f : {}_{\cQ} \cP_\cV \to {}_{\cQ'} \cP'_{\cV'}$ be a morphism in $\operatorname{BMod}(\PrL)$ from a $\cQ$-$\cV$-bimodule $\cP$ to a $\cQ'$-$\cV'$-bimodule $\cP'$. Applying \cref{cons:EM}, this induces a map of coCartesian fibrations between $\LMod(\cP)^{\otimes} \to \Alg(\cQ) \times \RM^{\otimes}$ and $\LMod(\cP')^{\otimes} \to \Alg(\cQ') \times \RM^{\otimes}$ and hence  a natural transformation \[ 
\begin{tikzcd}[column sep=large]
\Alg(\cQ) \arrow[d, "\Alg(g)"']  \arrow[r, "\LMod_{-}(\cP)"]& \arrow[dl, Rightarrow] \RMod_{\cV}(\PrL) \arrow[d, "\otimes_{\cV} \cV'"] \\
\Alg(\cQ') \arrow[r, "\LMod_{-}(\cP')"'] & \RMod_{\cV'}(\PrL).
\end{tikzcd} \, .
\]
\end{constr}

\begin{prop}\label{prop:changeofenragree}
For $X\in \Spaces$ and $f: \cV \to \cW$ in $\Alg(\PrL)$, the induced coCartesian transport $\MModXV \to \MModXW$, restricted to the fibers over a fixed space $X$ of objects,  agrees under the equivalence $\MModXV \simeq \vcatXV$ from \cref{thm:charessim} with the functor from \cref{constr:changeofenrquiv}. 
\end{prop}
\begin{proof}
	The functor $f$ induces a map $\PSh(X) \otimes \cV \to \PSh(X) \otimes \cV \otimes_{\cV} \cW \simeq \PSh(X) \otimes \cW$ in $\RMod(\PrL)$ by coCartesian transport. Applying \cref{obs:endofunctoriality} to this morphism and the left $\PrL$-action on $\RMod(\PrL)$ we obtain a map ${}_{\LinEnd(\PSh(X) \otimes \cV} (\PSh(X) \otimes \cV)_{\cV} \to {}_{\LinEndW(\PSh(X) \otimes \cW} (\PSh(X) \otimes \cW)_{\cW}$ in $\LMod(\RMod(\PrL)) = \operatorname{BMod}(\PrL)$ whose underlying morphism between endomorphism algebras agrees with the one constructed in \cref{constr:changeofenrquiv}. Applying \cref{obs:EMfunct}, we obtain a natural transformation
\[ 
\begin{tikzcd}[column sep=large]
\vcatXV = \Alg(\LinEnd(\PSh(X) \otimes \cV)) \arrow[d]  \arrow[r, "\PShV"]& \arrow[dl, Rightarrow] \PrV \arrow[d, " -\otimes_{\cV} \cW"] \\
\vcatXW=\Alg(\LinEndW(\PSh(X) \otimes \cW)) \arrow[r, "\PShW"'] & \PrW
\end{tikzcd}
\]
where the left vertical morphism is the functor $f_!$ from \cref{constr:changeofenrquiv}. This induces a natural transformation 
\[\begin{tikzcd}
\vcatXV  \arrow[d, "f_!"']  \arrow[r, "\simeq"]& \arrow[dl, Rightarrow] \MModXV \arrow[d, " -\otimes_{\cV} \cW"] \\
\vcatXW \arrow[r, hook, "\simeq"'] & \MModXW
\end{tikzcd} \, .
\] We will conclude by proving that this natural transformation is invertible. Let therefore $\cC \in \vcatXV$ and since $\MModXV \simeq \Alg(\LinEnd(\PSh(X) \otimes \cV)) \to\Fun(X \times X, \cV)$ is conservative, it suffices to show invertibility on the underlying graphs where $f_!$ is simply given by postcomposition. 
But for $x, y \in X$, the $\cW$-morphism $\iHom_{\PShV(\cC) \otimes_{\cV} \cW} (\yoV_{\cC}(x) \otimes 1_\cW, \yoV_{\cC}(y) \otimes 1_{\cW}) \to f(\iHom_{\PShV(\cC)}(\yoV_{\cC} (x), \yoV_{\cC}(y))$ is by \cref{cor:internalhombasechange} an isomorphism since $\yoV_{\cC}(x)$ is atomic, proving the claim. 
\end{proof}

	\begin{cor}
		\label{cor:coerelativeadjoint}
		For a morphism $f:\cV \to \cW$ in $\Alg(\PrL)$, the functor $f_! : \MModV \to \MModW$ admits a right adjoint $f^\rR_!$ whose unit and counit are sent to equivalences\footnote{In the terminology of \cite[\HAsubsec{7.3.2}]{HA},  $f_! \dashv f_!^\rR$ is a \emph{relative adjunction} over $\Spaces$, cf.~\cref{reminder:relativeadjoints}.}  under the respective functors $\ob: \MModV \to \Spaces$ and $ \ob: \MModW \to \Spaces$. 	\end{cor}
	\begin{proof}
	Both categories are coCartesian fibrations over $\Spaces$ by \cref{cor:coCartMModV}. Hence,  by \HA{Prop.}{7.3.2.6} it suffices to show that for each $X\in \Spaces$, the induced map on fibers $ \MModXV \to \MModXW$ admits a right adjoint, and that $f_!: \MModV \to \MModW$ preserves coCartesian morphisms. The former follows from \cref{prop:changeofenragree} and \cref{constr:changeofenrquiv}, the latter follows from the characterization of coCartesian morphisms in \cref{obs:coCartMModVDesc} and the fact that $-\otimes_{\cV} \cW$ preserves colimits. 
	\end{proof}

	\begin{cor}
		The coCartesian fibration $\MMod \to \Alg(\PrL)$ from \cref{prop:coCart} is also Cartesian. In particular, $\MMod$ admits limits and colimits and $\MMod \to \Alg(\PrL)$ preserves them. 	
	\end{cor}
	\begin{proof}
	It follows from~\cref{cor:coerelativeadjoint} that the coCartesian transports admit right adjoints and hence that $\MMod \to \Alg(\PrL)$ is also a Cartesian fibration by \HTT{Cor.}{5.2.2.5}. 
	Since the base $\Alg(\PrL)$ has limits and colimits, the fibers $\MModV$ have limits and colimits by \cref{prop:CatVcolimSpace} and the Cartesian transports preserve limits  while the coCartesian transports preserve colimits, it follows from  \HTT{Cor.}{4.3.1.11} that $\MMod$ has limits and colimits which are preserved by $\MMod \to \Alg(\PrL)$.
		\end{proof}

	\begin{lemma}
		\label{prop:changeofenrprespre}
		If $f: \cV \to \cW$ is a morphism in $\Alg(\PrL)$, the coCartesian transport $f_! : \MModV \to \MModW$ is a map of Cartesian fibrations over $\Spaces$.
	\end{lemma}
	\begin{proof}
		By \cref{obs:coCartMModVDesc} a morphism $G: (\iota_{\cM}: X \to \cM) \to (\iota_{\cN}: Y \to \cN)$ in $\MModV$ is Cartesian if and only if $\cM \to \cN$ is fully faithful. Since $\cM$ is atomically generated by $X$ and $\cM \to \cN$ is internal left adjoint by \cref{prop:MModVequiv}, this is by  \cref{lem:fflemma}  equivalent to the condition that for all $x,x' \in X$ the induced map on internal homs 
		\[ \iHom_{\cM}(\iota_{\cM} (x), \iota_{\cM}(x') ) \to \iHom_{\cN}(\iota_{\cN}G(x),\iota_{\cN} G (x') )\]
		is an isomorphism in  $\cV$. Applying $f:\cV \to \cW$ results in an isomorphism in $\cW$ and it therefore follows from \cref{cor:internalhombasechange} that so  is the map 
			\[\iHom_{ \cM \otimes_{\cV} \cW}(\iota_{\cM} x \otimes 1_{\cW} , \iota_{\cM}x' \otimes 1_{\cW})   \to \iHom_{\cN \otimes_{\cV}\cW}(\iota_{\cN} Gx \otimes 1_{\cW}, \iota_{\cN} Gx' \otimes 1_{\cW}) \, .\qedhere
		\]
			\end{proof}

	\begin{cor}
		\label{cor:vEnrtwosided}
		The projection $\MMod \to \Spaces \times \Alg(\PrL)$ is a two-sided fibration (in the sense of \cref{reminder:twosidedfib}) and hence straightens to a functor $\MModbb: \Spaces^{\op} \times \Alg(\PrL) \to \widehat{\cat}$.
	\end{cor}
	\begin{proof}
	By \cref{prop:coCart}, $\MMod \to \Alg(\PrL)$ is a coCartesian fibration with coCartesian morphisms sent to equivalences under $\MMod \to \Spaces$. Moreover, the fibers $\MModV \to \Spaces$ are Cartesian fibrations by \cref{cor:coCartMModV} and Cartesian morphisms are preserved by coCartesian transport $f_!$ due to \cref{prop:changeofenrprespre}.
	\end{proof}
	
	\begin{warning} The full inclusion $\MMod \to \Arr(\lcat)\times_{\lcat} \RMod(\PrL)$ is a map over $\Spaces \times \Alg(\PrL)$ which is a map of coCartesian fibrations over $\Alg(\PrL)$, but is not a map of two-sided fibrations. In particular, it does not straighten to a natural transformation $\MModqb \To (\Pr_{-})_{?/}$ (but merely a lax natural transformation). Indeed, by \cref{ex:carttransportgraph} Cartesian transport along non-surjective maps of spaces in $\MMod$ changes not just the underlying space but also the presentable module category. 
	\end{warning}

Analogous to~\cref{prop:MModVequiv}, we may rewrite~\cref{defin:MMod} in terms of functors out of $\PSh(X) \otimes \cV$:

\begin{notat}\label{not:restriLcdom} We let $\int_{\cV} \Arr(\PrV) \to \Alg(\PrL)$ denote the coCartesian unstraightening of the functor $\Arr \circ \RMod_{-}(\PrL) : \Alg(\PrL) \to \widehat{\cat}$. Since internal left adjoints and colimit-dominant functors are preserved by coCartesian transport by \cref{cor:iLreltens} and \cref{cor:cdomreltens}, the functor $\cV \mapsto \Arr^{\mathrm{iL, cdom}}(\RMod_{\cV}(\PrL)) \subseteq \Arr(\RMod_{\cV}(\PrL)$ by \kerodon{01VB} induces a full inclusion $\int_{\cV} \Arr^{\mathrm{iL,cdom}} \left( \PrV \right) \hookrightarrow \int_{\cV} \Arr(\PrV) $ on unstraightenings that is a map of coCartesian fibrations over $\Alg(\PrL)$. 
\end{notat}

\begin{prop}\label{prop:bigpullback} The following define pullback squares (see \cref{notat:intPrVC} and \ref{not:restriLcdom}):
\[
		\begin{tikzcd}
			\MMod \arrow[d, hook] \arrow[r] \arrow[dr, phantom, "\scalebox{1}{$\lrcorner$}" , very near start, color=black] & \smallint\nolimits_\cV \Arr^{\mathrm{iL,cdom}} \left( \PrV \right) \arrow[d, hook]                          \\
			\smallint\nolimits^X_\cV  (\PrV)_{X/} \arrow[d, hook] \arrow[r] \arrow[dr, phantom, "\scalebox{1}{$\lrcorner$}" , very near start, color=black] & \smallint\nolimits_\cV \Arr \left( \PrV \right) \arrow[dr, phantom, "\scalebox{1}{$\lrcorner$}" , very near start, color=black]  \arrow[d, hook]  \arrow[r]&\Alg(\PrL) \arrow[d,hook, "\cV \mapsto \id_{\cV}"]
			\\
			(\Spaces \times \Alg(\PrL)) \times_{\RMod(\PrL)} \Arr(\RMod(\PrL))\arrow[r] \arrow[d]\arrow[dr, phantom, "\scalebox{1}{$\lrcorner$}" , very near start, color=black] & \Arr(\RMod(\PrL)) \arrow[r]  \arrow[d, " \mathrm{src}"]& \Arr(\Alg(\PrL))\\
			 \Spaces \times \Alg(\PrL) \arrow[r, "X{,} \cV \mapsto \PSh(X) \otimes \cV_{\cV}"']         & \RMod(\PrL)      
		\end{tikzcd}
\]
Here, $\hookrightarrow$ denote fully faithful functors. 
Moreover, the above squares lift to pullback squares in the subcategory  $\widehat{\mathrm{coCart}}_{/\Alg(\PrL)}$ of $\widehat{\cat}_{/\Alg(\PrL)}$ of coCartesian fibrations and functors preserving coCartesian morphisms.\end{prop}
\begin{proof}
We first show that all squares are pullback squares in $\lcat$. The bottom left one is by definition, the middle right by \cref{prop:grothfacts} $(2)$. The pullback of the middle left square $(\Spaces \times \Alg(\PrL)) \times_{\RMod(\PrL)} \Arr(\RMod(\PrL)) \times_{\Arr(\Alg(\PrL))} \Alg(\PrL)$ is a full subcategory of  $(\lcat \times \Alg(\Catcolim)) \times_{\RMod(\Catcolim)} \Arr(\RMod(\Catcolim)) \times_{\Arr(\Alg(\Catcolim))} \Alg(\Catcolim)$. Using \cref{obs:relativeAdjObs} for the relative adjunction 
\[  \PSh^{\tav\text{-rex}}(-)\otimes ?: \widehat{\cat} \times \Alg(\Catcolim) \rightleftarrows \RMod(\Catcolim): \mathrm{forget}\]
over $\Alg(\Catcolim)$ (composed from the free  cocompletion under small colimits $\PSh^{\tav\text{-rex}}$ (\HTT{Cor.}{5.3.6.10}) and the free algebra functor (\cref{lem:freesm})), the latter category is equivalent to 
\begin{align*}
	& \Alg(\Catcolim) \times_{\Arr(\Alg(\Catcolim))}  \Arr(\lcat \times \Alg(\Catcolim)) \times_{\lcat \times \Alg(\Catcolim)} \RMod(\Catcolim)  \\
	&\simeq \Arr(\widehat{\cat}) \times_{\widehat{\cat}} \RMod(\Catcolim)\, .
\end{align*}
 Unwinding these equivalences,  the pullback of the middle left square is equivalent to the  full subcategory \[\smallint\nolimits^X_{\cV} (\PrV)_{X/}:= \Spaces \times_{\lcat}  \Arr(\lcat) \times_{\lcat} \RMod(\PrL) \subseteq \Arr(\lcat) \times_{\lcat} \RMod(\Catcolim).\] 
Lastly, the top left square is a pullback since a functor $X\to \cM$ with $\cM \in \PrV$ defines a marked module if and only if the induced functor $\PSh(X) \otimes \cV \to \cM$ is internally left adjoint and colimit-dominant by \cref{thm:charessim}.

Since the functor $\widehat{\mathrm{coCart}}_{/\Alg(\PrL)} \to \widehat{\cat}$ creates weakly contractible limits and in particular pullbacks (\cref{reminder:grothadjoints} and \cref{prop:grothfacts} $(1)$), it suffices to show that all categories in the above diagram are coCartesian fibrations over $\Alg(\PrL)$ and all functors are maps of coCartesian fibrations. 
 
The bottom horizontal functor is formally defined as the composite of the presheaf category functor $\PSh: \Spaces \subseteq \cat \to \PrL$ and the free module functor  $\PrL \times \Alg(\PrL) \simeq \Alg_{\operatorname{Triv} \sqcup \Ass}(\PrL) \to \Alg_{\RM}(\PrL)$ induced by the inclusion of operads $\operatorname{Triv} \sqcup \Ass \hookrightarrow \RM$. 
 A morphism in $\Spaces \times \Alg(\Pr) \to \Alg(\Pr)$ is coCartesian if its projection to $\cS$ is an equivalence. On the other hand, a morphism $(\cV \to \cW, \cM_{\cV} \to \cN_{\cW})$ in $\RMod(\PrL) \to \Alg(\PrL)$ is coCartesian if the induced functor $\cM\otimes_{\cV}\cW \to \cN$ is an equivalence. Hence, the bottom horizontal functor is a map of coCartesian fibrations over $\Alg(\PrL)$. A morphism in $\Arr(\Alg(\PrL))$ is coCartesian if the corresponding square in $\Alg(\PrL)$  is a pushout square, hence the functor $\Alg(\PrL) \to \Arr(\Alg(\PrL))$ is a map of coCartesian fibrations. Since $\Arr(\RMod(\PrL)) \to \RMod(\PrL)$ is itself a coCartesian fibration, it follows that $\Arr(\RMod(\PrL)) \to \RMod(\PrL) \to \Alg(\PrL)$ is a coCartesian fibration and the first functor is a map of coCartesian fibrations over $\Alg(\PrL)$. Using that $\Arr(\RMod(\PrL)) \to \Arr(\Alg(\PrL))$ is a coCartesian fibration by \kerodon{01VG}, the same argument shows that it is also a map of coCartesian fibrations over $\Alg(\PrL)$. 
 Finally, the full inclusion   $\smallint\nolimits_\cV \Arr^{\mathrm{iL,cdom}} \left( \PrV \right) \to \int_{\cV}\Arr(\PrV)$ is a map of coCartesian fibrations by \cref{not:restriLcdom}. Since all squares are pullbacks, all remaining functors are also morphisms of coCartesian fibrations.
\end{proof}
	
	We will also often use the following equivalent rephrasing of \cref{prop:bigpullback}:
	\begin{obs}\label{obs:MModPrsubfunctor} Analogous to \cref{obs:MModsubfunctor}, let \[\Spaces \times_{\Pr_{-}} \Arr(\Pr_{-}) : \Alg(\PrL) \to \lcat\] denote the straightening of the coCartesian fibration \[(\Spaces \times \Alg(\PrL))\times_{\RMod(\PrL)} \Arr(\RMod(\PrL)) \times_{\Arr(\Alg(\PrL))} \Alg(\PrL) \to \Alg(\PrL)\; .\] 
	\cref{prop:bigpullback} exhibits $\MModb: \Alg(\PrL) \to \lcat$ as a full subfunctor of $\Spaces \times_{\Pr_{-}} \Arr(\Pr_{-})$. At a fixed $\cV \in \Alg(\PrL)$, the inclusion agrees with the one from \cref{prop:MModVequiv}.
	\end{obs}

	\subsection{Functorial comparison of enriched categories and marked modules}
	\label{subsec:heinecomparison}

We now make the equivalence between marked modules and enriched categories from \S\ref{sec:markednew} functorial in the space of objects and in the enriching category. 
	This functorial comparison is necessarily somewhat technical and will crucially rely  on results from \cite{heine}.	The reader not interested in the technical details of this comparison may safely skip this section; its main \cref{thm:functorialcomparison}  will be later used as a blackbox to relate marked modules and enriched categories.

	For an operad $O$ and a space $X$, Gepner and Haugseng define in \cite[\S4.3]{haugseng} a valent $\cV$-enriched category (in loc.cit.\ referred to as a \emph{categorical algebra}) with space of objects $X$ as a $\Delta_X^{\op}$-algebra in $O$ for a certain generalized operad $\Delta_X^{\op}$. By definition, this assembles into a functor \[\vcat_{ (-)}(-):\Spaces^{\op} \times \OpAss \to \cat\] sending a pair $(X,O)$ to $\Alg_{\Delta_X^{\op}}(O)$.

	\begin{notat} 
	We let $\widehat{\vcat}_{-}(-):  \widehat{\Spaces}^{\op} \times \OpAss \to \lcat$ denote the large variant of the functor in \cite{haugseng},  and write  $\vcat^{\text{large spc}}_{(-)}(-):\widehat{\Spaces}^{\op}  \times  \Alg(\PrL)  \to \lcat$ for its restriction along  the subcategory inclusion $\Alg(\PrL) \to \lOpAss$ and (abusing notation)  $\vcat_{(-)}(-): \Spaces^{\op} \times  \Alg(\PrL)  \to \lcat$ for its further restriction to small spaces. We write $\vcat^{\text{large spc}}(-): \Alg(\PrL) \to \lcat$ and $\vcat(-): \Alg(\PrL) \to \lcat$ for the respective unstraightenings in the $\Spaces$-direction and  $\vEnr_{\text{large spc}} \to \widehat{\Spaces} \times  \Alg(\PrL)$ and $\vEnr \to  \Spaces \times \Alg(\PrL)$ for the corresponding two-sided fibrations.
	\end{notat} 

	\begin{reminder}\label{rem:EndQuiver}
	
	In \cite{heine}, Heine identifies maps of generalized operads $\Delta_X^{\op} \to \cV$ with algebras in a certain generalized operad $\QuivXV$. This was originally defined in \cite{hinich}; and both definitions are equivalent by \cite{macpherson2019operad}, see also \cite[Prop. 4.20]{heine}. Moreover, if $\cV$ is presentably monoidal, this operad $\QuivXV$ is itself a presentably monoidal category \cite[Theorem 4.4.8]{hinich}, \cite[Lem. 4.34]{heine}. By \cite[Def. 3.4.1]{hinich}, $\QuivXV$ has a canonical left action\footnote{The conventions for compositions in enriched categories in \cite{heine} and \cite{hinich} are opposite to ours, which is why we use left modules where they use right modules.} on $\Fun(X^{\op}, \cV)$ which by  \cite[Prop.\ 4.5.3]{hinich} identifies it with the endomorphism algebra $\LinEnd(\Fun(X^{\op}, \cV))$ for the $\Pr$ action on $\PrV$.
	In particular, the induced equivalence $\Alg_{\Delta_X^{\op}}(\cV) \simeq \Alg(\QuivXV) \simeq \Alg(\LinEnd(\Fun(X^{\op}, \cV)))$ identifies valent enriched categories in the sense of~\cite{haugseng} with \cref{def:valentVcat}.
	
	In \cite[Not. 4.21]{heine} Heine also combinatorially defines a left action of $\QuivXV$ on $\Fun(X^{\op}, \cV)$, which agrees with the above canonical action by \cite{heinecomparison}. This means that the presheaf functor  \[\Alg(\QuivXV) \to \PrV, \; \cC \mapsto \LMod_{\cC}(\Fun(X^{\op}, \cV))\] from \cite[Not. 4.32]{heine} agrees with our presheaf functor from~\cref{def:PShVfunctor}. 
	\end{reminder}

	\begin{prop}[{\cite[Prop.~A.5.10]{haugseng}}]\label{prop:GHcartesian}
		For any $X \in \Spaces$ and $f: \cV \to \cW$ in $\Alg(\Pr)$, the induced functor $\vcat_X(\cV) \to \vcat_X(\cW)$ preserves colimits.
	\end{prop}
	
	It follows that the two-sided fibration $\vEnr \to \Spaces \times \Alg(\PrL)$ classifying $\vcat_{-}(-)$ admits a left adjoint by the following lemma:	
\begin{lemma}\label{lem:initialsection}
		Let $p:E \to C \times D$ be a two-sided fibration (c.f.\ \cref{reminder:twosidedfib}), such that the fiber $E_{c, d}$ for any $c \in C, d \in D$ admits an initial object $\emptyset_{c, d}$ and the coCartesian transports $f_! : E_{c, d} \to E_{c, d'}$ for any map $f: d \to d'$ in $D$ preserve initial objects. Then $p$ admits a fully faithful left adjoint, sending $(c, d)$ to $\emptyset_{c, d}$. Moreover, this left adjoint $C\times D\to E$ is a map of coCartesian fibrations over $D$. 
	\end{lemma}
	\begin{proof}
		We  show that $\emptyset_{c, d}$ corepresents the functor $\Map_{C \times D}((c, d), p(-)) : E \to \Spaces$. To this end, we will show that for any $e\in E$ the map		\begin{align*}
			&\Map_{E}(\emptyset_{c, d}, e) \overset{p_*}{\to} \Map_{C \times D}(p(\emptyset_{c, d}), p(e)) \simeq \Map_{C \times D}((c, d), p(e))
		\end{align*}
		is an equivalence. First we prove it is surjective: Given $(f, g): (c, d) \to p(e) = (c', d')$ in $C \times D$, since coCartesian transport preserves inital objects there exists a coCartesian morphism $\emptyset_{c', d} \to \emptyset_{c', d'}$ over $(\id_{c'}, g)$. Also choose a Cartesian lift $\hat{f} : e_0 \to \emptyset_{c', d}$ of $(f, \id_{d})$. Since $e_0 \in E_{c,d}$ there is a unique morphism $\emptyset_{c,d} \to e_0$ covering the identity $\id_{c,d}$, and similarly there is such a morphism $\emptyset_{c',d'} \to e$. Then the composition $\emptyset_{c,d} \to e_0 \to \emptyset_{c', d} \to \emptyset_{c', d'} \to e$ lies over $(f, g)$.
		
		We finish by proving $p_*$ is a monomorphism: Its fiber over $(f, g): (c, d) \to p(e)$  is given by $\Map^{g}_{E_{d}}(g_! \emptyset_{c, d}, e) \simeq \Map_{E_{c, d'}}(g_! \emptyset_{c, d}, f^*(e)) \simeq \Map_{E_{c, d'}}(\emptyset_{c, d'}, f^*(e))$ which is contractible since $\emptyset_{c, d'} \in E_{c, d'}$ is initial.
		
		By construction, the unit of this adjunction is invertible. Moreover, a morphism $(f,g): (c,d) \to (c',d')$ in $C\times D$ is coCartesian over $D$ iff $f$ is an isomorphism; by the above description this gets sent to the coCartesian transport $\emptyset_{c,d} \to g_! \emptyset_{c,d} = \emptyset_{c,d}$.
	\end{proof}
	
	\begin{cor}\label{cor:Free}
	The functor $\vEnr \to \Spaces \times \Alg(\PrL)$ admits a fully faithful left adjoint
	\[\mathrm{Free}(-,-): \Spaces \times \Alg(\PrL) \to \vEnr\] which is a map of coCartesian fibrations over $\Alg(\PrL)$. 
	\end{cor}
	\begin{proof}For $(X, \cV) \in \Spaces \times \Alg(\PrL)$, by \cref{rem:EndQuiver} the fiber identifies with $\Alg(\LinEnd(\PSh(X) \otimes \cV))$, which admits an initial object by \HA{Prop.}{3.2.1.8}; the trivial algebra. Moreover, coCartesian transport by \cref{prop:GHcartesian} preserves initial objects. The statement then follows from \cref{lem:initialsection}.
	\end{proof}
	
	\begin{obs}\label{obs:vEnrrelative} It follows from \cref{cor:Free}  and \cref{exm:ffimpliesrel} that the adjunction $\Spaces \times \Alg(\PrL) \rightleftarrows  \vEnr$ is a relative adjunction over $\Alg(\PrL)$ in the sense of \cref{reminder:relativeadjoints}.	\end{obs}

	We will now use results from \cite{heine} to make the presheaf functor $\PShV: \vcatXV \to \PrV$ from \cref{def:PShVfunctor} functorial in $X$ and $\cV$.

	\begin{prop}[{cf.~\cite{heine,ben2024naturality}}]\label{prop:globalpsh}
	There is a commutative diagram of coCartesian fibrations over $\Alg(\PrL)$
	\[
	\begin{tikzcd} 
	\vEnr \arrow[rr, "\PSh_{\mathrm{enr}}"] && \RMod(\PrL) \\
	& \Spaces \times \Alg(\PrL) \arrow[ul, "{\mathrm{Free}(-,-)}"] \arrow[ur, "{(X, \cV) \mapsto \PSh(X) \otimes \cV_{\cV}}"']
	\end{tikzcd}\;,
	\]
	equivalently a diagram of natural transformations between functors $\Alg(\PrL) \to \lcat$: 
	\[\begin{tikzcd}
	\vcat(-) \arrow[rr, Rightarrow, "\PSh_{-}"] && \Pr_{-} \\
	&\Spaces \arrow[ul, Rightarrow] \arrow[ur, Rightarrow] 
	\end{tikzcd}
	\]	Moreover, the following hold:
	\begin{enumerate}[(1)]
	\item For every  $\cV \in \Alg(\PrL)$  and $X\in \Spaces$, the composition \[\vcat_X(\cV) \to \vcat(\cV) \overset{\PSh_{\cV}(-)}{\to} \RMod_{\cV}(\PrL)\] agrees with the presheaf functor from~\cref{def:PShVfunctor} under the equivalence $\vcatXV\simeq \Alg(\LinEnd(\Fun(X^{\op}, \cV)))$ from \cref{rem:EndQuiver}. 
	\item The functor $\PSh_{\mathrm{enr}}$ is a partial left adjoint\footnote{A functor $G: D \to C$ admits a \emph{partial left adjoint} on a full subcategory $C_0 \subseteq C$ it there exists a functor $F: C_0 \to D$ and an isomorphism $\Map_{D}(F(c_0), d) \simeq \Map_C(c_0, G(d))$ that is natural in $c_0 \in C_0, d \in D$. As for ordinary adjunctions this uniquely specifies the \emph{partial left adjoint} $F$; and it exists iff for any $c_0 \in C_0$ the presheaf $\Map_C(c_0, G(-)) \in \PSh(D)$ is representable. See~ \cite[\S 3]{ben2024naturality}.} of a functor $\chi: \RMod(\PrL) \to \vEnr_{\text{large spc}}$. The unit $\cC\to \chi (\PSh_{\mathrm{enr}}(\cC))$ of this partial adjunction is sent to an isomorphism by $\vEnr_{\text{large spc}} \to \Alg(\PrL)$. (In other words, $\PSh_{\mathrm{enr}}$ is a relative partial left adjoint over $\Alg(\PrL)$ in the sense of \cref{reminder:relativeadjoints}.)
	\end{enumerate}
	
	\end{prop}	
\begin{proof}
In \cite[Not. 4.40]{heine} Heine defines a functor $w\! \vcat_{(-)}(-) : \OpAss \times \Spaces^{\op} \to \lcat$ which contains $\vcat_{(-)}(-)$ as a full subfunctor. We denote by $\widehat{w \! \vEnr}_{\lOpAss} \to \lOpAss \times \lSpaces$ the two-sided unstraightening of its large variant and by ${w \! \vEnr}_{\text{large spc}} \to \Alg(\Pr) \times \lSpaces$ its restriction to $\Alg(\Pr)$. 
By \kerodon{01VB} we obtain a fully faithful map of two-sided fibrations $\vEnr_{\text{large spc}} \subseteq {w \! \vEnr}_{\text{large spc}}$. By \cite[Rem. 4.41]{heine} $\widehat{w \! \vEnr}_{\lOpAss} \to \lOpAss$ is additionally a Cartesian fibration, so the same is true for the pullback ${w \! \vEnr}_{\text{large spc}} \to \Alg(\Pr)$.

 Further Heine constructs in \cite[Thm. 4.59, see Prop. 3.33]{heine} a functor $\lOpRM \to w \! \widehat{\vEnr}_{\lOpAss}$ over $\lSpaces$ (where $\lOpRM \to \lSpaces$ takes the maximal subgroupoid of the module component) which by \cite[Lem. 4.65]{heine} is a map of Cartesian fibrations over $\lOpAss$. Note that $\lOpAss$-Cartesian morphisms in $\lOpRM$ and $\Alg(\PrL)$-Cartesian morphisms in $\RMod(\PrL)$ are precisely the `restriction of scalars' morphisms, i.e.\ those whose module-component is an isomorphism, c.f.\ \cite[Prop. 3.32]{heine} and \HA{Cor.}{4.2.3.2} respectively. Hence, the subcategory inclusion $\RMod(\PrL) \to \Alg(\PrL) \times_{\lOpAss} \lOpRM$ is a map of Cartesian fibrations over $\Alg(\PrL)$. Consequently the composite $\RMod(\Pr) \to {w \! \vEnr}_{\text{large spc}}$ is also a map of Cartesian fibrations over $\Alg(\PrL)$. By \cite[Cor.~6.13]{heine}, it factors through the full subcategory $\vEnr_{\text{large spc}}$. Since  full inclusions of Cartesian fibrations reflect Cartesian morphisms, we obtain a diagram  \begin{center}
 	\begin{tikzcd}
 		\RMod(\PrL) \arrow[rd] \arrow[rr, "\chi"] &                                     & \vEnr_{\text{large spc}} \arrow[ld] \\
 		&  \lSpaces \times \Alg(\PrL) &                                  
 	\end{tikzcd}	,
 \end{center}
where $\chi$ is a map of Cartesian fibrations over $\Alg(\Pr)$, and where $\RMod(\PrL) \to \lSpaces$ once again sends $\cM \mapsto \cM^\simeq$.
		
		It is shown in \cite{ben2024naturality} that for a fixed $\cV \in \Alg(\PrL)$, the induced functor $\RMod_{\cV}(\Pr) \to \vcat^{\text{large spc}}(\cV)$ admits a partial left adjoint  on the full subcategory $\vcat(\cV)$ which sends a $\cC \in \vcat(\cV)$ to $\PShV(\cC)$. Since we need this left adjoint to agree functorially with our presheaf functor from \cref{def:PShVfunctor}, we will explicitly unpack this adjunction:
		By  \cite[Proposition 4.64 and Theorem 5.1]{heine} (using functoriality of the comparison functor $\Psi$ from  \cite[Not. 5.13, 5.14, 5.17]{heine}), it follows that for $\cM \in \PrV$ and any fixed $X \in {\Spaces}$, the following functors are equivalent:
		\[
		\Map_{\vcat_{\text{large spc}}(\cV)}(-, \chi(\cM)) \simeq \LinFun(\PShV(-), \cM)^\simeq : {\vcat}_X (\cV)^{\op} \to \widehat{\Spaces}.
		\]
		Hence, the induced functor $\chi_{\cV}: \RMod_{\cV}(\PrL) \to \vcat^{\text{large spc}}(\cV)$ admits a partial left adjoint on the full subcategory $\vcat(\cV) \subseteq \vcat^{\text{large spc}}(\cV)$. Moreover, at a fixed $X\in \Spaces$ and $\cV \in \Alg(\PrL)$, this left adjoint $\vcat_X(\cV) \to \RMod_{\cV}(\PrL)$ is given by the presheaf functor which according to~\cref{rem:EndQuiver} agrees with our presheaf functor from~\cref{def:PShVfunctor}. 
		Analogous to \HA{Prop.}{7.3.2.6}, it follows that the map of Cartesian fibrations $\chi: \RMod(\PrL) \to \vEnr$ over $\Alg(\PrL)$  admits a partial left adjoint $\PSh_{\mathrm{enr}}: \vEnr \to \RMod(\PrL)$  relative to $\Alg(\PrL)$ which is automatically a map of coCartesian fibrations.
		
		Using that $\Spaces \times \Alg(\PrL) \to \RMod(\PrL)$  is partially left adjoint to $\RMod(\Pr) \to \lSpaces \times \Alg(\Pr), \, \cM_{\cV} \mapsto (\cM^\simeq, \cV)$, taking (partial) left adjoints at $\Spaces \times \Alg(\PrL) \subseteq \lSpaces \times \Alg(\PrL)$  of the above commuting diagram then results in the commuting diagram in the statement. The fact that the diagonal functors are maps of coCartesian fibrations over $\Alg(\PrL)$ follows from \cref{cor:Free} and  \cref{prop:bigpullback}. 
	\end{proof}

	We will now show that the presheaf functor $\PSh_{\mathrm{enr}}: \vEnr \to \RMod(\PrL)$ induces a functor over $\Spaces \times \Alg(\PrL)$ from $\vEnr$ to the two-sided fibration  $\int^X_{\cV} (\PrV)_{X/} \to \Spaces \times \Alg(\PrL)$ from \cref{notat:intPrVC}, classifying  the functor $(X, \cV) \mapsto (\PrV)_{X/}$. In other words,  for any $\cC \in \vEnr$ there is a marking $\ob \cC \to \PShV(\cC)$ which is functorial in $\ob \cC$ and $\cV$.

\begin{prop}\label{prop:penvstar}
The presheaf functor $\PSh_{\mathrm{enr}}: \vEnr \to \RMod(\PrL)$ from \cref{prop:globalpsh} lifts to a commutative diagram
\[\begin{tikzcd}
\vEnr \arrow[rr,"\PSh_{\mathrm{enr}}^{\star}" ] \arrow[dr]  &&\smallint\nolimits^X_{\cV} (\PrV)_{X/} \arrow[dl] \\
& \Spaces \times \Alg(\PrL) 
\end{tikzcd}
\]
of  maps of coCartesian fibrations over $\Alg(\PrL)$. Equivalently, the natural transformation $\PShV(-): \vcat(-) \To \Pr_{-}$ from \cref{prop:globalpsh} lifts to a commutative diagram of natural transformations between functors $\Alg(\PrL) \to \lcat$ (see \cref{obs:MModPrsubfunctor} for notation):
\[\begin{tikzcd}
\vcat(-) \arrow[rr, Rightarrow] \arrow[dr, Rightarrow]  &&\smallint\nolimits^X (\Pr_{-})_{X/} =  \Spaces \times_{\Pr_{-}} \Arr(\Pr_{-}) \arrow[dl, Rightarrow] \\
& \Spaces
\end{tikzcd}.
\]
 Moreover, at a fixed $X\in \Spaces$ and $\cV \in \Alg(\PrL)$, the induced functor $\vcatXV \to (\PrV)_{X/}$ agrees with the functor from \cref{thm:charessim}.
\end{prop}
\begin{proof}
Recall from  \cref{prop:bigpullback}
 the equivalence between  $ \smallint\nolimits^X_{\cV} (\PrV)_{X/}$  and \[(\Spaces \times \Alg(\PrL)) \times_{\RMod(\PrL)} \Arr(\RMod(\PrL)) \times_{\Arr(\Alg(\PrL))} \Alg(\PrL).\]
 By \cref{obs:counit} there is a functor $\vEnr \to (\Spaces \times \Alg(\PrL) )\times_{\vEnr} \Arr(\vEnr)$ over $\Spaces \times \Alg(\PrL)$ sending a $\cC \in \vEnr$ to the counit of the adjunction from \cref{cor:Free}. The commutative diagram from \cref{prop:globalpsh} induces a functor \[(\Spaces \times \Alg(\PrL)) \times_{\vEnr}\Arr(\vEnr) \to (\Spaces\times \Alg(\PrL))\times_{\RMod(\PrL)} \Arr(\RMod(\PrL))
\]
over $\Spaces \times \Alg(\PrL)$. By \cref{obs:vEnrrelative},  the composite of these two functors factors through the full subcategory 
\[
(\Spaces\times \Alg(\PrL))\times_{\RMod(\PrL)} \Arr(\RMod(\PrL)) \times_{\Arr(\Alg(\PrL))} \Alg(\PrL)
\]
on those arrows whose underlying arrow in $\Alg(\PrL)$ is invertible. 
Since this pullback is  by \cref{prop:bigpullback} a limit of coCartesian fibrations over $\Alg(\PrL)$, to verify that the induced functor from $\vEnr$ preserves coCartesian morphisms, it suffices to show that the projections $\vEnr \to \Spaces \times \Alg(\PrL)$ and $\vEnr \to \Arr(\RMod(\PrL)) \overset{\mathrm{tgt}}{\to} \RMod(\PrL)$ preserve coCartesian morphisms. But the former follows since it is a two-sided fibration and the latter follows from \cref{prop:globalpsh}.

	For a fixed $X\in \Spaces$ and $\cV \in \Alg(\PrL)$, the so constructed functors unwinds fiberwise  to the functor $\vcatXV \simeq \vcatXV_{\operatorname{Free}(X, \cV)/} \to (\PrV)_{\PSh(X) \otimes \cV /}$ sending $\cC \in \vcatXV$ to the counit $\PSh(X) \otimes \cV \to \PShV(\cC)$ of the adjunction $\Spaces \times \Alg(\PrL) \rightleftarrows \RMod(\PrL)$ at $\PShV(\cC) \in \RMod(\PrL)$, which is how we had constructed $\EM$ in \cref{cons:EM}.
\end{proof}

\begin{prop}\label{prop:PVenrstaradjunction} For $\cV \in \Alg(\PrL)$ the horizontal functor in the commutative diagram from \cref{prop:penvstar}
\[\begin{tikzcd}
\vcat(\cV)  \arrow[rr, "\cP_{\cV}^{\star}(-)"] \arrow[dr] && \arrow[dl]\int^X (\PrV)_{X/}\simeq \Spaces \times_{\PrV} \Arr(\PrV) \\
& \Spaces
\end{tikzcd}
\] admits a right adjoint. The unit of this adjunction is sent to an equivalence by $\vcat(\cV) \to \Spaces$ (i.e.\ $\PSh_{\cV}^{\star}$ is a relative left adjoint over $\Spaces$ in the sense of \cref{reminder:relativeadjoints}).
 \end{prop}
\begin{proof}
By construction, $\cP_{\cV}^{\star}$ is the composite \[ \vcat(\cV) \to \Spaces \times_{\vcat(\cV)} \Arr(\vcat(\cV)) \to \Spaces \times_{\PrV} \Arr(\PrV)
\] where the first functor is induced per \cref{obs:counit} by the counit of the adjunction $\Spaces \rightleftarrows \vcat(\cV)$ and the second functor is given by composing with $\PShV(-): \vcat(\cV) \to \PrV$. 
Abbreviating $\lvcat(\cV):= \vcat^{\text{large spc}}(\cV)$, this functor $\PShV: \vcat(\cV) \to \PrV$ is a partial left adjoint of a functor $\chi_{\cV}: \PrV \to \lvcat(\cV) $ (using the induced partial adjunction on fibers over $\cV \in \Alg(\PrL)$ from \cref{prop:globalpsh}(2)). Thus, the second functor in the composite defining $\cP_{\cV}^{\star}$ is a partial left adjoint of the functor $\Spaces \times_{\PrV} \Arr(\PrV) \to \Spaces\times_{\lvcat(\cV)} \Arr(\lvcat(\cV))$ induced by $\chi_{\cV}$, relative to $\Spaces$. Since $\lvcat(\cV) \to \lSpaces$ is a Cartesian fibration,  it follows from \cref{lem:Cartadjunction} that the functor $\lvcat(\cV) \to \lSpaces \times_{\lvcat(\cV)} \Arr(\lvcat(\cV))$ admits a right adjoint over $\lSpaces$. Pulling back to $\Spaces \hookrightarrow \lSpaces$, this results in an adjunction $ \Spaces \times_{\lvcat(\cV)} \Arr(\lvcat(\cV)) \leftrightarrows \vcat(\cV)$
whose left adjoint factors through the full subcategory $\Spaces \times_{\vcat(\cV)} \Arr(\vcat(\cV))$  and whose unit is send to an isomorphism in $\Spaces$.
Thus, the composite 
\[ \Spaces \times_{\PrV} \Arr(\PrV) \to \Spaces \times_{\lvcat(\cV)} \Arr(\lvcat(\cV)) \to \vcat(\cV)
\]
defines a right adjoint to $\cP_{\cV}^{\star}$ whose unit is sent to an isomorphism in $\Spaces$.
\end{proof}

\begin{rem}  The right adjoint of $\PSh_{\cV}^{\star}$ sends an $(X \to \cM) \in \int^X (\PrV)_{X/}$ to the internal end in $ \Alg(\LinEnd(\PSh(X) \otimes \cV)) \simeq \vcat_X(\cV)$ of the induced $\cV$-linear colimit-preserving functor in $\LinFun(\PSh(X) \otimes \cV, \cM)$ with respect to its right $\LinEnd(\PSh(X) \otimes \cV)$ action induced by post-composition. Indeed, for a fixed $X\in \Spaces$ by \cref{obs:fiberwiseAdj} the restriction of the right adjoint of $\PSh_{\cV}^{\star}: \vcat(\cV) \to \int^X (\PrV)_{X/}$ along $(\PrV)_{X/} \to \int^X (\PrV)_{X/}$ agrees with the right adjoint of the fiberwise $\vcat_X(\cV) \to (\PrV)_{X/}$ and is by \cref{thm:monadicff}  therefore given by the above  description. 
\end{rem}

	\begin{theorem}\label{thm:functorialcomparison}The functor $\PSh_{\mathrm{enr}}^{\star}: \vEnr \to \int^X_{\cV}(\PrV)_{X/} $ is fully faithful with image the full subcategory $\MMod$. In other words, the presheaf functor induces an equivalence $\vEnr \simeq \MMod$ over $\Spaces \times \Alg(\PrL)$.	\end{theorem}
	\begin{proof}
	Since  $\PSh_{\text{enr}}^\star$ is a map of coCartesian fibrations over $\Alg(\PrL)$ by \cref{prop:penvstar}, it suffices to show that for a fixed $\cV \in \Alg(\PrL)$ the induced functor $\vcat(\cV) \to \int^X(\PrV)_{X/}$ is fully faithful with image $\MModV$. 
	By \cref{prop:PVenrstaradjunction}, this functor has a right adjoint whose unit is sent to an equivalence in $\Spaces$. In particular, the unit at a $\cC \in \vcat(\cV)$ is a morphism in the fiber $\vcat_{\ob \cC} (\cV)$ over $\ob \cC\in \Spaces$ and is in fact the unit of the induced fiberwise adjunction $\vcat_{\ob \cC}(\cV) \rightleftarrows (\PrV)_{\ob \cC/}$ (cf. \cref{obs:fiberwiseAdj}). By  \cref{prop:penvstar} the left adjoint of this fiberwise adjunction is the fully faithful functor from \cref{thm:charessim} and hence has invertible unit. Thus, $\vcat(\cV) \to \int^X (\PrV)_{X/}$ is fully faithful with image characterized fiberwise by  \cref{thm:charessim}.
	\end{proof}

\begin{obs}\label{obs:naturaltrafovEnr}
	In particular, \cref{prop:penvstar} constructs a natural transformation (with notation as in \cref{obs:MModPrsubfunctor}) \[\PSh_{-}^{\star}:\vcat(-) \To  \Spaces\times_{\Pr_{-}}\Arr(\Pr_{-}): \Alg(\PrL) \to \lcat \] 
	and \cref{thm:functorialcomparison} establishes that it is fully faithful at each $\cV \in \Alg(\PrL)$ with full image $\MModV \subseteq \Spaces \times_{\PrV} \Arr(\PrV)$. 
\end{obs}


	\subsection{Properties of presheaves}

	Using the functorial identification of enriched categories and marked modules, we obtain a number of immediate corollaries.	
	
	\begin{obs} \label{obs:functorialPSh}Unwinding \cref{prop:penvstar}, for a $\cV \in \Alg(\PrL)$,  the presheaf functor $\PShV:  \vcatV \to  \PrV$ is equivalent to the composite	\[ \vcatV \simeq \MModV \subseteq \Arr(\widehat{\cat}) \times_{\widehat{\cat}} \PrV \to \PrV  \, . \]
	By \cref{prop:MModVequiv}, it factors through the subcategory $\PrViL$.
	\end{obs}

	\begin{prop}
		The functor $\PShV : \vcatV \to \PrViL$ preserves colimits, and thus by  \cref{lem:iLclosedcolim} so does the composite $\vcatV \to \PrViL \to \PrV$.\end{prop}
	\begin{proof}
	 	Following~\cref{prop:MModVequiv}, this functor is equivalent to the forgetful functor  $\MModV \to \Arr^{\cdom}(\PrViL) \overset{\mathrm{tgt}}{\to} \PrViL$. The functor $\Arr^{\cdom}(\PrViL) \to \Arr(\PrViL)$ preserves colimits by \cref{obs:cdomstab} and so do the projections $\mathrm{src}, \mathrm{tgt}: \Arr(\PrViL) \to \PrViL$ since colimits are computed pointwise in arrow categories. Moreover, the functor $\Spaces \to \PrViL$ sending $X \mapsto \PSh(X) \otimes \mathcal{V}$ is the composition of $\Spaces \hookrightarrow \cat$, $\PSh: \cat \to \PrL$ and the free functor $\PrL \to \PrV$, all of which preserve colimits. Thus, $\MModV \to \Arr^{\cdom}(\PrViL)$ is by \cref{prop:MModVequiv} a pullback of colimit-preserving functors and hence colimit-preserving. The functor $\MModV \to \Arr^{\cdom}(\PrViL) \overset{\mathrm{tgt}}{\to} \PrViL$ is therefore a composite of colimit-preserving functors.
	\end{proof}

\begin{cor}
		\label{cor:changeofenrff}
		Suppose $f: \cV \to \cW$ is a morphism in $\Alg(\PrL)$. If $f$ is fully faithful, then the following hold:
		\begin{enumerate}[(1)]
		\item The change of enrichment functor $f_!: \vcatV \to \vcatW$ is fully faithful with image those $\cW$-enriched categories whose graph factors through $\cV \subseteq \cW$;
		\item For $\cC \in \vcatV$, the functor $\PShV(\cC) \to \PShW(f_! \cC) \simeq \PShV(\cC) \otimes_{\cV} \cW$ induced from coCartesian transport along $f$ is fully faithful. 
		\end{enumerate}
		If instead  $f^\rR : \cW \to \cV$ is fully faithful, then the following hold:
				\begin{enumerate}[(3)]
		\item The change of enrichment functor $f^\rR_! : \vcatW \to \vcatV$ is fully faithful with image those $\cV$-enriched categories whose graph factors through $\cW \subseteq \cV$;

		\item[(4)] For $\cC \in \vcatV$, the functor $\PShV(f_!^\rR \cC) \to \PShW(\cC)$ induced from Cartesian transport along $f$ admits a fully faithful right adjoint. 		\end{enumerate}
		\end{cor}
	\begin{proof}
		For (1) and (3), by \cref{cor:coerelativeadjoint} $f_!$ and $f^\rR_!$ are maps of coCartesian and Cartesian fibrations respectively, and by \cref{obs:changeofenrff} they are fiberwise fully faithful, hence they are fully faithful by \kerodon{01VB}.
		Property (2) follows from \cref{cor:fullyfaithfulatomic}. For (4), note that this map factors through the counit \[ \PSh_{\cV}(f^\rR_! \cC) \simeq \PSh_{\cV}(f^\rR_! \cC) \otimes_{\cV} \cV \to \PSh_{\cV}(f^\rR_! \cC) \otimes_{\cV} \cW \simeq \PSh_{\cV}(f_! f^\rR_! \cC) \to \PShW(\cC) \, . \] This counit $f_! f^\rR_! \to \operatorname{Id}_{\vcatW}$ is an isomorphism by $(3)$. Using \cref{lem:localizationtensor} and the fact that localizations are closed under colimits as the left class in a factorization system on $\Pr$~\cite[Prop.\ 2.13]{monadictower}, it follows that the first map is a localization.
	\end{proof}


	\begin{defin}[{\cite[Def. 5.3.1]{haugseng}}] A functor of valent $\cV$-enriched categories $F: \cC \to \cD$ is called \emph{fully faithful} if for all $c,c' \in \ob \cC$, the induced map on graphs $\Hom_{\cC}(c,c') \to \Hom_{\cD}(Fc, Fc')$ is an isomorphism in $\cV$. 
	\end{defin}

	\begin{prop}
		\label{prop:charff}
		Let $\cV \in \Alg(\PrL)$. For a morphism $F:\cC \to \cD$ in $\vcatV$, the following are equivalent:
		\begin{enumerate}[(1)]
			\item $F$ is a fully faithful functor of $\mathcal{V}$-enriched categories. 
			\item $\PShV(F): \PShV(\mathcal{C}) \to \PShV(\mathcal{D})$ is fully faithful.
			\item $F$ is a Cartesian morphism with respect to the Cartesian fibration $\vcatV \to \Spaces$.
		\end{enumerate}
	\end{prop}
	\begin{proof}
		$(2) \Leftrightarrow (3)$ is \cref{obs:coCartMModVDesc}, and $(2) \Rightarrow (1)$ follows from the first part of \cref{lem:fflemma}. Since $\PShV(\cC)$ is atomically generated by the image of the Yoneda embedding $\ob \cC \to \PShV(\cC)$ and $\PShV(F)$ is internally left adjoint  by \cref{prop:MModVequiv}, $(1) \Rightarrow (2)$ follows from the second part of \cref{lem:fflemma}.
			\end{proof}

%

	\begin{defin}\label{def:surjective}
		A functor of valent $\mathcal{V}$-categories $F: \cC \to \cD$ is called \emph{surjective on objects} if the underlying map of spaces $\ob \cC \to \ob \cD$ is surjective on connected components. 
	\end{defin}

	\begin{lemma}\label{obs:FFSisequivonIm}\label{lem:equonPShV} If $F: \cC \to \cD$ in $\vcat(\cV)$ is fully faithful and surjective on objects, then $\PShV(\cC) \to \PShV(\cD)$ is an equivalence which restricts to an equivalence between the full images $\operatorname{Im}(\ob \cC \to \PSh(\cV)) \to \operatorname{Im}(\ob \cD \to \PShV(\cD))$.
	\end{lemma}
	\begin{proof}
	The functor $\PShV(\cC) \to \PShV(\cD)$ is fully faithful by \cref{prop:charff}, restricting to a fully faithful functor $\operatorname{Im}(\ob\cC) \to \operatorname{Im}(\ob\cD)$. By assumption, this restricted functor is also surjective and hence an equivalence. Thus, $\PShV(\cC)$ is a full subcategory of $\PShV(\cD)$ which is closed under colimits and $\cV$-tensoring and which contains the atomically generating full image of $\ob \cD$. Hence, it is all of $\PShV(\cD)$. 
	\end{proof}


	\begin{lemma}
	\label{prop:changeofenrpres}
			If $f: \cV \to \cW$ is a morphism in $\Alg(\PrL)$, the change-of-enrichment functor $f_! : \vcatV \to \vcatW$ preserves the classes of fully faithful functors and of surjective-on-objects functors. 
	\end{lemma}
	\begin{proof}
	Since $f_!$ does not change spaces of objects, preservation of surjective-on-objects functors is clear. For fully faithful functors combine \cref{prop:changeofenrprespre} with \cref{prop:charff}.
	\end{proof}

	\section{Univalence and flagged categories}
	\label{sec:univalence}

	In this section, we re-develop and re-prove various concepts and statements from \cite{haugseng} surrounding univalent (aka Rezk-complete) enriched categories from the perspective of marked modules, both to demonstrate the usefulness of this language  and for future use.

	\subsection{Univalence}
	
	For a valent $\cV$-enriched category $\cC$, the Yoneda embedding $\yoV_{\cC}: \ob \cC \to \PShV(\cC)$ is not a subcategory inclusion, i.e.\ there can be isomorphisms in $\PShV(\cC)$ between the image of $\yoV_{\cC}$ that are not present in the space $\ob \cC$.

	A valent $\cV$-enriched category $\cC$ has two notions of `isomorphism'; paths in the space $\ob\cC$, or isomorphisms in $\PShV(\cC)$ between representable presheaves. We now define univalent enriched categories, where these two notions are forced to coincide.
	
	\begin{defin}
		\label{def:univalence}
		Let  $\cV \in \Alg(\PrL)$.  A valent $\cV$-enriched category $\cC \in \vcatV$ is called \emph{univalent} if the map $\yoV_{\cC} : \ob \cC\to \PShV(\cC)$ is a subcategory inclusion (i.e.\ a monomorphism in $\widehat{\cat}$ by \kerodon{04W5}). Explicitly, this means that $\yoV_{\cC}$ exhibits $\ob \cC$ as the maximal subgroupoid $\operatorname{Im}( \ob \cC \to \PShV(\cC))^\simeq$ of the full image of $\ob \cC$ in $\PShV(\cC)$. We write  $\catV \subseteq \vcatV$ and $\Enr \subseteq \vEnr$ for the full subcategories on the univalent  enriched categories.  
	\end{defin}
	In \cref{cor:GHcomparison} below, we will show that this coincides with the `complete $\cV$-enriched categories' of \cite{haugseng}.

	\begin{obs}Our equivalence $\vcatV \simeq \MModV$ therefore exhibits $\catV$ as the full subcategory of $\Arr(\widehat{\cat})\times_{\widehat{\cat}} \PrV$ on those functors $\yoV: X \to \cM$ for which $X$ is a space, and for which $\yoV$ atomically generates and is a subcategory inclusion.
	\end{obs}


	\begin{prop}
		\label{lem:PShValmostff}
		For $\cC \in \vcatV$ and univalent $\cD \in \catV$, the enriched presheaf functor $\PShV: \vcatV \to \PrV$ from \cref{obs:functorialPSh} induces a monomorphism on mapping spaces
		\[ \Map_{\vcatV}(\cC, \cD) \hookrightarrow \Map_{\PrV}(\PShV(\cC), \PShV(\cD)) \]
		exhibiting the former as the full subspace on those functors $\PShV(\cC) \to \PShV(\cD)$ which send representable presheaves (i.e.\ objects in the image of $\yoV_{\cC}: \ob \cC \to \PShV(\cC)$) to representable presheaves.
	\end{prop}
	\begin{proof}
	Identifying $\vcatV$ via \cref{thm:functorialcomparison} with the full subcategory $\MModV$ of $\Arr(\widehat{\cat}) \times_{\widehat{\cat}} \PrV$, identifies $\Map_{\vcat(\cV)}(\cC, \cD)$ with the pullback
	\[\begin{tikzcd}
	\Map_{\vcatV}(\cC, \cD) \arrow[r] \arrow[d] \arrow[dr, phantom, "\scalebox{1}{$\lrcorner$}" , very near start, color=black] & \Map_{\widehat{\cat}}(\ob \cC, \ob \cD)\arrow[d]
	\\
	\Map_{\PrV}(\PShV(\cC), \PShV(\cD)) \arrow[r] & \Map_{\widehat{\cat}}(\ob \cC, \PShV(\cD))
	\end{tikzcd}\;.
	\]
	By assumption, $\ob\cD \to \PShV(\cD)$ is a monomorphism in $\widehat{\cat}$ and hence the right vertical map is a monomorphism in $\Spaces$. Thus, so is the left vertical map.
	\end{proof}

		\begin{warning} In other words, \cref{lem:PShValmostff} shows that the composite functor $\catV \subseteq \vcatV \to \PrV$ is \emph{faithful}, a fact which is not true for  $\vcatV \to \PrV$. We  warn the reader that $\catV \to \PrV$ is \emph{not} a monomorphism in $\widehat{\cat}$, i.e.\ a subcategory inclusion: Indeed, an equivalence $\PShV(\cC) \to \PShV(\cD)$ amounts to a Morita equivalence between the enriched categories which does not necessarily come from an equivalence of enriched categories $\cC \to \cD$. 
	\end{warning}

	\begin{rem}
		\cref{lem:PShValmostff} implies that if $\cD \in \catV$ is univalent, then it is a \emph{property} of a morphism $\PShV(\cC) \to \PShV(\cD)$ in $\PrV$ to come from an enriched functor $\cC \to \cD$ (instead of extra structure which needs to be specified in the form of a map $\ob \cC \to \ob \cD$). 
			
	A univalent marked module may therefore equivalently be described as an $\cM \in \PrV$ equipped with a subset $A \subseteq \pi_0 \cM^{\simeq}$ (of `representable presheaves') which atomically generates, and a morphism of univalent marked modules as a morphism $F:\cM \to \cN$ in $\PrV$ with the property that it preserves the chosen subsets. 
	\end{rem}

%

\begin{notat} For a functor $F:C\to D$ between (non-enriched) categories, we let $\mathrm{Im}(F)$ denote the full image.
\end{notat}

\begin{defin}The \emph{univalization} $u\cC \in \catV$ of a valent $\cV$-enriched category $\cC\in \vcatV$ is the enriched category associated via~\cref{thm:charessim} to the marked module \[\operatorname{Im}\left(\vphantom{\frac{a}{b}}\yoV_{\cC}: \ob \cC \to \PShV(\cC)\right)^{\simeq} \to \PShV(\cC).\] 
\end{defin}

\begin{prop}		\label{prop:univalization}
		Univalization defines a left adjoint to the full subcategory inclusion $\catV \hookrightarrow \vcatV$. 
\end{prop}
\begin{proof}
Using the functor $(-)^{\simeq}: \PrV \to \widehat{\cat} \to \widehat{\Spaces}$, we may identify $\MModV$ and hence $\vcatV$ with the full subspace of $\Arr(\widehat{\Spaces})\times_{\widehat{\Spaces}} \PrV$ on those $\cM \in \PrV$ and maps of spaces $X\to \cM^{\simeq}$ whose image atomically generates $\cM$. By definition, $\catV$ then identifies with the pullback of this full subspace to $\Arr^{\mathrm{mono}}(\widehat{\Spaces}) \times_{\widehat{\Spaces}} \PrV$. The (-1)-connected/(-1)-truncated factorization system on $\widehat{\Spaces}$ by \HTT{Lem.}{5.2.8.19} induces an adjunction  $\Arr^{\mathrm{mono}}(\widehat{\Spaces}) \leftrightarrows\Arr(\widehat{\Spaces}) $ whose left adjoint sends a map of spaces $f:X\to Y$ to the full subspace inclusion $\mathrm{Im}(f) \to Y$. 
This in turn induces an adjunction $\Arr^{\mathrm{mono}}(\widehat{\Spaces})\times_{\widehat{\Spaces}} \PrV \leftrightarrows\Arr(\widehat{\Spaces})\times_{\widehat{\Spaces}} \PrV$ which restricts to an adjunction between the relevant full subspaces, given by our claimed expression for $u\cC$.
\end{proof}
\begin{obs}\label{obs:unituni}
Unwinding \cref{prop:univalization}, the unit $\cC \to u\cC$ of the adjunction $\catV \rightleftarrows \vcatV$ corresponds to the map of marked modules 
\[\begin{tikzcd}
\ob \cC \arrow[r] \arrow[d]& \operatorname{Im}(\yoV_{\cC}) \arrow[d]\\
\PShV(\cC) \arrow[r, equal] & \PShV(\cC)
\end{tikzcd}\,.
\]
\end{obs}

We now show that univalence in the sense of \cref{def:univalence} agrees with the notion of  `completeness' from \cite[Def.~5.2.2]{haugseng}. For this, we establish the following variant of \cite[Cor.\ 5.6.4]{haugseng}.

	\begin{theorem}
		\label{thm:univalentFFS} 
		The reflective subcategory $\catV \subseteq \vcatV$ agrees with the localization $\vcatV[FFS^{-1}]$ at the class $FFS$ of enriched functors that are fully faithful and surjective on objects.
	\end{theorem}
	\begin{proof} It suffices to verify that $u:\vcatV\to \catV$ sends $FFS$-morphisms to isomorphisms in $\catV$ and that the unit $\cC \to u\cC$ is in $FFS$.
	The former follows from \cref{obs:FFSisequivonIm} and the latter is immediate from \cref{obs:unituni}. \end{proof}

	

\begin{cor}\label{cor:GHcomparison}
The full subcategory $\catV\subseteq \vcatV$ coincides with the full subcategory on the `complete $\cV$-enriched categories' in the sense of \cite[Def.~5.2.2]{haugseng}.
\end{cor}
\begin{proof}
Let $\catV^{GH} \subseteq \vcatV$ denote the full subcategory of `complete $\cV$-enriched categories' in the sense of ~\cite[Def.~5.2.2]{haugseng}. In
\cite[Def.~5.3.3]{haugseng} defines a notion of `essentially surjective functor' between enriched categories and it is  shown in \cite[Cor.~5.6.4]{haugseng} that $\catV^{GH} $ is the localization of $\vcatV$ at the essentially surjective and fully faithful functors.  
Since any functor which is surjective on objects in the sense of \cref{def:surjective} is in particular essentially surjective (but not vice versa!), it follows that $\catV^{GH} \subseteq \catV$. 
On the other hand, \cite[Def.~5.1.6]{haugseng} defines a map of $\cV$-enriched categories $s^0 : E^1 \to E^0$ and \cite[Prop.~5.4.1]{haugseng} shows that $\catV^{GH}$ is also the localization of $\catV$ at the single morphism $s^0$. Since the underlying map of spaces of $s^0$ is by definition $*\sqcup *  \to *$, $s^0$ is in particular surjective on objects and fully faithful and  hence $\catV \subseteq \catV^{GH}$.
\end{proof}

\begin{rem}\label{rem:es}
Using~\cref{cor:GHcomparison}, we may express Gepner-Haugseng's notion of essential surjectivity in terms of marked modules as follows: A functor $\cC \to \cD$ of valent $\cV$-enriched categories is essentially surjective in the sense of \cite[Def.~5.3.3]{haugseng} if and only if the map of spaces \[\ob \cC \to \operatorname{Im}\left( \vphantom{\frac{a}{b}} \yoV_{\cD}: \ob \cD \to \PShV(\cD) \right)^{\simeq}\] is surjective on connected components. 

Indeed, inspecting the proof of \cite[Prop.~5.3.9]{haugseng}, it follows that if $F: \cC \to \cD$ and $G: \cD \to \cE$ are enriched functors and $G$ is fully faithful and essentially surjective, then $F$ is essentially surjective if and only if $G\circ F$ is. In particular, $\cC \to \cD$ is essentially surjective if and only if $\cC \to \cD \to u \cD$ is. Using \cite[Cor.~5.2.10]{haugseng}, a functor into a `complete' $\cV$-enriched category is essentially surjective if and only if it is surjective on objects. Since the localizations $\catV^{GH}\subseteq \vcatV$ and $\catV \subseteq \vcatV$ agree by \cref{cor:GHcomparison}, we therefore conclude that $\cC \to \cD$ is essentially surjective if and only if $\cC \to \cD\to u \cD$ is surjective on objects, which unpacks to the claimed characterization.\end{rem}

	\subsection{Flagged categories and valent enriched categories}

	\begin{defin}
		A \emph{flagged category} is a category $C$ equipped with a surjective functor $X \to C$ from a space $X$. Denoting by $\cat_{X/^{\mathrm{surj}}} \subseteq \cat_{X/}$ the full subcategory on the surjective functors, we define the \emph{category of flagged categories} 
		\[ \FCat := \smallint\nolimits^X \cat_{X/^{\mathrm{surj}}} = \Spaces \times_{\cat} \Arr^{\mathrm{surj}}(\cat) \subseteq \Arr(\cat) \] 
		as the full subcategory on those surjective arrows in $\cat$ that start at a space.
	\end{defin}
	
	Flagged categories were introduced in \cite{flagged} as a model-independent interpretation of non-complete ($n$-fold) Segal spaces. We now use marked modules to show that flagged categories agree with valent $\Spaces$-enriched categories. 
	In \cref{prop:flaggeduniv}, we will generalize this to express a valent $\cV$-enriched category as a flagged univalent $\cV$-enriched category. Compare \cite[Cor. 6.28]{heine}.

	\begin{constr}\label{con:associatedmarked}
	The \emph{associated marked $\Spaces$-module} of a flagged category $X\to C$ is the marked module $X \to C \to \PSh(C)$. This assembles into a functor $\FCat \to \MModS$ as induced by the pullback presentation of $\MModS$ from \cref{prop:MModVequiv}:		\[
			\begin{tikzcd}[sep=.4cm]
				&[+11pt] \MModS \arrow[rr] \arrow[dd] \arrow[dr, phantom, "\scalebox{1}{$\lrcorner$}" , very near start, color=black] & &[-13pt] \Arr^{\mathrm{iL,cdom}}(\PrL) \arrow[dd, "\mathrm{src}" {pos=0.6}]\\
				\FCat \arrow[rr, crossing over] \arrow[dd] \arrow[ur, dashed] \arrow[dr, phantom, "\scalebox{1}{$\lrcorner$}" , very near start, color=black] &   & \Arr^{\mathrm{surj}}(\cat) \arrow[ur, outer sep=-2pt, "\PSh(-)"']  & \\
				& \Spaces \arrow[rr] &  & \PrL  \\
				\Spaces \arrow[ur, Rightarrow, no head]\arrow[rr] &   & \cat \arrow[from=uu, crossing over, "\mathrm{src}" {pos=0.6}] \arrow[ur, outer sep=-2pt, "\PSh(-)"'] &
			\end{tikzcd}
		\]
	\end{constr}
				\begin{defin}
		\label{def:underlyingflagged}
		Let $\cV \in \Alg(\PrL)$. The \emph{underlying flagged category} of a $\cC \in \vcatV$ is the surjective functor $\ob \cC \to \operatorname{Im}\left(\yoV_{\cC}: \ob \cC \to \PShV(\cC)\right)$ into the full image of $\yoV_{\cC}$. 
	\end{defin}

	%

	\begin{prop}
		\label{prop:freeunderlyingadj}
		Let $\cV \in \Alg(\PrL)$ and let $\iota: \Spaces \to \cV$ in $\Alg(\PrL)$ denote its unit. The composite \[ \FCat \to \MModS \simeq \vcatS \overset{\iota_!}{\to} \vcatV\]  admits a right adjoint, which sends a valent $\cV$-enriched category $\cC$ to its underlying flagged category $\ob \cC \to \operatorname{Im}\left(\yoV_{\cC}: \ob \cC \to \PShV(\cC)\right)$ .	\end{prop}

	\begin{proof}
	Using  \cref{thm:charessim}, we work at the level of marked $\cV$-modules. 
	From \cref{ex:basechangerelative} and \cref{prop:changeofenragree} recall that under the equivalence \cref{thm:charessim}, the change-of-enrichment functor $f_!: \vcatV \to \vcatW$ for a $f:\cV \to \cW$ in $\Alg(\PrL)$ sends a marked $\cV$-module $X\to \cM$ to $X \to \cM \otimes_{\cV}\cW$.
	Let $X\to C$ be a flagged category with image under $\FCat \to \MModV$ given by $(X \to \PSh(C) \otimes \cV)\in \MModV$ and  let $\yoV: Y \to \cM$ be a marked $\cV$-module. Consider the diagram	
	\[
	\begin{tikzcd}
	\Map_{\FCat}\left(\oArr{X}{C}, \oArr{Y}{\operatorname{Im}(\yo)}\right) \arrow[r] \arrow[d] & \Map_{\widehat{\cat}}(C, \operatorname{Im}(\yo)) \arrow[r] \arrow[d] & \Map_{\widehat{\cat}}(C, \cM) \arrow[d] \\
	\Map_{\Spaces}(X, Y) \arrow[r] & \Map_{\widehat{\cat}}(X, \operatorname{Im}(\yo)) \arrow[r] & \Map_{\widehat{\cat}}(X, \cM) 
	\end{tikzcd}
	\]
	The left square is a pullback by definition of the mapping spaces of $\FCat$. Orthogonality of surjective and fully faithful functors of (large) categories shows that the right square is a pullback since $X\to C$ is surjective and $\operatorname{Im}(\yoV) \to \cM$ is fully faithful. Thus, the total square is a pullback and hence
	\[\Map_{\FCat} \left( \oArr{X}{C}, \oArr{Y}{\operatorname{Im}(\yoV)} \right) \simeq \Map_{\Spaces}(X, Y) \times_{\Map_{\widehat{\cat}}(X, \cM)} \Map_{\widehat{\cat}}(C, \cM) \]\[\overset{\text{Lem.~\ref{lem:freetensored}}}{\simeq}  \Map_{\Spaces}(X, Y) \times_{\Map_{\widehat{\cat}}(X, \cM)} \Map_{\PrV}(\PSh(C) \otimes \cV, \cM)  =: \Map_{\MModV} \left( \oArr{X}{\PSh(C) \otimes \cV}, \oArr{Y}{\cM} \right). \qedhere \]
	\end{proof}

Inspecting the unit and counit of this adjunction, we conclude that $\FCat \to \MModS \simeq \vcatS$ is an equivalence:

	\begin{theorem}\label{thm:vcatS} The associated-marked-$\Spaces$-module functor $\FCat \to \MModS \simeq \vcatS$ from \cref{con:associatedmarked} is an equivalence.
	\end{theorem}		
	\begin{proof} Let $(X\to C)\in \FCat$. The unit of the adjunction from \cref{prop:freeunderlyingadj} is the map of flagged categories $(X\to C) \To (X\to \operatorname{Im}(\yo:C \to \PSh(C)))$. Since $C\to \PSh(C)$ is fully faithful, this is an isomorphism. 
    	Conversely, let $(f:Y \to \cM)\in \MModS$. The counit of the adjunction from~\cref{prop:freeunderlyingadj} is the map of marked $\Spaces$-modules $(Y \to \operatorname{Im}(f) \to \PSh(\operatorname{Im}(f))) \To (Y \to \cM)$. Since the image of $f$ consists by definition of completely compact (i.e.\ $\Spaces$-atomic) objects and generates under colimits, it follows from \HTT{Cor.}{5.1.6.11} that $\PSh(\operatorname{Im}(f)) \to \cM$ is an equivalence and hence that the counit is invertible. 
	\end{proof}

	\begin{obs}\label{obs:underlyingisbasechange}
	Let $\cV \in \Alg(\PrL)$ with unit $\iota: \Spaces \to \cV$ in $\Alg(\PrL)$. It then follows from \cref{prop:freeunderlyingadj} that after identifying $\vcatS \simeq \FCat$,  the change-of-enrichment functor $\iota_!^\rR: \vcatV \to \vcatS \simeq \FCat$ sends a $\cC \in \vcatV$ to its underlying flagged category $\ob \cC \to \operatorname{Im}(\yoV_{\cC}:\ob \cC \to \PShV(\cC))$. 
	\end{obs}
	
	\cref{thm:vcatS} may be generalized to $\cV$-enriched categories as follows:  
	\begin{defin}
	A \emph{flagged $\cV$-enriched category} is a univalent $\cV$-enriched category $\cC$ together with a space $X$ and a map of spaces $X\to \ob \cC$ which is surjective on components. Let $\FCat(\cV):= \Arr^{\mathrm{surj}}(\Spaces)\times_{\Spaces} \catV$ denote the category of flagged $\cV$-enriched categories. 
	\end{defin}
	
\begin{prop}\label{prop:flaggeduniv} For $\cV \in \Alg(\PrL)$,  $\FCat(\cV)\simeq \vcatV$.  
\end{prop}
\begin{proof} Use the equivalence $\vcatV \simeq \MModV$ and identify the latter as a full subcategory of $\Arr(\widehat{\Spaces}) \times_{\widehat{\Spaces}} \PrV$ (where the pullback is along the target projection and $(-)^{\simeq}: \PrV \to \widehat{\cat} \to \widehat{\Spaces}$). Using the $(-1)$-connected/truncated (aka epi/mono) factorization system on $\widehat{\Spaces}$, one finds that $\Arr(\widehat{\Spaces}) \simeq \Arr^{\mathrm{surj}}(\widehat{\Spaces}) \times_{\widehat{\Spaces}} \Arr^{\mathrm{mono}}(\widehat{\Spaces})$ compatible with the source and target projection. It is then immediate to see that the induced equivalence 
\[ \Arr(\widehat{\Spaces}) \times_{\widehat{\Spaces}} \PrV \simeq \Arr^{\mathrm{surj}}(\widehat{\Spaces}) \times_{\lSpaces} \Arr^{\mathrm{mono}}(\lSpaces) \times_{\lSpaces} \PrV 
\]
restricts to one between the full subcategories $\vcatV$ and $\Arr^{\mathrm{surj}}(\widehat{\Spaces}) \times_{\lSpaces} \catV$.
\end{proof}

	\subsection{Univalence via the underlying flagged category}

	\begin{defin}
		\label{def:univalentflagged}
		A flagged category $(X \to \ccC)$ is \emph{univalent} if this functor is a monomorphism in $\cat$, or equivalently if the induced map $X \to \ccC^\simeq$ to the maximal subspace of $\ccC$ is an equivalence in $\Spaces$.
	\end{defin}

		\begin{cor} \label{cor:underlyingunivalent}Let $\cV \in \Alg(\PrL)$ and $\cC \in \vcatV$. Then, $\cC$ is univalent in the sense of \cref{def:univalence} if and only if its underlying flagged category is in the sense of \cref{def:univalentflagged}. 
	\end{cor}
	\begin{proof} 
	A $\cC \in \vcatV$ is univalent if $\yoV_{\cC}: \ob \cC \to \PShV(\cC)$ is a monomorphism which is equivalent to requiring $\ob \cC \to \operatorname{Im}(\yoV_{\cC})$ to be a monomorphism. 
	\end{proof}

\begin{obs}
		\label{obs:flaggedadjointscat}
		The target projection $\FCat \to \cat$ admits a right adjoint sending $C$ to $(C^\simeq \to C)$. This exhibits $\cat \subseteq \FCat$ as the full subcategory on  the univalent flagged categories.
	\end{obs}
	
\begin{cor}
		\label{thm:spacesenruniv}
		The equivalence $\vcat(\Spaces) \simeq \FCat$ from \cref{thm:vcatS} induces an equivalence between the full subcategories $\cat(\Spaces) \subseteq \vcatS$ and $\cat \subseteq \FCat$. 	\end{cor}
	\begin{proof}
Immediate from \cref{cor:underlyingunivalent}. \end{proof}

\begin{cor}\label{cor:changeofenrpreserves} Let $f: \cV \to \cW$ be a morphism in $\Alg(\PrL)$. Then, the change-of-enrichment functor $f_!^\rR: \vcatW \to \vcatV$ preserves and reflects univalence.
\end{cor}
\begin{proof} Let $\iota_{\cV}: \Spaces \to \cV$ and $\iota_{\cW}: \Spaces \to \cW$ denote the units of $\cV$ and $\cW$. Then, by \cref{cor:underlyingunivalent}  combined with \cref{obs:underlyingisbasechange}, a $\cC \in \vcatW$ is univalent if and only if $(\iota_{\cV})_! \cC \in \vcat(\Spaces)$ is. The result then follows since $f \circ \iota_{\cV} \simeq \iota_{\cW}$. 
\end{proof}

	\begin{warning} \label{warn:basechangeunivalence}The left adjoint $f_!: \vcatV \to \vcatW$ does generally not preserve univalence: For example consider the terminal functor $f: \cV \to *$ in $\Alg(\PrL)$ for which $f_!(X \to \PShV(\cC)) = (X \to \PShV(\cC) \otimes_\cV *) = (X \to *)$, which is only univalent if $X$ is contractible. 
	\end{warning}
%
%
%
%

\begin{cor}\label{cor:EnrcoCart}	Univalization induces a left adjoint to the full inclusion $\Enr \subseteq \vEnr$.The composite $\Enr \subseteq \vEnr \to \Alg(\PrL)$ is a Cartesian and coCartesian fibration and univalization $u:\Enr \to \vEnr$ is a map of coCartesian fibrations over $\Alg(\PrL)$. It follows that $\Enr$ has limits and colimits and that $\Enr \to \Alg(\PrL)$ preserves them.
	\end{cor}
	\begin{proof}
	 For any morphism $f: \cV \to \cW$ in $\Alg(\PrL)$,  coCartesian transport $f_! : \vcatV \to \vcatW$ preserves fully faithful and surjective-on-objects functors by \cref{prop:changeofenrpres}, and hence it follows from \cite[Prop.\ 5.7.4]{haugseng} that  $\Enr\to \vEnr \to \Alg(\PrL)$ is a coCartesian fibration and that $u:\Enr \to \vEnr$ preserves $\Alg(\PrL)$-coCartesian morphisms. On the other hand, Cartesian transport $f_!^\rR: \vcatW \to \vcatW$  preserves univalence by \cref{cor:changeofenrpreserves} and hence the restriction $\Enr \subseteq \vEnr \to \Alg(\PrL)$ is also a Cartesian fibration. The statement about limits and colimits follows analogously to \cref{prop:CatVcolimSpace}.	\end{proof}
	 
	\begin{obs}Unwinding the proof of~\cref{cor:EnrcoCart}, for an $f: \cV \to \cW$ in $\Alg(\PrL)$, Cartesian transport is given by restricting $f_!^\rR: \vcatW\to \vcatV$ to the full subcategory $\catW$ while coCartesian transport is given by the composite $\catV \subseteq\vcatV \overset{f_!}{\to} \vcatW \to \catW$ with univalization.
	\end{obs}
	
	\begin{cor}\label{cor:catasfunctor}
	The coCartesian fibration $\Enr \to \Alg(\PrL)$ straightens to a functor 
	\[\cat(-): \Alg(\PrL) \to \lcat
	\]
	and univalization assembles into a natural transformation $u:\vcat(-) \To \cat(-)$. 
	\end{cor}
			
	\begin{warning}
	It follows from \cref{warn:basechangeunivalence} that the full inclusion $\Enr \to \vEnr$ is not a map of coCartesian fibrations over $\Alg(\PrL)$. In particular, the inclusions $\cat(\cV) \subseteq \vcat(\cV)$ do not assemble into a natural transformation $\cat(-)\To \vcat(-): \Alg(\PrL) \to \lcat$ (only a lax natural transformation). 	\end{warning}

	\section{The tensor product of enriched categories}
	\label{sec:multiplicativity}
	
	One major advantage of working with marked modules is the simplicity with which one can construct a tensor product of enriched categories and prove that is compatible with colimits, which we will do in this section. We will compare this tensor product with the one defined in \cite[Section 4.3]{haugseng}, \cite{haugseng2023tensor} in \S \ref{sec:comptensor}.

	\subsection{The external and internal tensor product of enriched categories}
	
	We identified $\MMod$ in \cref{prop:bigpullback} as the full subcategory of 	\[
	(\Spaces \times \Alg(\PrL)) \times_{\RMod(\PrL)} \Arr(\RMod(\PrL)) \times_{\Arr(\Alg(\PrL))} \Alg(\Pr)
	\]
	on those functors $\PSh(X) \otimes \cV \to \cM$ which are colimit-dominant and internal left adjoint.
	
	 Hence by~\cref{thm:functorialcomparison}, there is a fully faithful functor \[\vEnr \to  (\Spaces \times \Alg(\PrL)) \times_{\RMod(\PrL)} \Arr(\RMod(\PrL)) \times_{\Arr(\Alg(\PrL))} \Alg(\Pr)\] 
	 sending a valent $\cV$-enriched category $\cC$ to its underlying space of objects $\ob \cC$, its enrichment category $\cV$, and the $\cV$-linear colimit-preserving extension of its Yoneda functor $(\PSh(X) \otimes \cV \to \PShV(\cC)) \in \Arr(\RMod(\PrL))$.
.
\begin{constr} The category $(\Spaces \times \Alg(\PrL)) \times_{\RMod(\PrL)} \Arr(\RMod(\PrL)) \times_{\Arr(\Alg(\PrL))} \Alg(\Pr)$ inherits a symmetric monoidal structure from the Cartesian product on $\Spaces$ and the pointwise symmetric monoidal structure on  $\Alg(\PrL), \RMod(\PrL)$ and $\Arr(\RMod(\PrL))$ induced (see \cref{ex:pointwisetensor}) by Lurie's symmetric monoidal structure on $\PrL$. 

We briefly explain that all relevant functors are indeed symmetric monoidal: The functor $\Spaces \hookrightarrow \cat \overset{\PSh}{\to} \PrL$ is a composition of symmetric monoidal functors by \HA{Rem.}{4.8.1.8}. The functor $\PrL \times \Alg(\PrL) \simeq \Alg_{ \operatorname{Triv} \sqcup \Ass}(\PrL) \to \RMod(\PrL)$ sending $(\cP, \mathcal{V})$ to the free $\cV$-module $\cP \otimes \mathcal{V}$ is symmetric monoidal by \cref{lem:freesm}. The projections $\Arr(\Alg(\PrL)) \overset{\mathrm{src}}{\to} \Alg(\PrL)$ and $\Arr(\RMod(\PrL)) \overset{\mathrm{src}}{\to} \RMod(\PrL)$ as well as the diagonal $\Alg(\PrL) \to \Arr(\Alg(\PrL))$ are symmetric monoidal by the functoriality described in \cref{ex:pointwisetensor}. The projection $\RMod(\PrL) \to \Alg(\PrL)$ is symmetric monoidal by \cref{lem:algforgetsm}. Hence, $(\Spaces \times \Alg(\PrL)) \times_{\RMod(\PrL)} \Arr(\RMod(\PrL)) \times_{\Arr(\Alg(\PrL))} \Alg(\Pr)$ is a limit of symmetric monoidal categories and symmetric monoidal functors and hence itself symmetric monoidal.
\end{constr}

	\begin{theorem}\label{thm:externaltens} There is a unique symmetric monoidal structure  $\boxtimes$ on the category  $\vEnr$ such that the fully faithful functor 
	\[\vEnr \to(\Spaces \times \Alg(\PrL)) \times_{\RMod(\PrL)} \Arr(\RMod(\PrL)) \times_{\Arr(\Alg(\PrL))} \Alg(\Pr)
	\] is symmetric monoidal.
	\end{theorem}
	\begin{proof} It suffices to show that the unit and the binary tensor product of objects in the full subcategory $\vEnr$ remains in the full subcategory. Indeed, its unit is $\Spaces \to \Spaces$ which is in $\vEnr$ and the tensor product of two objects in $\vEnr$ remains in $\vEnr$ by \cref{lem:colimdomtensor} and \cref{lem:propertiestensorbetter}.
	\end{proof}

	\begin{defin}\label{defin:exteriortensor} We refer to the symmetric monoidal structure $\boxtimes$ on $\vEnr$ from \cref{thm:externaltens} as the \emph{external tensor product} of enriched categories. 
	\end{defin}
	
	\begin{obs} \label{thm:PShunivsm}By definition, the external tensor product is set up so that the functor 
	\[ \ob: \vEnr \to \Spaces
	\] sending a valent enriched category $\cC$ to its space of objects and by \cref{obs:functorialPSh} the functor
	\[\PSh_{\mathrm{enr}} : \vEnr \to \RMod(\PrL)
	\] sending $\cC$ to the right $\cV$-module $\PShV(\cC)$ are symmetric monoidal. \end{obs}

	\begin{obs}
		\label{obs:formulagraphext}
		The unit of the external tensor product structure on $\vEnr$ is the terminal $\Spaces$-enriched category $B*$ with one object and contractible endomorphisms (with associated marked module $(\{*\}\to \Spaces)\in \MModS$). For $\cV, \cW  \in \Alg(\PrL), \cC \in \vcatV$ and $\cD \in \vcatW$, their external tensor product $\cC\boxtimes \cD$ is enriched in $\cV \otimes \cW$,  has space of objects $\ob \cC \times \ob \cD$, enriched presheaf category $\PShV(\cC) \otimes \PShW(\cD)$ and Yoneda embedding given by \[\ob \cC \times \ob \cD \to \PShV(\cC) \times \PShW(\cD)\to \PShV(\cC) \otimes \PShW(\cD)\,.\] 
		
For $x\in \ob \cC$ and $y \in \ob \cD$, it follows from \cref{lem:propertiestensorbetter} that
		\[
			\iHom_{\PShV(\mathcal{C}) \otimes \PSh_\cW(\cD)}(\yo {}^\cV_x \otimes \yo {}^\cW_y , - ) \simeq \iHom_{\PShV(\mathcal{C})}(\yo {}^\cV_x, -) \otimes \iHom_{\PSh_\cW(\cD)}(\yo {}^\cW_y, -)\, .
		\]
By \cref{cor:PShVgraph}, it therefore follows that 		\[
		\Hom_{\cC \boxtimes \cD} ((x,y),(x',y')) = \Hom_{\cC}(x,x') \otimes \Hom_{\cD} (y, y') \in \cV \otimes \cW \, .
		\]
	\end{obs}
		
\begin{obs} \label{obs:operadlax}The functor $\RMod(\PrL) \to \Alg(\PrL)$ is a coCartesian fibration of operads and hence by \cref{cor:coCartoperadscrit} straightens to a lax\footnote{In fact, it factors through $\Catcolim$ where it is symmetric monoidal by \HA{Thm.}{4.8.5.16}}  symmetric monoidal structure on the functor $\Pr_{-}: \Alg(\PrL) \to \lcat$. 
In particular, for any operad $O$ it induces a functor 
		\[
		\Alg_{O \otimes \E_1}(\PrL) \simeq \Alg_{O} (\Alg(\PrL)) \to \Alg_{O}(\widehat{\cat})\,  , 
		\]
		i.e.\ for any $\cV \in \Alg_{O \otimes \E_1}(\PrL)$, the category $\PrV:= \RMod_{\cV}(\PrL)$ inherits an $O$-monoidal structure. 
\end{obs}
	
\begin{lemma}\label{lem:coCartOpPr} The symmetric monoidal functor \[(\Spaces \times \Alg(\PrL)) \times_{\RMod(\PrL)} \Arr(\RMod(\PrL)) \times_{\Arr(\Alg(\PrL))} \Alg(\PrL) \to \Alg(\PrL)\] is a coCartesian fibration of operads and the symmetric monoidal full subcategory $\vEnr$ is a sub-coCartesian fibration of operads. 
\end{lemma}
\begin{proof} 
	By \cref{thm:externaltens} all relevant functors are symmetric monoidal and the underlying functor $(\Spaces \times \Alg(\PrL)) \times_{\RMod(\PrL)} \Arr(\RMod(\PrL)) \times_{\Arr(\Alg(\PrL))} \Alg(\PrL) \to \Alg(\PrL)$ is a coCartesian fibration. Moreover, the underlying category of the symmetric monoidal subcategory $\vEnr$ is closed under coCartesian transport. By \cref{cor:coCartoperadscrit}, it therefore suffices to show that tensoring with objects in  $(\Spaces \times \Alg(\PrL)) \times_{\RMod(\PrL)} \Arr(\RMod(\PrL)) \times_{\Arr(\Alg(\PrL))} \Alg(\PrL)$ preserves $\Alg(\PrL)$-coCartesian morphisms. A coCartesian morphism over $\cV \to \cV'$ in $\Alg(\PrL)$ is of the form $(X, \cV, \PSh(X) \otimes \cV \to \cM) \to (X, \cV', \PSh(X) \otimes \cV' \to \cM \otimes_{\cV} \cV')$. So tensoring with any other $(Y, \cW, \PSh(Y) \otimes \cW \to \cN)$ yields
	\[
	\begin{tikzcd}
		(X \times Y, \cV \otimes \cW, \PSh(X \times Y) \otimes \cV \otimes \cW \to \cM \otimes \cN ) \arrow[d] \\
		\left(X \times Y, \cV' \otimes \cW, \PSh(X \times Y) \otimes \cV' \otimes \cW \to (\cM \otimes_{\cV} \cV') \otimes \cN \simeq (\cM \otimes \cN) \otimes_{\cV \otimes \cW} (\cV' \otimes \cW) \right)
	\end{tikzcd}
	\]
	which is once again determined by extension-of-scalars along $\cV \otimes \cW \to \cV' \otimes \cW$ and hence is coCartesian.
\end{proof}

\begin{constr}\label{constr:laxPr}By \cref{cor:coCartoperadscrit}, the coCartesian fibration of operads from \cref{lem:coCartOpPr} straightens to a lax symmetric monoidal structure on the natural transformation from \cref{obs:naturaltrafovEnr}:
\[\Spaces \times_{\Pr_{-}} \Arr(\Pr_{-}) : \Alg(\PrL) \to \lcat
\] 
In particular, as in \cref{obs:operadlax} for any operad $O$ and any presentably $O\otimes \E_1$-monoidal category $\cV$ this induces an $O$-monoidal structure on $\Spaces \times_{\PrV} \Arr(\PrV)$.
\end{constr}

	\begin{cor}\label{cor:laxsymmetric}
	The functor \[ \vcat(-): (\Alg(\PrL), \otimes) \to (\widehat{\cat}, \times)\]
	admits a unique lax symmetric monoidal structure for which the fully faithful natural transformation \[\vcat(-) \To \Spaces \times_{\Pr_{-}} \Arr(\Pr_{-}) : \Alg(\PrL) \to \widehat{\cat}\]  from \cref{obs:naturaltrafovEnr} is symmetric monoidal for the lax symmetric monoidal structure on $\Spaces \times_{\Pr_{-}} \Arr(\Pr_{-})$ from \cref{constr:laxPr}. 	\end{cor}
	\begin{proof} 
	Immediate application of \cref{cor:coCartoperadscrit} using \cref{lem:coCartOpPr} and \cref{thm:externaltens}. 
	\end{proof}
	
	Using this lax symmetric monoidal structure, we immediately obtain the following corollary:
	\begin{cor}\label{constr:internaltensor}
	For any operad $O$, and $\cV \in  \Alg_{O \otimes \E_1}(\PrL) \simeq \Alg_O \Alg_{\E_1}(\PrL)$ a presentably $O \otimes \E_1$-monoidal category, there is a unique $O$-monoidal structure on $\vcat(\cV)$ making the fully faithful functor $\vcat(\cV) \to \Spaces\times_{\PrV}\Arr(\PrV)$ $O$-monoidal. 
	\end{cor}
	
	\begin{defin} We refer to the $O$-monoidal structure from \cref{constr:internaltensor}  as the \emph{internal $O$-monoidal structure} on $\vcat(\cV)$.
	\end{defin}

	\begin{ex}
			\label{ex:internaltensor}
		By Dunn additivity $\E_n \simeq \E_{n-1} \otimes \E_1$, so if $\cV \in \Alg_{\E_n}(\PrL)$ is $\E_n$-monoidal then we obtain an $\E_{n-1}$-monoidal structure on $\vcatV$, which we call the \emph{internal tensor product}. In particular, if $\cV \in \CAlg(\PrL) = \Alg_{\E_{\infty}}(\PrL)$, then the internal tensor product is also symmetric monoidal.	\end{ex}

		\begin{obs}
		\label{cor:PShVOmonoidal}
		For $O$ an operad and $\cV \in \Alg_{O\otimes \E_1}(\PrL)$, the target projection $\Spaces \times_{\PrV}\Arr(\Pr_{\cV}) \to \PrV$ is $O$-monoidal by \cref{constr:laxPr}. Hence, it follows from  \cref{constr:internaltensor} that the presheaf functor 
		\[\PShV: \vcatV \to \PrV
		\]
		is $O$-monoidal. 		\end{obs}

\begin{obs}\label{obs:internaltensorexpression}
Unwinding \cref{constr:internaltensor}, if $\cV \in \Alg_{O \otimes \E_1}(\Pr)\simeq \Alg_{O} (\Alg_{\E_1}(\Pr))$ is represented by a map of operads $O \to \Alg(\Pr)$, then the category of operators $\vcatV^{\otimes}$ of the $O$-monoidal category $\vcatV$ is given by the pullback 
		\[
			\begin{tikzcd}
				\vcatV^\otimes \arrow[r] \arrow[d] \arrow[dr, phantom, "\scalebox{1}{$\lrcorner$}" , very near start, color=black] & \vEnr^\boxtimes \arrow[d] \\
				\mathcal{O}^\otimes \arrow[r] & \Alg(\PrL)^\otimes
			\end{tikzcd}\,.
		\]
		By the decription of coCartesian morphisms in \cref{lem:coCartoperadscriterion}, given colors  $X_1, \ldots, X_n, Y \in \underline{O}$ and an $n$-ary multimorphism $\alpha \in \Map_{O}(X_1,\ldots, X_n ; Y) $ (i.e.\ a morphism in $O^{\otimes}$ covering the unique map $\langle n \rangle \to \langle 1 \rangle$ in $\Fin$), the induced operation is given by 
		\[ \bigotimes\nolimits_{\alpha}: \prod_i \vcat(\cV_{X_i}) \overset{\boxtimes}{\to} \vcat(\otimes_i \cV_{X_i}) \overset{(\otimes_{\alpha})_!}{\to} \vcat(\cV_Y)
		\]
		where $\boxtimes$ denotes the external tensor product, $\cV_{X_i}, \cV_Y\in \Alg(\PrL)$ and the morphism $\bigotimes\nolimits_{\alpha} : \otimes_i \cV_{X_i} \to \cV_Y$ in $\Alg(\PrL)$ are the operations of the $O$-algebra $\cV$ in $\Alg(\PrL)$. 
		
		In particular, for $\cC_i \in \vcat(\cV_i)$, their internal tensor product $\bigotimes\nolimits_{\alpha} \cC_i$ has space of objects $\prod_i \ob \cC_i$ and graph 
		\[ \Hom_{\bigotimes\nolimits_{\alpha} \cC_i} \left( (x_i)_i, (y_i)_i \right) = \bigotimes\nolimits_{\alpha} \Hom_{\cC_i}(x_i, y_i).
		\]		
%
%
%
\end{obs}

In the terminology of \HA{Def.}{2.1.2.13}, \cref{lem:coCartOpPr} asserts that the functor $\vEnr^{\boxtimes} \to \Alg(\PrL)^{\otimes}$ defines an $\Alg(\PrL)$-monoidal category. We now show that this $\Alg(\PrL)$-monoidal category is compatible with colimits in the sense of \HA{Def.}{3.1.1.18}.

	We immediately recover the following main result from  \cite{haugseng2023tensor} (for our notion of external and internal tensor product, which we compare to those of \cite{haugseng2023tensor} in \S \ref{sec:comptensor}). 
	\begin{theorem}\label{thm:preservescolim}
	For $\cV, \cW \in \Alg(\PrL)$, the functor 
		\[\boxtimes:  \vcatV \times \vcatW \to \vcat(\cV \otimes \cW) \]
		induced by the external tensor product preserves colimits separately in both variables. In particular, for $\cV \in \Alg_{O \otimes \E_1}(\PrL)$ and any $n$-ary operation $\alpha \in \Map_{O}(X_1, \ldots, X_n; Y)$ in $O$, the internal  tensor product
		\[\bigotimes\nolimits_\alpha : \prod_i \vcat(\cV_{X_i})\overset{\boxtimes}{\to} \vcat(\otimes_i \cV_{X_i}) \overset{(\otimes_\alpha)_!}{\to} \vcat(\cV_Y) \]
		preserves colimits separately in each variable.
	\end{theorem}
	\begin{proof}
Let  $(Y \to \cM) \in \MModW$ be a fixed marked $\cW$-module. We now show that $- \boxtimes (Y \to \cM): \MModV \to \operatorname{Mod}^\star_{\cV \otimes \cW}$ preserves colimits. 
First note that all functors in the diagram
\[\begin{tikzcd}
\Spaces \arrow[rrr, "-\times Y"] \arrow[d, "\PSh(-) \otimes \cV"'] && &\Spaces\arrow[d, "\PSh(-) \otimes \cV \otimes \cW"] \\
\PrV \arrow[rrr, "-\otimes\PSh(Y) \otimes \cW"] &&& \Pr_{\cV\otimes \cW} \\
\Arr(\PrV) \arrow[u, "\text{src}"] \arrow[rrr, " - \otimes(\PSh(Y) \otimes \cW \to \cM)"] &&& \Arr(\PrL_{\cV \otimes \cW}) \arrow[u, "\text{src}"']
\end{tikzcd}
\]
preserve colimits (since both the monoidal structure on $\PrL$ and the Cartesian monoidal structure on $\Spaces$ preserve colimits separately in both variables), and hence so does the induced functor $\Spaces\times_{\PrV} \Arr(\PrV) \to \Spaces\times_{\PrL_{\cV\otimes \cW}} \Arr(\PrL_{\cV \otimes \cW})$. Since the full subcategory inclusions $\vcatV \simeq \MModV \subseteq \Spaces\times_{\PrV}\Arr(\PrV)$ are by \cref{cor:MModVclosed} closed under colimits, the result follows.
	\end{proof}
	
	Equip $\Catcolim$ with the symmetric monoidal structure from \HA{Cor.}{4.8.1.4} and note that the subcategory inclusion $\Catcolim \to \widehat{\cat}$ defines a lax symmetric monoidal functor (as a right adjoint to a strong symmetric monoidal functor). 

\begin{cor}
The lax symmetric monoidal functor $\vcat(-): \Alg(\Pr) \to \widehat{\cat}$ from \cref{cor:laxsymmetric} factors as a lax symmetric monoidal functor through the subcategory $\Catcolim \to \widehat{\cat}$. \end{cor}
\begin{proof}
Since $\vEnr \to \Alg(\PrL)$ is symmetric monoidal, it suffices to prove that for all $\cV \in \Alg(\PrL)$, the category $\vcat(\cV)$ has colimits which follows from \cref{prop:CatVcolimSpace}, that change-of-enrichment along morphisms $f:\cV \to \cW$ in $\Alg(\PrL)$ preserves colimits which is proven in \cref{cor:coerelativeadjoint} and finally that for $\cV, \cW \in \Alg(\PrL)$ the functor $ \boxtimes: \vcatV \times \vcatW \to \vcat(\cV \otimes \cW)$ induced by the external tensor product preserves colimits in both variables, which is \cref{thm:preservescolim}. 
\end{proof}


	\subsection{Tensor product and univalence}

	\begin{prop}
		\label{prop:extprodandfunctors}
		If $\cV , \cW \in \Alg(\PrL)$ and $F: \cC \to \cC'$ a morphism in $\vcatV$ and $G:\cD \to \cD'$ a morphism in $\vcatW$ which are either both surjective on objects or both fully faithful. Then, so is their external tensor product $F\boxtimes G: \cC \boxtimes \cD \to \cC' \boxtimes \cD'$ in $\vcat(\cV \otimes \cW)$.

		Moreover, if $\cV\in \Alg_{O \otimes \E_1}(\PrL)$, $\alpha$ is an $n$-ary multimorphism in $O$ and $F_1: \cC_1 \to \cD_1, \dots, F_n : \cC_{n} \to \cD_n$ are $\cV$-enriched functors all of which are either surjective on objects, or fully faithful, then so is their $\alpha$-tensor product:
		\[ \bigotimes\nolimits_{\alpha} F_i : \bigotimes\nolimits_{\alpha} \cC_i \to \bigotimes\nolimits_{\alpha} \cD_i \, . \]
	\end{prop}	
	\begin{proof}
		Since the internal tensor product is composed from external tensor product and change-of-enrichment, the second statement follows from the first together with \cref{prop:changeofenrpres}. By symmetric monoidality of $\boxtimes$ and the fact that surjective-on-objects and fully faithful functors are closed under composition, it suffices to prove the first statement for $G=\id_{\cD}$. 
		
		Surjectivity on objects follows since if $\ob \cC \to \ob\cC'$ is a surjective map of spaces, then so is the product $\ob \cC \times \ob \cD \to \ob\cC' \times \ob \cD$. 
		
		If $F$ is fully faithful, i.e.\ induces equivalences on graphs, then so does $F \boxtimes \cD$ since using~\cref{obs:formulagraphext} it acts as
		\begin{align*} 
		&\Hom_{\cC \boxtimes \cD} ( (c_0, d_0) , (c_1, d_1) ) \simeq \Hom_{\cC} (c_0, c_1) \otimes \Hom_{\cD} (d_0, d_1) \overset{\simeq}{\to} \\
		&\overset{F \boxtimes \id}{\to} \Hom_{\cC'} (Fc_0, Fc_1) \otimes \Hom_{\cD} (d_0, d_1) \simeq \Hom_{\cC' \boxtimes \cD} ( (F \boxtimes \cD)(c_0, d_0) , (F \boxtimes \cD)(c_1, d_1) ) \,. \qedhere
		\end{align*}
	\end{proof}

	\begin{warning}
	Like \cref{warn:basechangeunivalence}, if $\cV, \cW \in \Alg(\PrL)$ and $\cC\in \vcatV$ and $\cD\in \vcatW$ are univalent, then their external product $\cC \boxtimes \cD \in \vcat(\cV \otimes \cW)$ need not be. Consider as an example $\cC = (C^\simeq \to \PSh(C)) \in \cat(\Spaces)$ for $C \in \cat$  and $\cD := (* \to *)\in \cat({*})$. Then, $\cC \boxtimes \cD = (C^\simeq \times * \to \PSh(C) \otimes *) \simeq (C^\simeq \to *) \in \vcat({*})$, which is univalent iff $C^\simeq$ is contractible. 
	This non-preservation of univalence is due to our external tensor product being defined over $\Alg(\PrL)$ with its induced monoidal structure from $\PrL$. In \S \ref{subsec:smalltensor} we will discuss an external tensor product over $\Alg(\widehat{\cat})$ with its Cartesian monoidal structure which will preserve univalence. 
	\end{warning}

	\begin{cor}
		\label{prop:univalenttensor}
		The external tensor product induces a unique symmetric monoidal structure $\boxtimes_u$ on the full subcategory $\Enr \subseteq \vEnr$ of univalent enriched categories for which the univalization functor $u: \vEnr \to \Enr$ is symmetric monoidal. 
		\end{cor}
	\begin{proof}
		By \cref{prop:extprodandfunctors}, both external and internal tensor product preserve the class of $FFS$ functors, hence they are compatible with the localization $u: \vcatV \to \catV$ in the sense of \HA{Def.}{2.2.1.6}. Our statements then follow from \HA{Prop.}{2.2.1.9}.
	\end{proof}
	
	Straightening immediately yields the following corollary:
	\begin{cor}\label{cor:catpreservescolims}
	There is a unique lax symmetric monoidal structure on \[
		\cat(-): (\Alg(\PrL), \otimes) \to (\lcat, \times) 
		\]
		for which the natural transformation $u:\vcat(-) \To \cat(-) $ is symmetric monoidal.  Moreover, for $\cV, \cW \in \Alg(\PrL)$ the induced functor 
	\[ \cat(\cV) \times \cat(\cW) \to \cat(\cV \otimes \cW)\]
	preserves colimits separately in both variables.
	\end{cor}
	\begin{proof}
	Since $\vEnr \to \Enr$ is symmetric monoidal, it induces an adjunction $\vEnr^{\otimes} \leftrightarrows \Enr^{\otimes}$ between the categories of operations and hence by \cref{lem:coCartrefl}, it follows that $\Enr^{\otimes} \to \Alg(\PrL)^{\otimes}$ is a coCartesian fibration. Thus, it follows from \cref{cor:coCartoperadscrit}  that it straightens to a lax symmetric monoidal structure on the functor $\cat(-): \Alg(\PrL) \to \cat$ from \cref{cor:catasfunctor}. 	
	Since $\vEnr \to \Enr$ is symmetric monoidal, \cref{prop:univalenttensor} and its underlying functor is a map of coCartesian fibrations by \cref{cor:EnrcoCart}, it follows from \cref{cor:coCartoperadscrit} that it straightens to a symmetric monoidal natural transformation $\vcat(-) \To\cat(-)$.   
	By construction, the tensor product functor is given by $\cat(\cV) \times \cat(\cW) \subseteq \vcat(\cV) \times \vcat(\cW) \to \vcat(\cV \otimes \cW) \overset{u_{\cV \otimes \cW}}{\longrightarrow} \cat(\cV \otimes \cW)$. 
	Since $u_{\cV}$ is a left adjoint, preservation of colimits follows from \cref{thm:preservescolim}.
	\end{proof}
	
	\begin{cor}\label{cor:catVmonoidal}
	For $O$ an operad and $\cV \in \Alg_{O \otimes \E_1}(\PrL)$ a presentably $O\otimes \E_1$-monoidal category,  the internal tensor product induces a unique $O$-monoidal structure on $\catV$  such that the localization $u: \vcatV \to \catV$ is $O$-monoidal. This $O$-monoidal structure is compatible with colimits.
	\end{cor}

\begin{defin} \label{def:univalizedtensor}
	We call the external tensor product on $\Enr$ the \emph{univalized external tensor product}, and for $\cV \in \Alg_{O \otimes \E_1}(\PrL)$, we call the induced $O$-monoidal structure on $\cat(\cV)$ the \emph{univalized internal tensor product}.
\end{defin}

	\section{Enriching in small monoidal categories}
	\label{sec:small}

Our treatment of enriched categories in terms of marked modules, as developed above, is particularly suited for enrichment in presentably monoidal $\cV$. We now explain how this framework allows to handle enrichment in a general small monoidal category.

	\subsection{Enrichment over small monoidal categories}
	\label{subsec:smallmonoidal}
	
	\begin{notat}
		\label{notat:smallenrichment}
		For $\cV \in \Alg(\PrL)$ and $V_0 \subseteq \cV$ a full subcategory, we denote by \[\vcat(V_0 | \cV)\subseteq \vcat(\cV)\] the full subcategory on those valent $\cV$-enriched categories whose graphs factor through $V_0 \subseteq \cV$. Similarly, we denote by $\cat(V_0| \cV)$ the full subcategory of $\cat(\cV)$ on the univalent $\cV$-categories with graph in $V_0$.
		\end{notat}

\begin{lemma}\label{lem:easyoperad} Let $C \to D$ be a monoidal functor between monoidal categories and $C_0 \subseteq C$ a full subcategory so that $C_0 \to C \to D$ is fully faithful. Let $\Alg(C)|_{C_0} \subseteq \Alg(C)$ be the full subcategory of algebras in $C$ whose underlying object lies in $C_0$. Then, the induced functor $\Alg(C)|_{C_0} \to \Alg(D)|_{C_0}$ is an equivalence. 
\end{lemma}
\begin{proof} Let $C_0^{\otimes} \subseteq C^{\otimes}$ be the full subcategory on those objects $(c_1, \ldots, c_n) \in C^{\otimes}$ for which each $c_i \in C_0$. Then, $C_0^{\otimes}$ is an operad and $C_0^{\otimes} \to C^{\otimes}$ is a fully faithful operad map, cf.\ \cite[\HAsubsec{2.2.1}]{HA}.  Moreover, the composite $C_0^{\otimes} \to C^{\otimes} \to D^{\otimes}$ is fully faithful and hence induces a fully faithful functor $\Alg(C)|_{C_0} \simeq \Alg(C_0) \to \Alg(D)$ with image those algebras in $D$ in the image of $C_0$. 
\end{proof}

\begin{cor} \label{prop:restrictedgraph}Let $\cV \to \cW$ be a morphism in $\Alg(\PrL)$ and $V_0 \subseteq \cV$ a full subcategory so that the composite $V_0 \hookrightarrow \cV \to \cW$ is fully faithful. Then, the induced change-of-enrichment functor $\vcat(\cV) \to \vcat(\cW)$ restricts to an equivalence $\vcat(V_0|\cV) \to \vcat(V_0|\cW)$ which restricts to an equivalence $\cat(V_0|\cV) \to \cat(V_0|\cW)$. 
\end{cor}
\begin{proof}
Since change-of-enrichment fixes the underlying space of objects, it suffices to prove that $\vcat_X(V_0|\cV) \to \vcat_X(V_0|\cW)$ is an equivalence for any $X\in \Spaces$. Identifying this functor with the functor $\Alg(\LinEnd(\PSh(X) \otimes \cV)) \to \Alg(\LinEndW(\PSh(X) \otimes \cW))$, the result follows immediately from \cref{lem:easyoperad}. Moreover, a  $\cC \in \vcat_X(V_0 | \cV) \subseteq \Alg(\LinEnd(\PSh(X) \otimes \cV))$ is univalent if and only if the functor $\yoV_{\cC}:  \ob \cC \to \LMod_{\cC}(\PSh(X) \otimes \cV))$ is a monomorphism. Let $\LMod_{\cC}(\PSh(X) \otimes \cV)|_{\Fun(X, V_0)}$ denote the full subcategory on those $\cC$-modules whose underlying object lies in $\Fun(X, V_0) \subseteq \Fun(X, \cV) \simeq \PSh(X) \otimes \cV$. Then, $\yoV_{\cC}$ factors through this full subcategory and by assumption, postcomposition with $\cV \to \cW$ induces a fully faithful functor $\LMod_{\cC}(\PSh(X) \otimes \cV)|_{\Fun(X, V_0)} \to \LMod_{\cC}(\PSh(X) \otimes \cW)$. Hence $\yoV_{\cC}$ is a monomorphism if and only if its composite $X\to \PShV(\cC) \to \PShV(\cC) \otimes_{\cV} \cW$ is. 
\end{proof}

	\begin{defin}	\label{def:smallenrichment}	Given a small monoidal category $V \in \Alg(\cat)$, we define the category of \emph{valent, resp. univalent, $V$-enriched categories} as \begin{align*} \vcatsmall(V) &:= \vcat(V | \PSh(V))\\ \catsmall(V) &:= \cat(V | \PSh(V)). \end{align*}
		
	\end{defin}
	
	\begin{rem} In \cref{def:smallenrichment}, the superscript $\tav$ indicates that we consider enrichment in small monoidal categories $V\in \Alg(\cat)$ instead of $\cV \in \Alg(\PrL)$. This superscript will ultimately be superfluous: If $\cV \in \Alg(\PrL)$, we may forget its presentability and consider it as $\cV \in \Alg(\lcat)$ and consider the category $\vcatlarge(\cV)$ defined as in   \cref{def:smallenrichment}  by embedding $\cV$ into a large-presentable category. In \cref{cor:centralcomparison}, we will show that this category agrees with $\vcat(\cV)$, removing any ambiguity. We nevertheless keep the superscript for better readability of our arguments in the next two sections. 
	\end{rem}

			\cref{def:smallenrichment} is independent of the choice of embedding of $V$ into a presentable category, the choice of $\PSh(V)$ is merely a matter of convenience.
	\begin{cor} Let $V\in \Alg(\cat)$ and $V \hookrightarrow \cW$ be a fully faithful functor into a $\cW \in \Alg(\PrL)$. Then, the induced functor $\PSh(V) \to \cW$ induces an equivalence $\vcatsmall(V):= \vcat(V| \PSh(V))  \to \vcat(V| \cW)$ which restricts to an equivalence $\catsmall(V):= \cat(V|\PSh(V)) \to \cat(V|\cW)$. 
\end{cor}
\begin{proof} Immediate from \cref{prop:restrictedgraph}.
\end{proof}

\begin{rem} More generally, given a non-symmetric operad $O \in \OpAss$, we may define $\vcat^{\OpAss}(O):= \vcat(\underline{O}| \PSh(\Env(O)))$ where $\Env(O)$ is the monoidal envelope of $O$ and $\underline{O} \subseteq \Env(O) \subseteq \PSh(\Env(O))$ the full inclusion of the underlying category (`of colors') of $O$ into the presentably monoidal category $\PSh(\Env(O))$. All following results  will have evident analogs  in this more general setting. For brevity of exposition, we will however not develop this in this paper. 
\end{rem}

	\begin{defin}\label{defin:erichedalg}
	Let $\vEnr_{\Alg(\cat)} \subseteq\Alg(\cat) \times_{\Alg(\PrL)} \vEnr$ the full subcategory of the pullback along the functor $\PSh: \Alg(\cat) \to \Alg(\PrL)$ on those $\PSh(V)$-enriched categories for $V\in \Alg(\cat)$ whose graph factors through $V\subseteq \PSh(V)$. 
	Let $\Enr_{\Alg(\cat)}$ denote the full subcategory of $\vEnr_{\Alg(\cat)}$ on the univalent enriched categories.
%
	\end{defin}

	\begin{lemma}\label{lem:univalizationrestricts}The univalization functor $u:\Alg(\cat) \times_{\Alg(\PrL)} \vEnr \to \Alg(\cat) \times_{\Alg(\PrL)} \Enr$ pulled back from \cref{cor:EnrcoCart} restricts to a functor between the full subcategories $\vEnr_{\Alg(\cat)} \to \Enr_{\Alg(\cat)}$ left adjoint to the inclusion. 
	\end{lemma}
	\begin{proof}
		It suffices to show that if $V \in \Alg(\cat)$ and $\cC \in \vcat(\PSh(V))$ has graph factoring through $V$, then so does $u \cC$. But by definition $\ob \cC \to \ob (u \cC)$ is surjective, so this follows by applying orthogonality of surjective and fully faithful functors to find a filler for the following commutative square:
		\[\begin{gathered}[b] 
			\begin{tikzcd}
				\ob \cC \times \ob \cC \arrow[d, "\text{surj}"] \arrow[r] & V \arrow[d, hook] \\
				\ob (u \cC) \times \ob (u \cC) \arrow[r] \arrow[ur, dashed] & \PSh(V)
			  \end{tikzcd}\\[-\dp\strutbox]
		\end{gathered}\qedhere \]
	\end{proof}

	\begin{lemma}
		\label{prop:stilltwosided}
		The functor $\vEnr_{\Alg(\cat)} \to \Spaces \times \Alg(\cat)$ is a two-sided fibration and the functor $\Enr_{\Alg(\cat)} \to \Alg(\cat)$ is a coCartesian fibration. Moreover, $u: \vEnr_{\Alg(\cat)} \to \Enr_{\Alg(\cat)}$ is a map of coCartesian fibrations over $\Alg(\cat)$.
		In particular, we obtain straightened functors \begin{align*}
		\vcatsmall_{-}(-) &: \Spaces^{\op} \times \Alg(\cat) \to \widehat{\cat}\\\vcatsmall(-)&: \Alg(\cat) \to \widehat{\cat}\\\catsmall(-)&: \Alg(\cat) \to \widehat{\cat}. \end{align*} 
		and a univalization natural transformation $\vcatsmall \To \catsmall$.
	\end{lemma}
	
	\begin{proof}
		Pullbacks of two-sided/coCartesian fibrations remain two-sided/coCartesian fibrations which together with~\cref{cor:vEnrtwosided} and~\cref{cor:EnrcoCart} shows that $\Alg(\cat) \times_{\Alg(\PrL)} \vEnr$ is a two-sided fibration and $\Alg(\cat)\times_{\Alg(\PrL)} \Enr$ is a coCartesian fibration. 
		To show that the  full subcategories $\vEnr_{\Alg(\cat)}$ and  $\Enr_{\Alg(\cat)}$ remain two-sided resp.\ coCartesian fibrations, and univalization preserves coCartesian morphisms, by \cref{lem:coCartrefl} it suffices to show that for $f: V \to V'$ a monoidal functor the change-of-enrichment $f_!$ preserves the property of a graph factoring through $V$ which follows from  \cref{constr:changeofenrquiv}, the same holds for Cartesian transport along a map of spaces $g:X\to Y$ by \cref{obs:coCartMModVDesc}, and $u$ is a relative  left adjoint to the inclusion $\Enr_{\Alg(\cat)} \subseteq \vEnr_{\Alg(\cat)}$.
	\end{proof}


		The definition of enriched categories and valent enriched categories  in \cite{haugseng} gives rise to functors $\vcat^{GH}(-): \Alg(\cat) \to \widehat{\cat}$ and $\cat^{GH}(-): \Alg(\cat) \to \widehat{\cat}$, in loc.cit.\ denoted by $\Alg_{\mathrm{cat}}(-)$ and $\cat(-)$, respectively. We denote their coCartesian unstraightenings by $\vEnr^{GH}_{\Alg(\cat)} \to \Alg(\cat)$ and $\Enr^{GH}_{\Alg(\cat)} \to \Alg(\cat)$.

	\begin{cor}\label{cor:GHcomparisonsmall} There are equivalences of coCartesian fibrations over $\Alg(\cat)$:  \[\vEnr_{\Alg(\cat)}\simeq \vEnr^{GH}_{\Alg(\cat)} \hspace{0.75cm} \text{ and } \hspace{0.75cm} \Enr_{\Alg(\cat)} \simeq \Enr^{GH}_{\Alg(\cat)}\]	
	In particular, there are equivalences of functors $\Alg(\cat) \to \widehat{\cat}$ between \[\vcatsmall(-) \simeq \vcat^{GH}(-)\hspace{0.75cm} \text{ and } \hspace{0.75cm} \catsmall(-) \simeq \cat^{GH}(-)\;.\] 
	\end{cor}	
	\begin{proof}  Let $\vEnr_{\Alg(\cat)}^{GH} \to \Alg(\cat)$ denote the coCartesian unstraightening of $\vcat^{GH}: \Alg(\cat) \to \widehat{\cat}$ and let $\vEnr^{GH}_{\Alg(\widehat{\cat})} \to \Alg(\widehat{\cat})$ denote its large variant allowing for large monoidal categories but only small spaces of objects. \cref{thm:functorialcomparison} shows that the pullback along $\Alg(\PrL) \to \Alg(\widehat{\cat})$ induces an equivalence $\vEnr \simeq \Alg(\PrL) \times_{\Alg(\widehat{\cat})} \vEnr^{GH}_{\Alg(\widehat{\cat})}$ of coCartesian fibrations over $\Alg(\PrL)$. Pulling back further along $\PSh: \Alg(\cat) \to \Alg(\PrL)$ we thus obtain an equivalence between $ \Alg(\cat) \times_{\Alg(\PrL)} \vEnr$ and the pullback of 
	\[ \Alg(\cat) \overset{\PSh}{\to} \Alg(\widehat{\cat}) \leftarrow \vEnr^{GH}_{\Alg(\widehat{\cat})}
	\]
	which classifies the functor $\vcat^{GH}(\PSh(-)): \Alg(\cat) \to \widehat{\cat}$. From the definition it follows directly that  $\vcat^{GH}(-)$ is the full subfunctor of $\vcat^{GH}(\PSh(-))$ on those $\PSh(V)$-enriched categories whose graph factors through $V\subseteq \PSh(V)$. Hence, it follows from \kerodon{01VB} that $\vEnr^{GH}_{\Alg(\cat)}$ is a full subcategory of this pullback and agrees with the full subcategory of $\vEnr_{\Alg(\cat)} \subseteq  \Alg(\cat) \times_{\Alg(\PrL)} \vEnr$ from \cref{defin:erichedalg}. 
	
	In order to extend this statement to univalent enriched categories, it suffices to note that in the Gepner-Haugseng picture, $\cC \in \vcat^{GH}(V)$ is univalent iff the corresponding $\PSh(V)$-enriched category is. This is immediate for instance by the characterization of univalence in \cite[Prop. 5.4.2]{haugseng} as locality with respect to $E^1 \to E^0$.
		\end{proof}

	\subsection{Tensor product for enrichment in small monoidal categories}
	\label{subsec:smalltensor}

	\begin{prop}\label{prop:constructsmallexterior}
	There are unique symmetric monoidal structures on $\vEnr_{\Alg(\cat)}$ and $\Enr_{\Alg(\cat)}$ for which the full inclusions \[\vEnr_{\Alg(\cat)} \to \Alg(\cat) \times_{\Alg(\PrL)} \vEnr\hspace{1cm}\Enr_{\Alg(\cat)} \to \Alg(\cat) \times_{\Alg(\PrL)} \Enr\] from \cref{defin:erichedalg}  are symmetric monoidal for the symmetric monoidal structures on $\vEnr$ and $\Enr$ from \cref{defin:exteriortensor}, \cref{def:univalizedtensor}. Moreover, univalization restricts to a symmetric monoidal functor 
	\[\vEnr_{\Alg(\cat)} \to \Enr_{\Alg(\cat)}
	\]
	and the projections $\vEnr_{\Alg(\cat)} \to \Alg(\cat)$ and $\Enr_{\Alg(\cat)} \to \Alg(\cat)$ are coCartesian fibrations of operads.
	\end{prop}
	\begin{proof} Since $ \cat \to \PrL$ is symmetric monoidal and $\CAlg(\widehat{\cat}) \to \widehat{\cat}$ preserves limits, the categories $\Alg(\cat) \times_{\Alg(\PrL)} \vEnr$ and $\Alg(\cat) \times_{\Alg(\PrL)} \Enr$ inherit symmetric monoidal structures from the external tensor product. The monoidal unit $B* \in \vcat(\Spaces)$ has graph factoring through $* \subseteq \PSh(*) = \Spaces$. Moreover, if $V, W \in \Alg(\cat)$ and  $\cC \in \vcat(\PSh(V)), \cD \in \vcat(\PSh(W))$ with graphs factoring through $V\subseteq \PSh(V)$, resp. $W \subseteq \PSh(W)$, then $\cC \boxtimes \cD \in \vcat(\PSh(V) \otimes \PSh(W))\simeq \vcat(\PSh(V \times W))$ has graph factoring through $V\times W \subseteq \PSh(V \times W)$ by \cref{obs:formulagraphext}. Hence, the symmetric monoidal structure on $\Alg(\cat) \times_{\Alg(\PrL)} \vEnr$ restricts to the full subcategory $\vEnr_{\Alg(\cat)}$ and similarly for $\Enr_{\Alg(\cat)}$.
		
	Since univalization $\vEnr \to \Enr$ is a symmetric monoidal localization by \cref{prop:univalenttensor} which pulls back to a symmetric monoidal functor $\Alg(\cat) \times_{\Alg(\PrL)} \vEnr \to \Alg(\cat) \times_{\Alg(\PrL)} \Enr$, it restricts by \cref{lem:univalizationrestricts} to a symmetric monoidal functor between the full symmetric monoidal subcategories $\vEnr_{\Alg(\cat)} \to \Enr_{\Alg(\cat)}$.
	
	As by \cref{lem:coCartOpPr} $\vEnr \to \Alg(\PrL)$ is a coCartesian fibration of operads, so is the pullback $\Alg(\cat) \times_{\Alg(\PrL)} \vEnr \to \Alg(\cat)$. To show this restricts to a coCartesian fibration of operads $\vEnr_{\Alg(\cat)} \to \Alg(\cat)$ which we already know is a symmetric monoidal subcategory, it suffices to show that it is closed under coCartesian transports. This follows from their description on graphs in \cref{prop:changeofenragree} since for $f: V \to W$ in $\Alg(\cat)$, the induced functor $\PSh(V) \to \PSh(W)$ sends representables to representables.
	An analogous proof shows the statement for $\Enr_{\Alg(\cat)}$.
	\end{proof}

\begin{cor}\label{cor:laxsmall}
There are unique lax symmetric monoidal structure on the functors \[\vcatsmall: (\Alg(\cat), \times) \to (\lcat, \times) \hspace{1cm} \catsmall: (\Alg(\cat), \times) \to (\lcat, \times) \] for which the fully faithful natural transformations $\vcatsmall(-)\To \vcat(\PSh(-)): \Alg(\cat) \to \lcat$ and $\catsmall(-) \To \cat(\PSh(-))$ are symmetric monoidal, for the lax symmetric monoidal structure on $\vcat(\PSh(-))$ by \cref{cor:laxsymmetric} and on $\cat(\PSh(-))$ induced by \cref{cor:catpreservescolims} respectively.
Moreover, the univalization natural transformation $\vcatsmall \To \catsmall$ inherits a symmetric monoidal structure. 
\end{cor}
\begin{proof}
The first statement follows from \cref{prop:constructsmallexterior} after we straighten according to \cref{cor:coCartoperadscrit}. The second statement follows from \cref{cor:coCartoperadscrit}  since $\vEnr_{\Alg(\cat)} \to \Enr_{\Alg(\cat)}$ is symmetric monoidal by \mbox{\cref{prop:constructsmallexterior}} and a map of coCartesian fibrations by \cref{cor:EnrcoCart}. 
\end{proof}

	\begin{defin}
		\label{defin:smallinternalOmonoidal}
		Given an operad $O$, applying the lax symmetric monoidal structure from \cref{cor:laxsmall} we obtain functors 
		\begin{align*}\vcatsmall(-)&: \Alg_{O\otimes \E_1}(\cat) =  \Alg_O(\Alg_{\E_1}(\cat)) \to \Alg_O(\lcat).
		\\\catsmall(-)&: \Alg_{O\otimes \E_1}(\cat) =  \Alg_O(\Alg_{\E_1}(\cat)) \to \Alg_O(\lcat).
		\end{align*}
		For $V\in \Alg_{O\otimes \E_1}(\cat)$, we refer to the induced $O$-monoidal structure on $\vcatsmall(V)$ and $\catsmall(V)$ as the  \emph{internal $O$-monoidal structure}. 
	\end{defin}
	By definition, this $O$-monoidal structure on $\vcatsmall(V)$ agrees with the restriction of the internal $O$-monoidal structure on $\vcat(\PSh(V))$ from \cref{constr:internaltensor} to the full subcategory $\vcatsmall(V)$ on those $\PSh(V)$-enriched categories whose graph factors through $V\subseteq \PSh(V)$.

	We will now show that this restricted symmetric monoidal structure is Cartesian  and will then conclude that under the equivalence from \cref{cor:GHcomparisonsmall} it agrees with the external tensor product defined in \cite{haugseng,haugseng2023tensor}.

	\begin{lemma}\label{lem:semiadditive} Let $C \in \Alg(\Catcolim)$, i.e.\ a monoidal category with colimits  whose monoidal product preserves colimits separately in both variables. If $C$ is semiadditive (see \SAG{Def.}{C.4.1.6}), the functor $\RMod_{-}(C): \Alg(C) \to \widehat{\cat}$ preserves finite products. 
	\end{lemma}
	\begin{proof}
	For algebras $A, B \in \Alg(C)$, write $A\times B \in \Alg(C)$ for their product and $A\underline{\times} B \in \Alg(C \times C)$ for the corresponding algebra in $C\times C$. 
	By \HA{Thm.}{4.8.5.16}(2), the canonical functor $\RMod_{A\underline{\times} B} (C\times C) \to \RMod_A(C) \times \RMod_B(C)$ is an equivalence. 
	Composed with this equivalence, the canonical projection $\RMod_{A\times B}(C) \to \RMod_A(C) \times \RMod_B(C)$ becomes the  composite of left adjoints
	\[\RMod_{A \times B}(C) \rightleftarrows \RMod_{(A\times B) \underline{\times} (A\times B)} (C\times C) \rightleftarrows \RMod_{A\underline{\times} B}(C\times C).
	\]
	Here, the first adjunction is induced by the symmetric monoidal diagonal functor $C \to C \times C$ with (necessarily lax symmetric monoidal) right adjoint $-\times-: C \times C \to C$, and the second is given by induction along the algebra map 
	$(A\times B) \underline{\times} (A\times B) \to (A \underline{\times} B)$ induced by the projections $A\times B \to A$ and $A\times B \to B$. 
	Given $M \in \RMod_{A\times B}(C)$, the unit of the overall adjunction is the morphism \[M \to \left(M \otimes_{A\times B} A\right)  \times \left( M \otimes_{A\times B} B\right) \simeq \left(M \otimes_{A\times B} (A \times B)\right) \simeq M\, ,\] where the second equivalence uses semiadditivity to distribute $\otimes $ over $\times$. Hence, the overall left adjoint is fully faithful. We will now conclude by proving that the overall right adjoint is conservative. 
	Indeed, the overall right adjoint sends  $M_1 \in \RMod_A(C)$ and $M_2 \in \RMod_{B}(C)$  to the object $M_1\times M_2 \in \RMod_{A\times B} (C)$. Since $\RMod_{A\times B}(C) \to C$ is conservative, we conclude by noting that the object $M_1\times M_2 \in C$ has $M_1$ and $M_2$ as retracts by semiadditivity. 
	\end{proof}	
	
	\begin{cor}
		\label{lem:vcatXpreservesproducts}
		Fixing a space $X \in \Spaces$, the following functors preserve finite products:
		\begin{enumerate}[(1)]
		\item The functor $\vcat_X(-): \Alg(\PrL) \to \widehat{\cat}$ that is classified by the two-sided fibration $\vEnr \to \Spaces \times \Alg(\PrL)$ from \cref{cor:vEnrtwosided};
		\item The functor $\vcatsmall_X(-): \Alg(\cat) \to \widehat{\cat}$ that is classified by the two-sided fibration $\vEnr_{\Alg(\cat)} \to \Spaces \times \Alg(\cat)$ from \cref{defin:erichedalg} and \cref{prop:stilltwosided}.	\end{enumerate}
	\end{cor}
	\begin{proof}	Using semiadditivity of $\PrL$, \cref{lem:semiadditive} implies that $\Pr_{\cV} \times \PrW \simeq \Pr_{\cV \times \cW}$. Moreover, if $F: \PSh(X) \otimes \cV \to \cM$ and $G:\PSh(Y) \otimes \cW \to \cN$ are colimit-dominant and internally left adjoint, so is their product: Indeed, by semiadditivity, it suffices to prove this for the coproduct which follows from \cref{obs:cdomstab} and \cref{obs:iLstability}. Thus,  $\MModXb: \Alg(\PrL) \to \widehat{\cat}$ preserves products. 
		
	For the second statement, for $V, W\in \Alg(\cat)$ the canonical projections factor as
	\[\vcat_X(V \times W| \PSh(V \times W)) \to \vcat_X(V\times W | \PSh(V) \times \PSh(W)) \to \vcat_X(V|\PSh(V)) \times \vcat_X(W|\PSh(W))\;. 
	\]
	Here, the first functor is induced from change-of-enrichment along the morphism $\PSh(V\times W) \to \PSh(V) \times \PSh(W)$ in $\Alg(\PrL)$ induced from the projections $V\leftarrow V\times W \to W$, and the second functor is induced by the projections (and both functors restrict as indicated to full subcategories of enriched categories with graphs restricted to the full subcategories $V,W$ and $V\times W$).
	By \cref{prop:restrictedgraph} and statement (1), respectively, both functors are equivalences. 	\end{proof}
	
%

	\begin{theorem}
		\label{thm:pullbackcartesian}
		The external tensor product symmetric monoidal structure on $\vEnr_{\Alg(\cat)}$ from \cref{prop:constructsmallexterior} is Cartesian.
	\end{theorem}
	\begin{proof}
		By \HA{Def,}{2.4.0.1},  it suffices to show that the unit ${B *} \in \vEnr_{\Alg(\cat)}$ is terminal, and for $\cC, \cD \in \vEnr_{\Alg(\cat)}$ the induced projections $\cC \simeq \cC \boxtimes {B *} \leftarrow \cC \boxtimes \cD \rightarrow {B *} \boxtimes \cD \simeq \cD$ exhibit $\cC \boxtimes \cD$ as the Cartesian product of $\cC$ and $\cD$. 
		
		Terminality of $B*$ (represented by the marked module $* \to \Spaces$) follows since it lies over the terminal objects in $\Spaces$ and in $\Alg(\cat)$ and since the fiber $\vcat_*(*| \Spaces) \simeq *$ is contractible.
		
		For products, since $\vcatsmall_X: \Alg(\cat) \to \lcat$ preserves products by \cref{lem:vcatXpreservesproducts}, it follows from \cite[Prop.\ 3.1]{haugseng2023tensor} that for $X,Y \in \Spaces$ and $V, W \in \Alg(\cat)$ with $\cV:= \PSh(V)$ and $\cW:= \PSh(W) \in \Alg(\PrL)$, an object $Q\in \vcat_{X\times Y} (V\times W | \PSh(V \times W))$ represents the product of objects $\cC \in \vcat_X(V|\cV)$ and $\cD \in \vcat_Y(W| \cW)$ if and only if $(\cC, \cD)$ and $Q$ have the same image under the functors
		\[\begin{tikzcd}[column sep= 5pt]
		 \vcat_{X}(V| \cV) \times \vcat_Y(W | \cW) \arrow[r, "(\pi^*_X{,} \pi^*_Y)"] &  \vcat_{X\times Y} (V|\cV) \times \vcat_{X\times Y} (W | \cW)\\
		&  \arrow[u, "((\pi_{\cV})_!{,} (\pi_{\cW})_!)"]\vcat_{X\times Y} (V \times W | \PSh(V\times W)) = \vcat_{X\times Y}(V\times W | \cV \otimes \cW)
		\end{tikzcd}
		\] 
		where the horizontal functor is induced by Cartesian transport along the projections $ \pi_X: X\times Y \to X$ and $\pi_Y: X\times Y \to Y$ and the right functor arises from coCartesian transport (i.e.\ change-of-enrichment) along $\pi_{\cV}:= \PSh(\pi_V: V\times W\to V): \cV \otimes \cW = \PSh(V\times W) \to \PSh(V)$  and $\pi_{\cW}: \cV \otimes \cW \to \cW$. 
		Thus, it  suffices to show that the canonical morphism $\pi_{\cV,!}(\cC \boxtimes \cD) \to \pi^*_X\cC$ is an isomorphism in $\vcat_{X\times Y} (V | \PSh(V))$. 
		
		Using that $\vcat_{X\times Y} (V | \PSh(V)) \to \Fun((X\times Y) \times (X\times Y), \PSh(V))$ is conservative, it suffices to compute the above morphism at the level of graphs. By \cref{constr:changeofenrquiv},  $\pi_{\cV,!}$ is given by postcomposition with $\pi_{\cV}$, and by \cref{ex:carttransportgraph},  $\pi_X^*$ by pre-composition with $\pi_X$.
		
		By \cref{obs:formulagraphext}, the graph of  $\cC\boxtimes \cD$ in $\PSh(V) \otimes \PSh(W)$ is given by the `pure tensor' $\yo(\Hom_{C}(x,x')) \otimes \yo(\Hom_{\cD}(y, y'))$ and thus the image of $(\Hom_{\cC}(x,x'), \Hom_{\cD}(y, y')) \in V \times W \to \PSh(V \times W) \simeq \PSh(V) \otimes \PSh(W)$. Thus, the graph of $\pi_{\cV, !} (\cC \boxtimes \cD)$ factors through $V\subseteq \PSh(V)$ and is given by $\Hom_{\cC}(x,x')$ which agrees with the graph of $\pi_X^* \cC$. 
		\end{proof}
	
	\begin{obs}
		\label{rem:externaltensorsmallunivalent}
		Since the full inclusion $  \Enr_{\Alg(\cat)}\subseteq \vEnr_{\Alg(\cat)} $ admits a left adjoint, it preserves limits, so the Cartesian monoidal structure restricts to univalent enriched categories. This is unlike the presentable case in \cref{prop:univalenttensor}, where we needed to univalize.
	\end{obs}

	\begin{notat}
		Denote by $\Monlax$ the category of monoidal categories and lax monoidal functors, constructed as the full subcategory of $\OpAss$ on those non-symmetric operads $\cV_0^\oast \to \Delta^{\op}$ that are coCartesian fibrations and hence monoidal categories. This contains $\Alg(\cat)$ as the subcategory spanned by all objects and (strong) monoidal functors.
	\end{notat}


	In~\cite[Cor. 3.2]{haugseng2023tensor}, the external tensor product of enriched categories is defined as the Cartesian monoidal structure on $\vEnr^{GH}_{\Monlax} \to \Monlax$.

	\begin{cor}
		\label{cor:GHmonoidalCartagrees}
		Pulling back the external tensor product on $\vEnr^{GH}_{\Monlax}$ from~\cite[Cor. 3.2]{haugseng2023tensor}  to  $\Alg(\cat) \times_{\Monlax} \vEnr^{GH}_{\Monlax} \overset{\text{Cor.~\ref{cor:GHcomparisonsmall}}}{\simeq} \vEnr_{\Alg(\cat)}$ agrees with the external tensor product constructed in \cref{prop:constructsmallexterior}. 
	\end{cor}
	\begin{proof}
	In \cite[Cor. 3.2]{haugseng2023tensor}, the external tensor product on $\vEnr^{GH}_{\Monlax}$ is defined as the Cartesian product. 	Since the subcategory inclusion $\Alg(\cat) \to \Monlax$ admits a left adjoint (the enveloping monoidal category, see \cite[Def. A.1.1]{haugseng})  it preserves limits and hence the Cartesian monoidal structure on $\vEnr^{GH}_{\Monlax}$ pulls back to a Cartesian monoidal structure on $\Alg(\cat) \times_{\Monlax} \vEnr^{GH}_{\Monlax}$ which we identify in \cref{cor:GHcomparisonsmall} with   $\vEnr_{\Alg(\cat)}$. By \cref{thm:pullbackcartesian}, the external tensor product therefore agrees with our external tensor product constructed in \cref{prop:constructsmallexterior}. 
		\end{proof}

	\begin{cor}\label{cor:smallinternalagree}
	The lax symmetric monoidal structure on $\vcat^{GH}(-): \Alg(\cat) \to \lcat$ from \cite[Prop. 4.3.11]{haugseng} agrees under the equivalence $\vcat^{GH}(-) \simeq \vcatsmall(-)$  from \cref{cor:GHcomparisonsmall} with the lax symmetric monoidal structure on $\vcatsmall$ in \cref{cor:laxsmall}. Similarly, the lax symmetric monoidal structure on $\cat^{GH}(-): \Alg(\cat) \to \lcat$ from \cite[Cor.~5.7.11]{haugseng} agrees under the equivalence $\cat^{GH}(-) \simeq \catsmall(-)$  from \cref{cor:GHcomparisonsmall} with the lax symmetric monoidal structure on $\catsmall$ constructed in \cref{cor:laxsmall}.
	\end{cor}
	\begin{proof}
	The valent statement is immediate from \cref{cor:GHmonoidalCartagrees}. For the univalent statement, note that univalization  $\vcat^{GH} \To \cat^{GH}(-)$ defines a symmetric monoidal natural transformation by a generalization of \cite[Prop.~5.7.14]{haugseng}. Since $\vcat^{GH}(V) \to \cat^{GH}(V)$ is a localization for every $V\in \Alg(\cat)$, this uniquely determines the lax symmetric monoidal structure on $\cat^{GH}(-)$ which hence by \cref{cor:GHcomparisonsmall} and \cref{cor:laxsmall} agrees with the one on $\catsmall(-)$ from \cref{cor:laxsmall}.\end{proof}

	\begin{cor}\label{cor:smalloperad}

	In particular, if $O$ is an operad and $V\in \Alg_O(\Alg(\cat))$ a small $O\otimes\E_1$-monoidal category, then the $O$-monoidal structures on $\vcatsmall(V)$ and $\catsmall(V)$ from \cref{defin:smallinternalOmonoidal} agree with the ones defined in  \cite[Cor.~4.3.12]{haugseng} and \cite[Cor.~5.7.12]{haugseng}.	\end{cor}

			\section{Comparing tensor products}
			\label{sec:comptensor}

		In \cref{cor:laxsymmetric} we have constructed a lax symmetric monoidal structure on the functor $\vcat(-): \Alg(\PrL) \to \lcat$. In \cref{cor:laxsmall} we have shown how this induces a lax symmetric monoidal structure on the small variant $\vcatsmall: \Alg(\cat) \to \lcat$ from \cref{prop:stilltwosided} and in \cref{cor:smallinternalagree} that this lax symmetric monoidal structure agrees with the one  constructed in \cite{haugseng, haugseng2023tensor}. In this final section, we will show that our original lax symmetric monoidal structure on $\vcat(-): \Alg(\PrL) \to \lcat$ from \S \ref{sec:multiplicativity} also agrees with the one constructed in \cite{haugseng, haugseng2023tensor}. More precisely, we will prove that restricting a large variant  $\vcatlarge:(-): \Alg(\lcat) \to \llcat$ of the lax symmetric monoidal functor  from \cref{prop:stilltwosided} along the lax symmetric monoidal functor $\Alg(\PrL) \to \Alg(\lcat)$ is equivalent to the  lax symmetric monoidal functor $\vcat(-): \Alg(\PrL) \to \lcat$ from \S \ref{sec:multiplicativity}.
		
		To do so, we need to develop some universe-hopping. 
						
	\subsection{Large presheaves and ind completion}
	Recall that throughout this paper we work with Grothendieck universes and have fixed from the outset uncountable inaccessible cardinals $\tav < \hat{\tav} < \doublehat{\tav}$, and call cardinals $\kappa$ small, large and very large if $\kappa<\tav, \kappa< \hat{\tav}$ and $\kappa<\doublehat{\tav}$, respectively. Replacing $\tav$ by $\hat{\tav}$ throughout, we obtain large variants of the definitions, constructions and statements in this paper. We follow the convention to denote by a $\widehat{-}$ the large analogs 
	 of categorical constructions. For example, we write $\lPrL$ for the category of \emph{large-presentable categories}: Very large categories that are locally large, admit large colimits, and can be written as $\operatorname{Ind}_\kappa(\ccC)$ for $\kappa$ a large regular cardinal and $\ccC$ a large category.

	For a presentable category, there are two common ways to produce a large-presentable category, compare \cite[Sections 3.11, 3.12]{kelly}:

	\begin{notat} We let \[\widehat{\PSh}(-):= \Fun(-^{\op}, \widehat{\Spaces}):  \widehat{\cat} \to \lPrL\] denote the \emph{large presheaf category}.
	
For $C \in \Catcolim$ (the category of large categories admitting small colimits and small-colimit-preserving functors), we let $\Ind_{\tav}(C)$ denote the full subcategory of $\widehat{\PSh}(\cV)$ on the functors that preserve small limits. By \HTT{Prop.}{5.3.6.2} this is a relative presheaf category, universally described as adjoining large colimits in a way preserving small colimits, and therefore also agrees with the free cocompletion of $C$ under $\tav$-filtered colimits. This assembles into a functor \[\lInd: \Catcolim \to \lPrL.\]
	\end{notat}

	\begin{ex} By the universal property of (relative) presheaf categories, it follows that $\hat{\Spaces} = \hat{\PSh}(*) \simeq \lInd(\PSh(*)) = \lInd(\Spaces)$.
		Similarly $\lInd(\Set) \simeq \widehat{\Set}, \lInd(\Spaces_*) \simeq \hat{\Spaces}_*, \lInd(\Sp) \simeq \widehat{\Sp}$ and so on.
	\end{ex}

	\begin{lemma}\label{lem:largeind}
	The following holds:
	\begin{enumerate}[(1)]
	\item $\widehat{\PSh}: \lcat \to \lPrL$ is symmetric monoidal. Therefore, the composite $\PrL \to \lcat \to \lPrL$ (which we will henceforth also denote by $\widehat{\PSh}$) is lax symmetric monoidal;
	\item $\lInd: \PrL \to \lPrL$ is symmetric monoidal and preserves small colimits, colimit-dominant functors and fully faithful functors;
	\item For $\cV \in \Alg(\PrL)$, the induced functor $\lInd: \PrV \to \lPrL_{\lInd(\cV)}$ preserves internal left adjoints;
	\item There is a symmetric monoidal natural transformation between lax symmetric monoidal functors $L:\widehat{\PSh} \To \lInd : \PrL \to \lPrL$. For $\cC \in \PrL$ the functor $L_C: \widehat{\PSh}(\cC) \to \lInd(\cC)$ is  left adjoint to the inclusion $\lInd(C) \subseteq \widehat{\PSh}(C)$.
	\end{enumerate}
	\end{lemma}
	\begin{proof}
		Let $\doublewidehat{\cat} {}^{\text{large colim}} \hookrightarrow \doublewidehat{\cat} {}^{\text{small colim}}$ be the inclusion of  the category of very large categories with large colimits into the category of very large categories with small colimits. By \HTT{Prop.}{5.3.6.2} it admits a left adjoint which extends the functor $\lInd: \Pr \to \lPrL$. But small colimits in $\Pr, \lPrL$ can be computed in  $\doublewidehat{\cat} {}^{\text{small colim}}$ and $\doublewidehat{\cat} {}^{\text{large colim}}$, respectively, since by the proof of \HA{Lem.}{4.8.4.2} $\Catcolim \subseteq \doublewidehat{\cat} {}^{\text{small colim}}$ is closed under small colimits\footnote{In fact, $\Catcolim$ consists of precisely the $\hat{\tav}$-compact objects in $\doublewidehat{\cat} {}^{\text{small colim}}$ and is hence closed under large colimits.}. Hence they are preserved by $\lInd:\Pr \to \lPrL$.
		
		$\lInd$ preserves fully faithful functors by \HTT{Prop.}{5.3.5.11}. Also given a \mbox{(small-)}colimit-dominant functor $F: \cM \to \cN$ in $\Pr$, then the image of $\lInd(F): \lInd(\cM) \to \lInd(\cN)$ contains $\operatorname{Im}(F) \subseteq \cN \subseteq \lInd(\cN)$. But since $\operatorname{Im}(F)$ generates $\cN$ under small colimits and $\cN$ generates $\lInd(\cN)$ under large colimits, this means $\lInd(F)$ is (large-)colimit-dominant in $\lPrL$.
		
		Since $\lInd: \Pr \to \lPrL$ is symmetric monoidal, it induces a functor $\PrV \to \lPrL_{\lInd(\cV)}$ as claimed. If $F: \cM \to \cN$ in $\PrV$ is internally left adjoint, our first observation is that the underlying functor of $\lInd(F): \lInd(\cM) \in \lInd(\cN)$ is left adjoint to $\lInd(F^\rR)$ (which is well-defined since $F^\rR$ preserves large colimits by assumption). This follows by restricting the evident adjunction $\lPSh(F) \dashv \lPSh(F^\rR)$ to small-limit-preserving presheaves. Hence, it suffices to show that the canonical lax $\lInd(\cV)$-linear structure on $\lInd(F)^{\rR} \simeq \lInd(F^\rR): \lInd(\cN) \to \lInd(\cM)$ is strong. But we already know this functor preserves large colimits, so we can reduce to checking its compatibility with the $\cV$-tensoring on $\cN$ where it agrees with $F^\rR$.
	
	By \HA{Rem.}{4.8.1.8}, \HA{Lem.}{4.8.4.2}, there exists a functor from the poset $[2] \to \Alg(\doublewidehat{\Pr})$ establishing the (relative) cocompletion functors \[\llcat \overset{\PSh^{\tav\mathrm{-rex}}}{\to}\doublewidehat{\cat} {}^{\text{small colim}} \overset{\lInd}{\to} \doublewidehat{\cat} {}^{\text{large colim}}\] as symmetric monoidal and their right adjoint subcategory inclusions as lax monoidal. Restricting to the full subcategories $\lcat \subseteq \llcat$, $\PrL \subseteq \doublewidehat{\cat} {}^{\text{small colim}}$ and $\lPrL \subseteq \doublewidehat{\cat} {}^{\text{large colim}}$ we obtain the desired (lax) symmetric monoidal structures on our functors. Moreover, letting $\iota: \doublewidehat{\cat} {}^{\text{small colim}}  \to \llcat$ denote the forgetful functor, the counit  of the adjunction $\PSh^{\tav\mathrm{-rex}} \dashv \iota$ induces a symmetric monoidal transformation 
	\[ L: \lPSh \circ \iota \simeq \lInd \circ \PSh^{\tav\mathrm{-rex}} \circ \iota \Rightarrow \lInd\;. \]
For fixed $\cM \in \Pr$, the component $L_\cM : \lPSh(\cM) \to \lInd(\cM)$ is a map in $\lPrL$, namely the unique large-colimit-preserving extension of $\cM\to \lInd(\cM)$. Hence its right adjoint is the canonical inclusion $\lInd(\cM) \hookrightarrow \lPSh(\cM)$.
\end{proof}

	\subsection{From presentable enrichment to large-presentable enrichment }

	\begin{constr}
		\label{constr:lIndfunct}
		The symmetric monoidal inclusion $\Spaces \hookrightarrow \lSpaces$ together with the symmetric monoidal functor $\lInd: \PrL\to \lPrL$ induce a symmetric monoidal functor from $(\Spaces \times \Alg(\PrL)) \times_{\RMod(\PrL)} \Arr(\RMod(\PrL)) \times_{\Arr(\Alg(\PrL))} \Alg(\Pr)$ to its large analog. This functor restricts to a symmetric monoidal \emph{universe-enlargement-functor} \[\begin{tikzcd}
		\vEnr\arrow[d] \arrow[r, "\Ind_{\tav, !}"] & \arrow[d] \lvEnr\\
		\Alg(\PrL) \arrow[r, "\lInd"] & \Alg(\lPrL) \end{tikzcd}\] between full subcategories, as it sends by \cref{lem:largeind}  internally left adjoint and colimit-dominant functors to functors satisfying the large versions of these properties respectively.
			\end{constr}


\begin{prop}\label{prop:universeEnlargmentFibration}The induced functor  $\vEnr \to (\Spaces \times \Alg(\Pr)) \times_{(\Alg(\lPrL) \times \lSpaces)} \lvEnr$  is a map of two-sided fibrations over $\Spaces \times \Alg(\Pr)$. 
\end{prop}
\begin{proof}
The functor $\lInd: \Pr \to \lPrL$ is symmetric monoidal and preserves small colimits by \cref{lem:largeind}(2), so in particular relative tensor products. Since by  \cref{lem:largeind}(2) it also preserves the colimit-dominant/fully faithful factorization system, it follows from  the explicit description of coCartesian morphisms over $\Alg(\Pr)$ in \cref{ex:basechangerelative} and Cartesian morphisms over $\Spaces$ in \cref{obs:coCartMModVDesc}  that $\Ind_{\tav, !}$ is a map of two-sided fibrations. \end{proof}

\begin{prop}\label{prop:indlax}
The functor $\Ind_{\tav, !}$ from \cref{constr:lIndfunct} induces a symmetric monoidal natural transformation between lax symmetric monoidal functors:
\[\begin{tikzcd}\hphantom{\Alg(\PrL)}&\hphantom{\Alg(\lPrL)}&\hphantom{\llcat}\\
	\Alg(\PrL) & \Alg(\lPrL) & \llcat
	\arrow[from=2-1, to=2-2, "\lInd"']
	\arrow[""{name=0, anchor=center, inner sep=0}, "\vcat(-)", curve={height=-30pt}, from=2-1, to=2-3]
	\arrow[from=2-2, to=2-3,"\widehat{\vcat}(-)"']
	\arrow[ Rightarrow, shorten <=8pt, shorten >=-3pt, shift right=1,  from=1-2, to=2-2]
\end{tikzcd}\]
For every $\cV\in \Alg(\PrL)$, the induced functor $\vcat(\cV) \to \widehat{\vcat}(\lInd(\cV))$ is fully faithful with image those $\lInd(\cV)$-enriched categories with small space of objects and whose graph factors through $\cV \subseteq \lInd(\cV)$. 
\end{prop}
\begin{proof}
By \cref{prop:universeEnlargmentFibration}, the functor $\vEnr \to \Alg(\PrL) \times_{\Alg(\lPrL)} \lvEnr$ is a map of coCartesian fibrations over $\Alg(\PrL)$. Since  the square in \cref{constr:lIndfunct} is a square of symmetric monoidal categories, the functor is moreover symmetric monoidal. Hence, by \cref{cor:coCartoperadscrit}, this functor is a map of coCartesian fibrations of operads and therefore straightens to a symmetric monoidal natural transformation between lax symmetric monoidal functors $\Alg(\PrL) \to \llcat$. 

We now show that $\vcat(\cV) \to \widehat{\vcat}(\lInd(\cV))$ is fully faithful. Since this is a map of Cartesian fibrations over $\Spaces$, it suffices by  \kerodon{01VB} to show that $\vcat_X(\cV) \to \widehat{\vcat}_X(\lInd(\cV))$ is fully faithful for every fixed $X\in \Spaces$. 
 Being symmetric monoidal, $\lInd$ sends the left action of $\PrL$ on $\PrV$ to the left action of $\lPrL$ on $\widehat{\Pr}_{\lInd(\cV)}$ and hence induces as in \cref{constr:changeofenrquiv}  an algebra homomorphism between endomorphism algebras 
		\[ \Fun(X^{\op} \times X , \cV) \simeq \LinEnd(\PSh(X) \otimes \cV) \hookrightarrow \End^{\mathrm{L}}_{\lInd(\cV)}(\widehat{\PSh}(X) \otimes \lInd(\cV)) \simeq \Fun(X^{\op} \times X , \lInd(\cV)) \]
whose underlying functor is given by postcomposing with the full inclusion $\cV \hookrightarrow \lInd(\cV)$. In particular, the induced functor \[\vcat_X(\cV) \simeq \Alg(\LinEnd(\PSh(X) \otimes \cV)) \to \Alg(\End^{\mathrm{L}}_{\lInd(\cV)}(\widehat{\PSh}(X) \otimes \lInd(\cV))) \simeq \widehat{\vcat}_X(\lInd(\cV))\] is fully faithful. By an analogous argument to \cref{prop:changeofenragree}, this functor agrees with the induced map on fibers of $\vEnr \to (\Spaces \times \Alg(\Pr)) \times_{\lSpaces \times \Alg(\lPrL)} \lvEnr$ over $\Spaces \times \Alg(\Pr)$.
\end{proof}

\begin{lemma} The functor $\Ind_{\tav, !}: \vEnr \to \lvEnr$ preserves fully faithful functors and surjective-on-objects functors, and univalent enriched categories.
\end{lemma}
\begin{proof}
	The first claim follows from \cref{lem:largeind} $(2)$ since $\lInd: \PrL \to \lPrL$ preserves fully faithful functors. The second is clear as $\Ind_{\tav,!}$ does not change the underlying space of objects. For the last statement, use that the marking of $\Ind_{\tav, !} \cC$ factors through $\PShV(\cC) \subseteq \lInd(\PShV(\cC)) = \PSh_{\lInd(\cV)}(\Ind_{\tav, !} \cC)$.
\end{proof}

\begin{constr}\label{constr:IndUniv}
Since  $\Ind_{\tav, !}: \vEnr \to \lvEnr$ preserves fully faithful and surjective-on-objects functors, there is a unique symmetric monoidal functor $\Ind_{\tav,!, \mathrm{univ}}:\Enr \to \lEnr$ making the following square of symmetric monoidal categories and symmetric monoidal functors commute:
\[\begin{tikzcd}
\vEnr \arrow[d, "u"] \arrow[r] & \lvEnr\arrow[d, "u"] \\
\Enr \arrow[r] & \lEnr
\end{tikzcd}\]
Moreover, denoting the full inclusions $ \Enr \subseteq \vEnr$ and $\lEnr \subseteq \lvEnr$ by $\iota$, the induced natural transformation $ \iota \Ind_{\tav,!, \mathrm{univ}} \To \Ind_{\tav, !} \iota$ is invertible since $\Ind_{\tav,!}$ preserves univalent categories and hence factors through the full subcategories $\Enr \subseteq \vEnr$. In other words, the above square is vertically right adjointable (or equivalently, defines a map of adjunctions). 
\end{constr}


\begin{prop} \cref{constr:IndUniv} induces a square of coCartesian fibrations over $\Alg(\PrL)$:
\[\begin{tikzcd}
\vEnr \arrow[d, "u"] \arrow[r] & \Alg(\Pr) \times_{\Alg(\lPrL)} \lvEnr\arrow[d, "u"] \\
\Enr \arrow[r] &\Alg(\Pr) \times_{\Alg(\lPrL)} \lEnr
\end{tikzcd}\]
\end{prop}
\begin{proof}
All corners of the square are coCartesian fibrations over $\Alg(\PrL)$ by \cref{prop:coCart} and \cref{lem:coCartrefl}. 
By \cref{prop:universeEnlargmentFibration} the top horizontal map is a map of coCartesian fibrations and by \cref{cor:EnrcoCart}, the vertical maps are maps of coCartesian fibrations. By the description of coCartesian morphisms in \cref{lem:coCartrefl} it follows that the bottom horizontal map is a map of coCartesian fibrations.
\end{proof}

\begin{cor}\label{cor:ind-presentableImage}
There is a unique symmetric monoidal natural transformation $\cat(-) \To \widehat{\cat}(\lInd(-)): \Alg(\PrL) \to \llcat$ making the diagram of symmetric monoidal transformations between lax symmetric monoidal functors commute:
\[\begin{tikzcd}
\vcat(-) \arrow[r] \arrow[d,"u"] & \widehat{\vcat}(\lInd(-))\arrow[d, "u"] \\
\cat(-) \arrow[r] & \widehat{\cat}(\lInd(-))
\end{tikzcd}\]
For every $\cV \in \Alg(\PrL)$, the induced functor $\cat(\cV) \to \widehat{\cat}(\lInd(\cV))$ is fully faithful with image those univalent $\lInd(\cV)$-enriched categories whose underlying space are small and whose graphs factor through $\cV \subseteq \lInd(\cV)$. 
\end{cor}
\begin{proof}
The first part follows from \cref{constr:IndUniv} and \cref{cor:coCartoperadscrit}.
For the second part, since the square in \cref{constr:IndUniv} is vertically right adjointable, the functor $\cat(\cV) \to \widehat{\cat}(\lInd(\cV))$ is a restriction of the functor $\vcat(\cV) \to \widehat{\vcat}(\lInd(\cV))$ to the respective full subcategories. Hence, fully faithfulness and characterization of the image follow from \cref{prop:indlax}.
\end{proof}

\subsection{From large-monoidal enrichment to large-presentable enrichment}

\begin{notat}
Let $\widehat{\vcat}_{\text{small spc}}(-): \Alg(\lPrL) \to \llcat$ denote the full subfunctor of $\widehat{\vcat}(-)$ on those valent enriched-categories whose underlying space of objects is small. 

For $V\in \Alg(\lcat)$, let $\vcatlarge(V):= \widehat{\vcat}_{\text{small spc}}(V|\widehat{\PSh}(V))$ denote the large analog of \cref{def:smallenrichment} with small spaces of objects. 
 \end{notat}
By the large version of \cref{cor:laxsmall}, the functor $\vcatlarge(-): \Alg(\lcat) \to \llcat$ has a unique lax symmetric monoidal structure compatible with the full inclusion into the functor $\widehat{\vcat}(\widehat{\PSh}(-)): \Alg(\lcat) \to \Alg(\lPrL) \to \llcat$.

\begin{cor}\label{cor:vcatlargesubfunctor}
Composing the (large variant of the) defining fully faithful symmetric monoidal natural transformations from \cref{cor:laxsmall} with the one from \cref{lem:largeind}(4) results in a composite symmetric monoidal transformation between lax symmetric monoidal functors:
\[\begin{tikzcd} \vphantom{\Alg(\PrL)}& \vphantom{\Alg(\lcat)}& \vphantom{\Alg(\lPrL)} \arrow[d, shorten <=-2pt, shorten >=-4pt, shift right=1, Rightarrow]& \vphantom{\llcat} \\
	{\Alg(\PrL)} \arrow[rr, bend right = 50, "\lInd"'] \arrow[r] &  \Alg(\lcat)  \arrow[rr, bend left = 50, "\vcatlarge(-)"] \arrow[r, "\widehat{\PSh}"] & {\Alg(\lPrL)} \arrow[r,"\widehat{\vcat}(-)"'] & \llcat\\
	 \vphantom{\Alg(\PrL)}& \vphantom{\Alg(\lcat)} \arrow[u, shorten <=-5pt, shorten >=-2pt, Leftarrow]& \vphantom{\Alg(\lPrL)}& \vphantom{\llcat} 
\end{tikzcd}.\]
For $\cV \in \Alg(\PrL)$, the induced functor $\vcatlarge(\cV) \to \widehat{\vcat}(\lInd(\cV))$ is fully faithful with image those valent $\lInd(\cV)$-enriched categories with small space of objects and graph factoring through $\cV \subseteq \lInd(\cV)$. 

Analogously, composing the symmetric monoidal natural transformations $\catlarge(-) \To \widehat{\cat}(\widehat{\PSh}(-))$ from \cref{cor:laxsmall} with the one from \cref{lem:largeind}(4)  results in a  symmetric monoidal natural transformation $\catlarge(-) \To \widehat{\cat}(\lInd(-))$ for which $\catlarge(\cV) \to \widehat{\cat}(\lInd(\cV))$ is fully faithful with image those univalent $\lInd(\cV)$-enriched categories with small space of objects and graphs factoring through $\cV \subseteq \lInd(\cV)$. 
\end{cor}
\begin{proof}
For $\cV \in \Alg(\PrL)$, the component of the natural transformation is the functor \[\vcatlarge(\cV) := \widehat{\vcat}_{\text{small spc}}(\cV | \widehat{\PSh}(\cV)) \subseteq \widehat{\vcat}(\widehat{\PSh}(\cV)) \to \widehat{\vcat}(\lInd(\cV))\] where the latter functor is change of enrichment along the morphism $\widehat{\PSh}(\cV) \to \lInd(\cV)$ in $\Alg(\lPrL)$. By \cref{prop:restrictedgraph},  it follows that it is fully faithful with the prescribed image. The same argument applies in the univalent case.
\end{proof}

\begin{cor}\label{cor:centralcomparison}
 There is an equivalence between the lax symmetric monoidal functor from  \cref{cor:laxsymmetric} and the restriction to $\Alg(\PrL)$ of the one from  \cref{cor:laxsmall}:
 \[\vcat(-) :\Alg(\PrL) \to \lcat \hspace{0.5cm} \text{ and } \hspace{0.5cm}\Alg(\PrL) \to \Alg(\lcat) \overset{\vcatlarge(-)}{\longrightarrow} \llcat.
 \]
In the univalent case, there is an equivalence between the lax symmetric monoidal functor from \cref{cor:catpreservescolims} and the restriction of the one from  \cref{cor:laxsmall}:
\[ \Alg(\PrL) \overset{\cat(-)}{\longrightarrow} \lcat \hspace{0.5cm} \text{ and } \hspace{0.5cm}\Alg(\PrL) \to \Alg(\lcat) \overset{\catlarge(-)}{\longrightarrow} \llcat. \]
 \end{cor}
\begin{proof}
In the valent case, by \cref{prop:indlax} and \cref{cor:vcatlargesubfunctor}, respectively, both are lax symmetric monoidal subfunctors  of $\widehat{\vcat}(\lInd(-)): \Alg(\PrL) \to \llcat$ with the same image, and hence agree. In the univalent case, the same argument applies using \cref{cor:ind-presentableImage} and  \cref{cor:vcatlargesubfunctor}.
\end{proof}

Let $\widehat{\vcat} {}_{\text{small spc}}^{GH}: \Alg(\lcat) \to \lcat $ denote the (large variant with small spaces of objects of the) lax symmetric monoidal functor from \cite[Prop.~4.3.11]{haugseng}, and let $\widehat{\vcat} {}^{GH}_{\text{small spc}}|_{\Alg(\PrL)}: \Alg(\PrL) \to \Alg(\lcat) \to \lcat$ be its restriction along the lax symmetric monoidal functor $\Alg(\PrL)\to \Alg(\lcat)$ (also considered in \cite[Cor. 4.3.16]{haugseng}).
Similarly, let $\widehat{\cat} {}_{\text{small spc}}^{GH}: \Alg(\lcat) \to \lcat $ denote the (large variant with small spaces of objects of the) lax symmetric monoidal functor from \cite[Cor.~5.7.11]{haugseng} and $\widehat{\vcat} {}^{GH}_{\text{small spc}}|_{\Alg(\PrL)}: \Alg(\PrL) \to \Alg(\lcat) \to \lcat$ denote its restriction to $\Alg(\PrL)$.

\begin{cor}The lax symmetric monoidal functors $\widehat{\vcat} {}_{\text{small spc}}^{GH}|_{\Alg(\PrL)}(-): \Alg(\PrL) \to \lcat$ from  \cite[Prop.~4.3.11]{haugseng}  and $\vcat(-): \Alg(\PrL)\to \lcat$ from \cref{cor:laxsymmetric} are equivalent. Analogously, the lax symmetric monoidal functors $\widehat{\vcat} {}^{GH}_{\text{small spc}}|_{\Alg(\PrL)}: \Alg(\PrL) \to \Alg(\lcat) \to \lcat$ from~\cite[Cor.~4.3.16]{haugseng} and $\cat(-)$ from \cref{cor:catpreservescolims} are equivalent. \end{cor}
\begin{proof}
By \cref{cor:centralcomparison}, the functor $\vcat(-)$ is equivalent to the restriction of the  lax symmetric monoidal functor $\vcatlarge(-)$ along $\Alg(\PrL)\to \Alg(\lcat)$. By \cref{cor:smallinternalagree}, this lax symmetric monoidal functor $\vcatlarge(-)$ agrees with $\widehat{\vcat} {}_{\text{small spc}}^{GH}$. The same argument applies in the univalent case.
\end{proof}

\begin{cor}\label{cor:finalcomparisonGH}
For an operad $O$ and  $\cV \in \Alg_{O\otimes \E_1}(\PrL)$, the $O$-monoidal structure on $\vcat(\cV)$ from \cref{constr:internaltensor} agrees with the one from  \cite[Cor. 4.3.12]{haugseng} and the $O$-monoidal structure on $\cat(\cV)$ from \cref{def:univalizedtensor} agrees with the one from \cite[Cor.~5.7.12]{haugseng}. 
\end{cor}

	\appendix
	
	\section{Some facts about the Grothendieck construction}
	\label{sec:groth}

	For $B\in \cat$, the \emph{Grothendieck construction}, also called \emph{unstraightening}, constructed for example in \cite[\HTTsec{3.2}]{HTT}, is a functor $\Fun(B, \cat) \to \cat_{/B}$ sending an $F:B \to \cat$ to a functor denoted $\int_B F \to B$. 
	
	\begin{reminder} \label{reminder:grothadjoints} The Grothendieck construction $\Fun(B, \cat) \to \cat_{/B}$ factors as an equivalence through the subcategory $\mathrm{coCart}_{/B}$ of coCartesian fibrations and maps of coCartesian fibrations (combine \HTT{Thm.}{3.2.0.1} with \HTT{Prop.}{5.2.4.6}, compare also \cite{NikolausLax}).
	 Moreover, it follows from \cite[Thm. 4.5]{NikolausLax} that it has a left adjoint $(E \to B) \mapsto (b \mapsto E \times_B B_{/b})$ from which it follows that it also has a right adjoint  $(E \to B) \mapsto \Fun_{/B} (B_{/b} , E)$.
	 \end{reminder}
	


	We now collect some facts about Cartesian and coCartesian fibrations and the Grothendieck construction, which we will use in the remainder of the paper.

	\begin{lemma}
		\label{lem:weaklycontrslice}
		Given a functor $p:B \to C$ between categories, the forgetful functor $\pi: \ccC_{/p} \to \ccC$ creates colimits parametrized by weakly contractible categories.
	\end{lemma}
	\begin{proof}
	The functor $C_{/p} \to C$ is a left fibration (i.e.\ classifies a functor $C\to \Spaces \subseteq \cat$). It then follows from \HTT{Prop.}{4.3.1.12} that every weakly contractible diagram in $C_{/p}$ is a $\pi$-colimit diagram. The claim is now immediate by transitivity of relative colimits \HTT{Prop.}{4.3.1.5}(2).
	\end{proof}
	Unwinding  \HTT{Prop.}{4.3.1.12}, \cref{lem:weaklycontrslice} follows from the observation that for any diagram  $q: L \to C_{/p}$, the colimit of $L \to C_{/p} \to C$ is equipped with a cone to the colimit over $L$ of the constant diagram on $p$. If $L$ is weakly contractible, this colimit over the constant diagram on $p$ agrees with $p$ and hence the colimit of $L \to C_{/p} \to C$ lifts to $C_{/p}$.

	\begin{prop}
		\label{prop:grothfacts}
		Let $B$ be a category and $F: B \to \cat$.
		\begin{enumerate}[(1)]
			\item The composite  $\int_B: \Fun(B, \cat) \to \cat_{/B} \to \cat$  preserves and reflects weakly contractible limits.
			\item For $K$ any category, there is a canonical equivalence over $B$
			\[ \smallint\nolimits_{b \in B} \Fun(K, F(b)) \simeq B \times_{\Fun(K, B)} \Fun(K, \smallint\nolimits_B F) \; . \]
		\end{enumerate}
	\end{prop}
	\begin{proof}
		By \cref{reminder:grothadjoints}, $\smallint\nolimits_B : \Fun(B, \cat) \to \cat_{/B}$ is right adjoint and hence limit-preserving. It is also a subcategory inclusion (onto $\mathrm{coCart}_{/B}$) and the slice projection $\cat_{/B} \to \cat$ is conservative so both reflect the limits they preserve. Together with \cref{lem:weaklycontrslice}, this verifies the first claim.
		
		Finally, use the explicit expression for the left adjoint of unstraightening from \cref{reminder:grothadjoints}: It suffices to show that for any functor $E \to B$, there is a canonical equivalence
		\[\Map_{\Fun(B, \cat)} \left(E \times_B B_{/-}, \Fun(K, F-)\right) \simeq \Map_{/B} (E, B \times_{\Fun(K, B)} \Fun(K, \smallint\nolimits_B F))\]
		But the right hand side is equivalent to \[\Map_{/B} (K \times E, \smallint\nolimits_B F) \simeq \Map_{\Fun(B, \cat)}(K \times E \times_B B_{/-}, F)\] finishing our proof.
	\end{proof}


\begin{reminder}[{\cite[Def. 5.19]{haugseng2023two}, \cite[Prop.\ 2.1.9]{stefanichcorr}}]
		\label{reminder:twosidedfib}
		Given categories $C, D, E \in \cat$, a functor $p=(p_1, p_2): E \to C \times D$ is a \emph{two-sided fibration} if
		\begin{itemize}
			\item $p_2 : E \to D$ is a coCartesian fibration,
			\item $p$ is a map of coCartesian fibrations over $D$, i.e.\ $p_1$ sends $p_2$-coCartesian morphisms to equivalences in $C$,
			\item The functor $D \to \cat_{/C}$ classified by $p_2$ factors through $\mathrm{coCart}_{/C}$, i.e.\ for all $d \in D$ the fiber $p_2^{-1}(d) \to C$ is a Cartesian fibration, and for all $f:d \to d'$ the coCartesian transport $f_!: p_2^{-1}(d) \to p_2^{-1}(d')$ preserves Cartesian morphisms over $C$.
		\end{itemize}
		This definition is symmetric in $C$ and $D$, i.e.\ $p_1$ is also a Cartesian fibration satisfying analogous properties with respect to $p_2$. Co- and contravariant unstraightening in $C$ and $D$ respectively (the order does not matter) by \cite[Prop.\ 2.1.16]{stefanichcorr} induces a subcategory inclusion
		\[ \smallint\nolimits_{D}^{C} : \Fun(C^{\op} \times D, \cat) \overset{}{\hookrightarrow} \cat_{/C \times D} \]
		with image spanned by the two-sided fibrations and \emph{maps of two-sided fibrations}: Morphisms of coCartesian fibrations over $D$ that restrict to morphisms of Cartesian fibrations over $C$ on each fiber.
	\end{reminder}
	If all fibers of $p$ are spaces, this agrees with Lurie's notion of \emph{bifibrations} in \HTT{Def.}{2.4.7.2}.
		\begin{ex}
		\label{ex:arrowtwosided}
		The source and target projections assemble into a two-sided fibration $(\mathrm{src},\mathrm{tgt}) : \Arr(C) \to C \times C$ classifying the mapping space functor $\Map : C^{\op} \times C \to \Spaces \subseteq \cat$, compare \HTT{Cor.}{2.4.7.11}, \cite[Prop.\ 2.2.4]{stefanichcorr}.
	\end{ex}

	\begin{obs}
		\label{obs:twosidedfib}
If $E \to C \times D$ is a two-sided fibration, then so is the pullback $E' \to C' \times D'$ along functors $C' \to C, D' \to D$. A map of two-sided fibrations $E \to \tilde{E}$ over $C \times D$ is fully faithful (resp. an equivalence) on total spaces if and only if it is so on the fiber over each $(c,d) \in C \times D$. Both statements follow from the respective statements about (co)Cartesian fibrations, e.g.\ \kerodon{01VB}.	\end{obs}

	\begin{lemma}
		\label{lem:adjunctionunstr}
		If $L: \ccC \rightleftarrows \ccD : R$ is an adjunction, then the natural isomorphisms
		\[
		\Map_\ccD( L (-), -) \simeq \Map_\ccC (-, R(-))
		\]
		assemble into an equivalence of categories over $C\times D$
		\[ \ccC \times_{\Fun(\{0\}, \ccD)}\Arr(\ccD) \simeq  \Arr(\ccC) \times_{\Fun(\{1\}, \ccC)} \ccD \, .\]
	\end{lemma}

	\begin{proof}
		By \cref{ex:arrowtwosided} and \cref{obs:twosidedfib}, the functors $\ccC \times_{D}\Arr(\ccD) \to C \times D$ and $\Arr(\ccC) \times_{C} \ccD \to C \times D$ are two-sided fibrations classifying $\Map_\ccD (L-, -)$ and $\Map_\ccC (-, R -)$ respectively. 
		The natural isomorphism between them (induced e.g.\ by composing with the counit natural transformation $LR \To \id$) unstraightens to an equivalence between these two-sided fibrations.
	\end{proof}
	\begin{cor}
		\label{cor:sliceadjoints}
		Given an adjunction $L: \ccC \rightleftarrows \ccD : R$ and $c \in \ccC$, the equivalence from \cref{lem:adjunctionunstr} induces an equivalence
		\[ \ccD_{L(c)/} \simeq \ccC_{c/} \times_\ccC \ccD =: \ccD_{c/} \, . \]
	\end{cor}
	\begin{proof}
		Take the fiber over $c \in \ccC$ in the equivalence in \cref{lem:adjunctionunstr}.
	\end{proof}

\begin{obs} \label{obs:counit}
It follows from \cref{lem:adjunctionunstr} that any adjunction $L : C\rightleftarrows D:R $ induces a functor
\vspace{-2mm}
\[\begin{tikzcd}
D \simeq \Arr(D) \times_{D} D \arrow[rr] \arrow[dr] && \arrow[dl]  \Arr(C) \times_C D  \overset{\text{Lem.}~\ref{lem:adjunctionunstr}}{\simeq}  C \times_{D} \Arr(D) \\
&C
\end{tikzcd}
\]
which may  informally  be thought of as sending $d\in D$ to the counit $(LR d \to d) \in \Arr(D)$.
\end{obs}

	\begin{defin}[{\HA{\S}{7.3.2}}]
		\label{reminder:relativeadjoints}
		Consider a commutative diagram of categories:
		\[
		\begin{tikzcd}
			\ccC \arrow[rr, "F"] \arrow[rd, "p"'] & & \ccD \arrow[ld, "q"] \\
			& B &
		\end{tikzcd}
		\]
		A functor $F^\rR: D \to C$ is a \emph{relative right adjoint} to $F$ if it is right adjoint to $F$, and the counit $ F \circ F^\rR \To \id_{D}$ is sent to an isomorphism by $q$.
		
		A functor $F^\rL: D\to C$ is a \emph{relative left adjoint} to $F$ if it is left adjoint to $F$ and the unit $\id_{D} \To F \circ F^\rL$ is sent to an isomorphism by $q$.

	\end{defin}
	Both conditions imply that also the unit, resp. counit, is sent to an isomorphism by $p$ and hence that the canonical map $p \circ F^\rR \To q$ and $q \To p \circ F^\rL$ are isomorphisms, respectively.
	
	\begin{obs}\label{obs:fiberwiseAdj}
	In the setting of \cref{reminder:relativeadjoints} a relative right/left adjoint of $F$ induces right/left adjoints of the induced functors $F_b: C_b \to D_b$ between fibers for every $b \in B$. Conversely, by \HA{Prop.}{7.3.2.6} if $p$ and $q$ are locally coCartesian fibrations and $F$ preserves locally coCartesian morphisms, and if all $F_b$ admit right/left adjoints  then these fiberwise adjoints assemble into a relative right/left adjoint of $F$.
	\end{obs}

	\begin{ex}\label{exm:ffimpliesrel}
	Given an adjunction $L: C \rightleftarrows D: R$ with fully faithful left adjoint and hence invertible unit,  the induced commutative diagram		\[
		\begin{tikzcd}
			C \arrow[rr, "L"] \arrow[rd, equal] & & D\arrow[ld, "R"] \\
			& C &
		\end{tikzcd}
		\]
		exhibits $R$ as a relative right adjoint to $L$ over $C$. Indeed, for the counit $\epsilon: LR \To \id_{D}$, the transformation $R\epsilon: RLR \To R$ has invertible right inverse $\eta R: R\to RLR $ by invertibility of the unit $\eta: \id_C \to RL$. 
			\end{ex}

			\begin{obs}\label{obs:relativeAdjObs}
						Given a relative adjunction as in \cref{reminder:relativeadjoints}, the induced equivalence $C\times_D \Arr(D) \simeq \Arr(C) \times_C D$ from \cref{lem:adjunctionunstr} is compatible with the projections \[C\times_D \Arr(D) \to \Arr(D) \overset{q}{\to} \Arr(B) \overset{p}{\leftarrow} \Arr(C) \leftarrow \Arr(C) \times_C D.\]
			In particular, pulling back along the fully faithful diagonal $B \hookrightarrow \Arr(B)$ (with image the invertible arrows) enhances the equivalence from \cref{lem:adjunctionunstr}  to  an equivalence 
			\[ C\times_D \Arr(D) \times_{\Arr(B)} B \simeq B \times_{\Arr(B)} \Arr(C) \times_C D.
			\]
			\end{obs}

	\begin{lemma}
		\label{lem:Cartadjunction}
		Given an adjunction $L: C \rightleftarrows D: R$ such that $R$ is a Cartesian fibration, then the induced map
		\[
		D \to C \times_{D} \Arr(D)
		\]
		from \cref{obs:counit} is a relative left adjoint over $C$.
	\end{lemma}
	\begin{proof}
	Note that $R$ is a Cartesian fibration if and only if the functor $\Arr(D) \to \Arr(C) \times_C D$ admits a right adjoint. (Explicitly, the right adjoint sends $(\alpha: c \to Rd)$ to a Cartesian lift $ \alpha^*d \to d$  of $\alpha$.) By construction, the map $D\to C\times_D \Arr(D)$ is the composite
	\[		D\to \Arr(D) \to \Arr(C) \times_C D \simeq C \times_D \Arr(D)
		\]
		where the last map is the equivalence over $C \times D$ from \cref{lem:adjunctionunstr}. Since  the first map admits a right adjoint sending an arrow to its source and the second map admits a right adjoint since $R$ is a Cartesian fibration, so does the composite. 
		Explicitly, the composite $\Arr(C) \times_C D \to D \to \Arr(C) \times_C D$ sends $(\alpha:c \to Rd)$ to $(\id: R\alpha^*d \to R\alpha^*d)$, and the counit is given by the commutative square
		\[\begin{tikzcd}[column sep=2cm]
		R\alpha^*d  \arrow[d, equal] \arrow[r, "\simeq"] & c \arrow[d, "\alpha"] \\
		R\alpha^*d\arrow[r, "R(\alpha^*d \to d)"]  & Rd
		\end{tikzcd} \]
	In particular, the projection to $C$ sends the counit to the equivalence $R\alpha^*d \simeq c$. 	\end{proof}

	\begin{lemma}[{c.f.~\cite[Prop.\ 5.7.4]{haugseng}}] 
		\label{lem:coCartrefl}
		Let $p: E \to B$ be a functor and $\iota: \widetilde{E} \subseteq E$ a full subcategory.
		\begin{enumerate}[(1)]
				\item If $p$ is a (co)Cartesian fibration and $\widetilde{E}$ is closed under (co)Cartesian transport, then $p \iota : \widetilde{E} \to B$ is a (co)Cartesian fibration. In this case, an arrow in $\widetilde{E}$ is $p \iota$-(co)Cartesian if and only if it is $p$-(co)Cartesian and $p \iota$-(co)Cartesian transport agrees with $p$-(co)Cartesian transport.
				\item If $p$ is a Cartesian fibration and $\iota$ admits a relative right adjoint $\iota^\rR$ over $B$, then $p \iota: \widetilde{E} \to B$  is also a Cartesian fibration. Given an arrow $f:b \to b'$ in $\cB$, the Cartesian transport of $e' \in \widetilde{E}_{b'}$ is given by the composite $ \iota \iota^\rR f^! e' \to f^! e' \to e'$, where $f^!e'$ denotes the $p$-Cartesian transport of $f$. In particular $\iota^\rR$ preserves Cartesian morphisms.
				\item If $p$ is a coCartesian fibration and $\iota$ admits a relative left adjoint $\iota^\rL$ over $B$, then $p \iota: \widetilde{E} \to B$  is also a coCartesian fibration. 
				Given an arrow $f:b \to b'$ in $\cB$, the coCartesian transport of $e \in \widetilde{E}_{b}$ is given by the composite $ e \to f_!e \to \iota \iota^\rL f_! e$, where $f_!e$ denotes the $p$-coCartesian transport of $f$. In particular $\iota^\rL$ preserves coCartesian morphisms.
		\end{enumerate}
		\end{lemma}
	\begin{proof}
		The first statement is immediate.
		For the second statement we need to show that for any $e'_0 \in \widetilde{E}$ over $b_0 \in B$, the following square is a pullback:
		\[
		\begin{tikzcd}
			\Map_{\widetilde{E}}(e'_0, \iota^\rR f^! e') \arrow[r] \arrow[d] & \Map_{\widetilde{E}} (e'_0, e') \arrow[d] \\
			\Map_{B}(b_0, b) \arrow[r] & \Map_{B}(b_0, b')
		\end{tikzcd}
		\]
		But the upper left mapping space agrees agrees with $\Map_E(\iota e'_0, f^! e')$ while the upper right agrees with $\Map_E(e'_0, e')$, so this follows from $p$-Cartesianity of $f$.
		The third statement is dual to the second.
	\end{proof}

		\begin{prop} \label{prop:ArrLCartesian} Let $C$ be a category with a factorization system $(L, R)$ and let $\Arr^{\rL}(C)$  be the full subcategory of $\Arr(C):= \Fun([1], C)$ on the morphisms in $L$. Then, the source projection $\Arr^\rL(C) \to C$ is a Cartesian fibration. If $C$ admits pushouts, it is also a coCartesian fibration.  A morphism $l_0 \to l_1$ in $\Arr^{\rL}(C)$, represented by a square 
	\[\begin{tikzcd}
	c \arrow[r, "\phi"]  \arrow[d, "l_0 \in L"']& d\arrow[d, "l_1 \in L"]\\
	c' \arrow[r, "\psi"] & d' 
	\end{tikzcd}
	\]
	\begin{itemize} 
	\item is Cartesian iff $\psi \in R$; 
	\item is coCartesian iff this square is a pushout square in $C$. 
	\end{itemize}
		\end{prop}
	\begin{proof}
	Recall that for any category $\cC$, the source projection $\Arr(\cC) \to \cC$ is a Cartesian fibration with Cartesian morphisms given by those squares which induce an equivalence on the target component and is a coCartesian fibration if $\cC$ admits pushouts with coCartesian morphisms given by  pushout squares \HTT{Lem.}{6.1.1.1}.
	By \HTT{Lem.}{5.2.8.19}, the full inclusion $\Arr^\rL(C) \to \Arr(C)$ has a right adjoint sending an $(a\to b)$ to the left half of its factorization.  Hence, the statement follows directly from \cref{lem:coCartrefl}.	\end{proof}


	\section{Background on operads}
	\label{sec:operads}

	\newcommand{\CMon}{\mathrm{CMon}}
	In this section, we briefly recall several statements from the theory of symmetric monoidal categories and operads.

	For a (symmetric) operad $O$, we write $O^{\otimes} \to \Fin$ for its category of operations and write $\underline{O}:= O_{\langle 1 \rangle}$ for its underlying category, defined as the fiber over $\langle 1 \rangle \in \Fin$. We will sometimes refer to the objects of $\underline{O}$ as the \emph{colors} of $O$ and the morphisms covering the terminal map $\langle n \rangle \to \langle 1 \rangle$ in $\Fin$ as the \emph{$n$-ary multimorphisms} or \emph{$n$-ary operations}.

		For an operad $O$, an \emph{$O$-monoidal category} is an $O$-algebra in $\cat$ with its Cartesian monoidal structure. By \HA{Prop.}{2.1.2.12}, this is equivalent to the unstraightened datum of a coCartesian fibration of operads $M^{\otimes} \to O^{\otimes}$, and by \HA{Prop.}{2.4.2.5} further equivalent to an \emph{$O$-monoid} object in $\cat$, i.e.\ a functor $O^{\otimes} \to \cat$ satisfying appropriate Segal conditions. An $O$-monoidal category is \emph{compatible with colimits} if all $o \in \underline{O}$, the associated categories have colimits and if for all $n$-ary multimorphisms in $O$, the associated functors preserve colimits in each variable separately.

		An important special case is the notion of a symmetric monoidal category: If $C$ is a category with products, a \emph{commutative monoid} in $C$ is a functor $M:\Fin \to C$ satisfying the Segal conditions:  $M(\langle 0 \rangle)$ is terminal and the projections induce equivalences $M(\langle n \rangle) \to M(\langle 1 \rangle)^{\times n}$. Denote the full subcategory of commutative monoids by $\CMon(C) \subseteq \Fun(\Fin, C)$. A symmetric monoidal category is an object in $\CMon(\cat) \simeq \CAlg(\cat)$. 
		
		 If $C$ has colimits, the full inclusion $\CMon(C) \to \Fun(\Fin,C)$ admits a left adjoint and the Day convolution symmetric monoidal structure on $\Fun(\Fin C)$ induced by the smash product on $\Fin$ is compatible with this localization, i.e.\ induces a unique symmetric monoidal structure on $\CMon(C)$, henceforth referred to as the \emph{tensor product of commutative monoids} for which the localization $\Fun(\Fin, C) \to \CMon(C)$ is symmetric monoidal. By construction, it follows that the left adjoint $C\to \CMon(C)$ of the forgetful functor is symmetric monoidal for the Cartesian monoidal structure on $C$. 
		
		By \HA{Prop.}{2.2.4.9}, the subcategory inclusion $\CMon(\cat) \to \Op$  admits a left adjoint $\Env: \Op \to \CMon(\cat)$, the \emph{enveloping symmetric monoidal category}. 		
		\begin{reminder}\label{reminder:absolutetensor}
		Using \cite{haugseng2024operads}, it is shown in \cite[Lem.~2.1.3]{barkan2022equifibered} that this left adjoint $\Env: \Op \to \CMon(\cat)$ is a monomorphism in $\widehat{\cat}$, i.e.\ it exhibits $\Op$ as a subcategory of $\CMon(\cat)$ (explicitly described in \cite{barkan2022equifibered}). By \cite{barkan2023segalification}, the tensor product of commutative monoids on $\CMon(\cat)$ restricts to a unique symmetric monoidal structure on $\Op$ for which $\Env$ is symmetric monoidal. We refer to this symmetric monoidal structure on $\Op$ as the \emph{Boardman-Vogt tensor product}. As shown in \cite[Thm. E]{barkan2023segalification}, the tensor product functor $-\otimes -: \Op \times \Op \to \Op$ agrees with the one constructed model dependently in \HA{\S}{2.2.5}. It follows from \HAss{Const.}{3.2.4.1}{Prop.}{3.2.4.3} that the Boardman-Vogt tensor product is a closed symmetric monoidal structure with inner hom between operads $O$ and $P$ given by an operad $\Alg_O(P)$ whose underlying category is the category of $O$-algebras in $P$. Explicitly, this means that $\Alg_O(P)$ is equipped with a map of operads $\operatorname{ev}: O \otimes \Alg_O(P) \to P$ inducing an equivalence 
				\[ \Map_{\Op} (K, \Alg_{O}(P)) \simeq \Map_{\Op} (K \otimes O, P)\ \]
		for any operad $K \in \Op$. 	
		\end{reminder}

	\begin{lemma}\label{lem:algforgetsm}
	Let $V\in \CMon(\cat)$. Then, the functor $\Alg_{-}(V) : \Op^{\op} \to \Op$ factors through the subcategory $\CMon(\cat)$. In other words, for any operad $O$ the operad $\Alg_O(V)$ is in fact a symmetric monoidal category and for all operad maps $O\to P$, the induced operad map $\Alg_P(V) \to \Alg_O(V)$ is a symmetric monoidal functor. 
	\end{lemma}
	\begin{proof}
	Using the adjunction $ \Op \rightleftarrows \CMon(\cat)$ and the fact that the left adjoint $\Env: \Op \to\CMon(\cat)$ is symmetric monoidal, it follows that $\Alg_O(V)$ agrees with the underlying operad of the inner hom in $\CMon(\cat)$ between $\Env(O)$ and $V$.
	\end{proof}
	It follows from \cref{lem:algforgetsm} that if $V$ is a symmetric monoidal category, $O$ is an operad and $o\in \underline{O}$ a color, then the restriction functor $\ev_o: \Alg_O(V) \to V$ is symmetric monoidal. Hence, we sometimes refer to the symmetric monoidal structure on $\Alg_O(V)$ as the \emph{pointwise tensor product}.

	\begin{ex}	\label{ex:pointwisetensor}
	By \HA{Rmk.}{2.1.4.10}, the functor $\underline{(-)}: \Op \to \cat$ sending an operad to its underlying category admits a fully faithful left adjoint $\mathrm{Free}_{\Op}: \cat \to\Op$ sending a category $C$ to the operad whose underlying category is $C$ and which has no $n$-ary operations for $n \neq 1$. Hence, the composite $\cat \overset{\mathrm{Free}_{\Op}}{\longrightarrow} \Op \to \CMon(\cat)$ is left adjoint to the forgetful functor,  i.e.\ agrees with the free commutative monoid functor and hence is symmetric monoidal. Since $\Op \to \CMon(\cat)$ is a subcategory inclusion, it induces a symmetric monoidal structure on $\mathrm{Free}_{\Op}: \cat \to \Op$. 	
	
	This implies that for an operad $O$ and category $K$, the category $\Fun(K, \underline{O})$ is the underlying category of an operad (namely of $\Alg_{\mathrm{Free}_{\Op}(K)}(O)$). In particular,  for a symmetric monoidal category $V$, $\Fun(K, V)$ inherits a \emph{pointwise} symmetric monoidal structure for which $\Fun(K, V) \to V$ is symmetric monoidal. Explicitly, if $V$ is represented by a commutative monoid object $V(-): \Fin \to \cat$, then the symmetric monoidal structure on $\Fun(K,V)$ is given by the commutative monoid object $\Fun(K, V(-)): \Fin \to \cat$. 
	 	\end{ex}

	\begin{obs}\label{lem:freesm}
	For $V$ a symmetric monoidal category, the (by \cref{lem:algforgetsm}) symmetric monoidal functor $\RMod(V)  \to\Alg_{\operatorname{Triv} \sqcup \Ass}(V) \simeq  V \times \Alg(V)$ obtained by restricting along the map of operads $\operatorname{Triv} \sqcup \Ass \to \RMod(V)$ admits a symmetric monoidal left adjoint $\operatorname{Free}: V\times \Alg(V) \to \RMod(V)$.
	
	Indeed, by  \HA{Cor.}{4.2.4.4}, it admits a left adjoint which sends a pair $(C,A)$ to $C \otimes A$ regarded as a right $A$-module; it follows that the oplax monoidal structure induced by the symmetric monoidal structure on $\RMod(V) \to V \times \Alg(V)$ from \cref{lem:algforgetsm} is indeed strongly monoidal.	\end{obs}


		\begin{notat}
	A \emph{lax symmetric monoidal functor} $C\to D$ between symmetric monoidal categories is an operad map $C\to D$. We set $\Fun^{\otimes\mathrm{lax}}(C, D):= \Alg_{C}(D)$.	
		\end{notat}
		
	\begin{lemma}[{\cite[Lem. 3.1.1]{calmes2024motivic}}]\label{lem:coCartoperadscriterion}
	For a symmetric monoidal functor $F:C \to D$ the induced functor $C^{\otimes} \to D^{\otimes}$ is a coCartesian fibration of operads (in the sense of \HA{Def.}{2.1.2.13}), if and only if the underlying functor $C\to D$ is a coCartesian fibration such that for every object $c\in C$, the functor $c\otimes- : C\to C$ preserves $F$-coCartesian morphisms. In this case, a morphism in $C^\otimes$ is coCartesian iff it can be written as the composition of a morphism that is coCartesian over $\Fin$, and a coCartesian morphism in the fiberwise fibration $C^\otimes_{\langle n \rangle} \simeq C^{\times n} \to D^\otimes_{\langle n \rangle} \simeq D^{\times n}$ over fixed $\langle n \rangle \in \Fin$.
	
	Given a morphism $C \to C'$ in $\CMon_{/D}$, the induced map $C^{\otimes} \to C'{}^{\otimes}$ preserves coCartesian morphisms if and only if the underlying functor $ C\to C'$ does.
	\end{lemma}

	\begin{cor}\label{cor:coCartoperadscrit} 
	If $D$ is a symmetric monoidal category, then unstraightening induces a monomorphism  $\Fun^{\otimes\mathrm{lax}}(D, \cat) \hookrightarrow \CMon_{/D}$ with image the subcategory with:
	\begin{itemize}
	\item Objects given by  those symmetric monoidal functors $C\to D$ for which $C^{\otimes} \to D^{\otimes}$ is a coCartesian fibration, or equivalently for which the underlying functor $C\to D$ is a coCartesian fibration and for which $c\otimes-$ preserves coCartesian morphisms for every $c \in C$
	\item Morphisms given by  those $C\to C'$ in $\CMon_{/D}$ whose underlying functor $C \to C'$ preserves coCartesian morphisms. 
	\end{itemize}
	\end{cor}
	\begin{proof}
		Combine \cref{lem:coCartoperadscriterion} with \HA{Prop.}{2.4.1.7} and \HA{Rem.}{2.4.2.6}.
	\end{proof}

	\bibliographystyle{alpha}
	\bibliography{cauchybib}

\end{document}